\documentclass[12pt,reqno]{amsart}
\usepackage{amsmath,amsfonts,amssymb,calligra,amscd,amsthm,amsbsy,epsf}
\usepackage[dvips]{graphicx}
\usepackage{comment}
\usepackage[dvips]{color}

\numberwithin{equation}{section}
\newtheorem{theorem}{Theorem}
\newtheorem{meta-thm}[theorem]{Meta-Theorem}
\newtheorem{lemma}[theorem]{Lemma}
\newtheorem{cor}[theorem]{Corollary}
\newtheorem{proposition}[theorem]{Proposition}

\newtheorem{remark}[theorem]{Remark}
\newtheorem{definition}[theorem]{Definition}

\newcommand{\noaverage}[1]{({#1})^0}

\DeclareMathAlphabet{\mathcalligra}{T1}{calligra}{m}{n}

\newcommand\beq[1]{ \begin{equation}\label{#1} }
\newcommand{\eeq}{ \end{equation} }

\newcommand\beqa[1]{ \begin{eqnarray} \label{#1}}

\newcommand{\eeqa}{ \end{eqnarray} }

\newcommand{\beqano}{ \begin{eqnarray*} }

\newcommand{\eeqano}{ \end{eqnarray*} }

\newcommand\equ[1]{{\rm (\ref{#1})}}

\def\dist{\operatorname{dist}}

\def\Id{\operatorname{Id}}
\def\Im{\operatorname{Im}}

\def\Lip{\operatorname{Lip}}
\def\Range{\operatorname{Range}}
\def\A{{\mathcal A}}
\def\B{{\mathcal B}}

\def\D{{\mathcal D}}
\def\G{{\mathcal G}}
\def\E{{\mathcal E}}
\def\tE{{\widetilde  E}}

\def\I{{\mathcal I}}
\def\LL{ {\mathcal{L}}}
\def\M{{\mathcal M}}
\def\N{{\mathcal N}}
\def\R{{\mathcal R}}
\def\S{{\mathcal S}}
\def\T{{\mathcal T}}
\def\U{{\mathcal U}}
\def\complex{{\mathbb C}}
\def\integer{{\mathbb Z}}
\def\nat{{\mathbb N}}

\def\real{{\mathbb R}}
\def\torus{{\mathbb T}}

\def\eps{\varepsilon}
\def\th{\theta}

\def\tJ{ {\widetilde J}}
\def\st{s}
\def\un{u}
\def\tgamma{ {\tilde \gamma}}
\def\ttgamma{{ {\tilde {\tilde \gamma}}}}
\def\txi{{ {\tilde \xi}}}

\def\tGamma{ {\widetilde \Gamma}}

\begin{document}
\title[Whiskered KAM tori of conformally symplectic systems]
{Existence of whiskered KAM tori of conformally symplectic systems}

\author[R. Calleja]{Renato C.  Calleja}
\address{Department of Mathematics and Mechanics, IIMAS, National
  Autonomous University of Mexico (UNAM), Apdo. Postal 20-726,
  C.P. 01000, Mexico D.F., Mexico}
\email{calleja@mym.iimas.unam.mx}

\author[A. Celletti]{Alessandra Celletti}
\address{
Department of Mathematics, University of Rome Tor Vergata, Via
della Ricerca Scientifica 1, 00133 Rome (Italy)}
\email{celletti@mat.uniroma2.it}

\author[R. de la Llave]{Rafael de la Llave}
\address{
School of Mathematics,
Georgia Institute of Technology,
686 Cherry St.. Atlanta GA. 30332-1160 }
\email{rll6@math.gatech.edu}

\thanks{R.C. was partially supported by UNAM-PAPIIT IA 102818. A.C. was partially supported by GNFM-INdAM and acknowledges MIUR Excellence Department
Project awarded to the Department of Mathematics of the University of
Rome ``Tor Vergata" (CUP E83C18000100006).
R.L. was partially supported by NSF grant DMS-1800241.}
\thanks{ This material is based upon work supported by the National Science Foundation under Grant No. DMS-1440140 while the  authors
were in residence at the Mathematical Sciences Research Institute in Berkeley, California, during the  Fall 2018 semester}

\baselineskip=18pt              


\begin{abstract}

Many physical problems are described by conformally symplectic
systems (i.e., systems whose evolution in time
transforms a  symplectic form into a multiple of itself).
We  study the existence of whiskered tori in
a family $f_\mu$ of conformally symplectic maps depending on parameters $\mu$
(often called \emph{drifts}). We recall that whiskered tori
are tori on which the motion is a rotation, but they
have as many expanding/contracting directions as allowed by the
preservation of the geometric structure.

Our main result is formulated in an  \emph{a-posteriori} format.
We fix $\omega$
satisfying Diophantine conditions. We assume
that  we are  given 1) a value of the
parameter $\mu_0$, 2) an embedding of the torus $K_0$ into the
phase space, approximately invariant under $f_{\mu_0}$ in the sense
that $f_{\mu_0} \circ K_0 - K_0 \circ T_\omega$ (where $T_\omega$
is the shift by $\omega$) is small (in some norm), 3) a splitting
of the tangent space at the range of $K_0$, into three bundles
which are approximately invariant under $D f_{\mu_0}$ and such that
the derivative satisfies \emph{``rate conditions''} on each of the
components.

Then, if some non-degeneracy conditions (verifiable by a finite
calculation  on the approximate  solution and which  do
not require any global property of the map) are satisfied, we
show that there is another
parameter $\mu_\infty$, an embedding $K_\infty$
and splittings close to the original ones which are invariant
under $f_{\mu_\infty}$. We also bound $|\mu_\infty - \mu_0|$,
$\|K_\infty - K_0 \|$ and the distance of
the initial and final splittings in terms of
the initial error.

We allow
that the stable/unstable bundles are nontrivial (i.e.,
not homeomorphic to a product bundle). On the other hand,
we show that the geometric set up has the
global consequence
that the center bundle is necessarily trivial (i.e.,
homeomorphic to a product bundle).

The proof of the main theorem
consists in describing an iterative process that takes advantage
of cancellations coming from the geometry. Then, we show
that the process converges to a true solution when started from
an approximate enough solution. The iterative process
leads to an efficient algorithm that is quite practical to implement.

The a-posteriori format of the theorem implies the usual
formulation of persistence under perturbations (the solutions
for the original systems are approximate solutions for
the perturbation), but it
also allows to justify approximate
solutions produced by any method (for example numerical solutions or
asymptotic formal expansions).
As an application, we study the singular problem of effects of
small dissipation on whiskered tori. We develop formal (presumably
not convergent) expansions in the perturbative parameter (which
generates dissipation) and use
them as input for the a-posteriori theorem. This allows
to obtain lower bounds for the domain of analyticity of
the tori as function of the perturbative parameter.

Even if we state only the theory for maps, our results apply also to flows.

\end{abstract}


\maketitle


\section{Introduction}\label{sec:intro}
Several physical problems are modeled by Hamiltonian systems affected by a
dissipation which enjoys a remarkably geometric property,
namely that the symplectic structure (preserved  by the Hamiltonian
evolution) is transformed into a multiple of itself.  Such systems
are called \emph{conformally symplectic systems}.

Examples of physical problems that are described by
conformally symplectic systems are:
\begin{itemize}
\item
Hamiltonian systems with a dissipative effect proportional to the velocity;
a concrete example is given by the spin--orbit problem in Celestial Mechanics
with a tidal torque -- see \cite{Celletti2010, CorreiaL2004, CorreiaL2009};
\item
Euler-Lagrange equations of exponentially discounted
systems; these models are often  found in finance, when inflation is
present and one needs to minimize the cost in present money
 --  see \cite{Bensoussan88, LomeliR, IturriagaS11, DFIZ2014,DaviniFIZ16}.
The exponential discount is also common in control theory
\emph{finite horizon models} - see \cite{MeadowsHER95};
\item
Gaussian thermostats, which are used in computations of non-equilibrium molecular dynamics -
see \cite{DettmannM96, WojtkowskiL98};
\item Nos\'e-Hoover dynamics in  Statistical Mechanics - see
\cite{DettmannM96b,Hoover}\footnote{There are several formulations
of the Nos\'e-Hoover dynamics in the literature. Some of them are
not conformally symplectic.}
\end{itemize}

Besides the interest in applications, (locally)  conformally symplectic systems
were also studied as natural problems by differential geometers (see \cite{Banyaga02, Agrachev,Vaisman85}).  For a detailed comparison between the
conformally symplectic systems and the slightly more general
locally conformally symplectic systems see \cite{CallejaCL18a}.

\begin{remark}
\label{extraproperties}
Conformally symplectic systems are
a very special kind of dissipative systems (for example, mechanical systems
with a friction proportional to a power different from $1$
 of the velocity are
not conformally symplectic). Therefore, conformally
symplectic systems enjoy remarkable properties which are not present in
more general dissipative
systems.

For example, in \cite{CallejaCL13a} it is shown that for
conformally symplectic systems in a neighborhood of
a  Lagrangian invariant torus,
the motion is smoothly conjugate to
a rotation along the torus and
a constant contraction in the normal direction. This is
clearly false in general dissipative systems.
In \cite{CCFL14} it is shown how this rigid
behaviour of conformally symplectic systems
leads to quantitative properties of phase locking regions (\cite{CCFL14}).  The
existence of a time-dependent variational principle also leads to
very dramatic qualitative properties not shared by arbitrary
dissipative systems (\cite{MaroS16}). The conformal symplectic
properties of a system also affect the dimension of the
parameters one needs to adjust to get a quasiperiodic
 solution (\cite{CallejaCL13a}).

For us, the most remarkable property of conformally symplectic
systems is the so-called \emph{``automatic reducibility"} that
shows that the ``center'' bundle around a rotational torus allows a
canonically defined system of coordinates in which the derivative
is particularly simple. We will also show in
Section~\ref{sec:geometry} that the conformal symplectic
geometry  implies that the ``center''
bundles around a rotational torus are trivial.
\end{remark}

The goal of this paper is to study the existence of whiskered tori
in conformally symplectic systems.
Whiskered tori were introduced in \cite{Arnold64,Arnold63a}
in Hamiltonian (a.k.a. symplectic) systems
and conjectured to be the key geometric structures leading to
instability for nearly integrable systems.

 We will present more precise
definitions of whiskered tori later (see
Definition~\ref{def:whiskered}), but we indicate that these are
tori in which the motion is conjugate to a rotation and which have
many hyperbolic directions (exponentially contracting in the
future or in the past under the linearized evolution). Definition~\ref{def:whiskered} includes
that the exponential rates characterizing the
center bundle straddle the conformally
symplectic constant and that the
dimension of the hyperbolic directions is as large as possible
given that the map is conformally symplectic.

The existence of quasi-periodic motions in dissipative
systems is very different from the Hamiltonian case.
The dissipation forces  many orbits to have the
same asymptotic behavior and, hence, the set of
asymptotic behaviours is smaller in the dissipative case than in
the Hamiltonian one.
Therefore, adding a small dissipation to a Hamiltonian
system is a very singular perturbation. In the Hamiltonian
case, under mild non-degeneracy conditions,
 to find tori of a certain frequency, it suffices
to choose the initial conditions. In the dissipative case,
since many solutions have the same asymptotic behavior one cannot
choose the initial conditions to obtain the desired long
term behavior in a fixed map. One needs to consider
families and adjust parameters to obtain a torus with a fixed frequency.
The results here apply also to symplectic systems
(the case considered in \cite{Arnold64,Arnold63a}), even if
the parameter count is different. Indeed, we find a
unified formalism in which one can continue from the symplectic
to the weakly symplectic.

Our first main result is Theorem~\ref{whiskered}, which is
expressed in the format of a-posteriori theorems of
numerical analysis. Given a family $f_\mu$
of conformally symplectic mappings,
we formulate an equation, called \emph{invariance equation}
for the parameterization of
a torus, say $K$, for the parameter $\mu$
in the family and for the splittings of the space. The invariance
equation expresses
that the parameterization and the splittings are invariant by
the map given by the parameter.

Theorem~\ref{whiskered} allows us to validate approximate solutions of the invariance equation
obtained by non-rigorous methods (one just needs to verify the
accuracy of the solution and the condition numbers).
Notably, one can justify numerical methods or asymptotic expansions.
Note that Theorem~\ref{whiskered} does not assume
that the system is close to integrable and it is, therefore, very
suitable to study phenomena that happen near the breakdown of
the tori. A result on the local uniqueness of the solution is
given in Theorem~\ref{thm:uniqueness}.
The a-posteriori theorems can also be used to justify computations in
concrete systems of interest. For example in Celestial
Mechanics and Astrodynamics, one is interested in finding
objects in N-body problems with very concrete values of
the parameters.

It is also important to remark that the condition numbers we introduce
 just involve averages
of algebraic expressions formed by the approximate solution
and its derivatives,
hence they are straightforwardly computable just knowing the
approximate solution.  They do not involve any global
assumption on the map such as the global twist.

Our second main result is Theorem~\ref{thm:domain}.
We consider dissipative perturbations of a Hamiltonian
system and
study the  domain of analyticity
in terms of the perturbation parameter of
the resulting whiskered tori (and the drift and the bundles).

The limit of small dissipation occurs naturally in
many problems in Celestial Mechanics and many
of the models for dissipation in common use are conformally symplectic
(\cite{Celletti2010}, \cite{CorreiaL2004}, \cite{CorreiaL2009}).
Examples of dissipations going to zero might appear
when a celestial body subject to tidal torque reaches a synchronous
rotation state,  or when a satellite is launched from Earth and reaches an
altitude where the atmosphere is absent, thus moving to a region where the
atmospheric drag dissipation is zero (\cite{CG2018}, \cite{Chao}).
The limit of small dissipation is also of interest in finance,
where it corresponds to small inflation (\cite{Bensoussan88}).
In control theory it  is common to consider
the limit when the finite  horizon is taken to infinity (\cite{MeadowsHER95}).
In the discounted action model,
in the limit of small dissipation the minimizing sets of
the action  and the
solutions of the Hamilton-Jacobi equation have been
considered in \cite{DFIZ2014,DaviniFIZ16}.

If we denote the perturbative parameter by $\eps$ (which we
consider complex),
we can study the analyticity domains for $\mu_\eps$, $K_\eps$ (see Theorem~\ref{thm:domain}).
In contrast with the Hamiltonian case,
in the dissipative case, unless one adjusts parameters, one does not expect that there
are quasi-periodic solutions. Hence, we do not expect that
perturbative expansions converge and that there
is no open neighborhood of $\eps = 0$ contained in the
analyticity domain of $\mu_\eps, K_\eps$.  We conjecture that the domain
we obtain in Theorem~\ref{thm:domain} (which indeed does not contain
any ball centered in the origin) is essentially optimal, see Section~\ref{sec:dom}.

\subsubsection{Application to flows}
Finally, we note that all the results we will present for maps apply also
to flows in continuous time by taking sections and considering
the return map (see the construction in
\cite{CallejaCL13a}). Of course, producing
a direct proof for flows by adapting the steps
in the proof for maps presented here is straightforward.

\subsection{Comparison with other results in the literature}
\label{literature}

The literature on the existence of whiskered tori in symplectic
systems  is very extensive, see
\cite{Graff74, Zehnder76,JorbaV97,LiYi05a, LiYi05b}.

One should also mention that there are persistence theorems
for whiskered tori  that apply
to  general systems (not necessarily symplectic). Indeed, the pioneering work
\cite{Moser67} already considered general perturbations (see also
\cite{BroerHTB90,BroerHS96,CanadellHaro}). Of course, the conformally symplectic
systems studied in this paper are a particular case of general
perturbations and the results of
 \cite{Moser67,BroerHTB90,BroerHS96,CanadellHaro}
 apply a-fortiori. The
extra assumption of conformally symplectic allows us to obtain some
results that are false for more general perturbations.  Notably, the
results here for  weakly dissipative systems could be
false if the dissipation is not conformally symplectic (see
Remark~\ref{extraproperties}).  We also note that,
for conformally symplectic systems
there are  constraints on the possible
exponents on the normal direction. We discuss
these constraints in  Section~\ref{sec:geometry}.
We also note that the dynamical and geometric
assumptions have the surprising global consequence that the
center bundle is trivial in the sense of bundle theory
(diffeomorphic to a product bundle).

{From} the point of
view of techniques,
we note that the papers \cite{Moser67,BroerHTB90,BroerHS96} are based on
transformation theory
(i.e., making changes of variables till the system is in a form which
manifestly does have an invariant torus), they also assume
that the stable and unstable bundles are trivial in the sense
of bundle theory (the bundle is a product bundle). This triviality of the bundle allows us
to use a global system of coordinates in the center manifold in which
the linearized equation can be reduced to constant coefficients.

This paper is based on making additive corrections to an embedding
and by solving the linearized equations using geometrically natural operations.

{From} the point of view of
computation, transformation theory is hard to implement as algorithms,
since it involves operating on functions
with many variables (as many variables as the
dimension of the phase space). The algorithms in this paper
only require to deal with functions with as many variables
as the dimension of the objects considered.

On the other hand, the transformation theory
gives not only
the existence of the whiskered tori,
but also substantial information on the behavior of
orbits in a neighborhood of the torus \cite{Locatelli}. A comparison of
the transformation theory and methods similar to ours
for Lagrangian tori in conformally symplectic systems
appears in \cite{LocatelliS15}.

An approach for the study of whiskered tori
close to ours was described in
\cite{Fontich2009}. In particular, it produces
an a-posteriori method, which was implemented in
\cite{FontichLS} for finite dimensional
Hamiltonian systems. Algorithms based on
the method of \cite{FontichLS} and a more
efficient reformulation appear in
\cite{HuguetLS11}. Generalizations
to  Hamiltonian lattice systems appear in  \cite{FontichLSII} and in
\cite{LlaveS} for PDE's, while presymplectic systems are investigated in \cite{LlaveX18}.  The paper
\cite{CanadellHaro} uses similar techniques to give
a result on the persistence of whiskered tori for general systems, while
\cite{CanadellHaro2} presents implementations of
the algorithms and studies of the phenomena that happen at breakdown.

\subsection{The a-posteriori format}
The method in this paper
is not
based on transformation theory and on the reduction to normal forms, but it is rather
based on formulating an equation for the parameterization which expresses the
invariance and which is solved by a Newton-like method.

The solution of the Newton equations uses the contraction on
the stable and unstable bundles
as well as  geometric identities obtained from the
fact that the  map is conformally symplectic. These identities have
the geometric meaning that there is a special system of
coordinates in which the Newton equations are
particularly simple. They  were used in \cite{CallejaCL13a} for Lagrangian
tori of symplectic systems and in \cite{CallejaC10} for conformally
symplectic systems. In the
symplectic case they were introduced in \cite{Llave01c,LlaveGJV05} for
Lagrangian tori. In the case of whiskered tori for
symplectic systems, they were used
in \cite{Fontich2009,FontichLS} (a more efficient variation was introduced in
\cite{HuguetLS11,FontichLSII}).

The method in this paper allows that the hyperbolic bundles are non-trivial (a
situation that happens near resonances, \cite{HaroL07}).
Some other examples were presented in \cite{Fontich2009}.
Most methods based on normal forms assume that the bundles are
trivial.

In contrast,  in this paper, we will  show
that for conformally symplectic systems (including symplectic
systems)
the center bundles  of whiskered tori are trivial
(see Lemma~\ref{lem:istrivial}): this is a rather surprising interaction
between global geometry and dynamics.

The main result of this paper is stated in an
a-posteriori format: assuming the
existence of sufficiently approximate solutions with respect to
\emph{condition numbers}, then we conclude that there are true solutions.

Notice that the a-posteriori format implies the usual results of
persistence under a small change of the system. If there is a system
which has an invariant whiskered torus,
the torus
and its bundles will be approximately invariant for all the approximate
systems. Then, applying
the a-posteriori theorem we will conclude
that the perturbed system also has an invariant torus.
As an immediate consequence of
the a-posteriori format,  we automatically obtain Lipschitz dependence
on the map or on the frequency (just observe
that the solution corresponding to a frequency $\omega$
satisfies the equation for a frequency $\omega'$ up
to an error bounded by $C|\omega - \omega'|$).
Validating the Lindstedt series, we obtain sharper differentiability
properties.
As mentioned in \cite{CallejaCL13a}, the a-posteriori format has several
consequences: smooth dependence on parameters, Whitney
regularity in the frequency, bootstrap of regularity, etc.

The a-posteriori format is very well suited for numerical analysis,
since it allows one to validate the approximate
solutions produced by a numerical calculation. We also
note that the method of proof leads to a very
efficient algorithm. The iterative method  is quadratically convergent,
it only requires to discretize functions with a number of
variables equal to the dimension of
the torus searched, rather than  the dimension of the phase
space,
 the storage space requirements are small
(order $O(N)$, where $N$ is the number of discretization modes)
 and the
operation count per step is small ($O(N \ln(N))$).
A continuation method based on the  algorithm presented here
 is guaranteed to
converge till the boundary of existence of the torus
or till some of the non-degeneracy conditions fail.
Hence, a careful continuation algorithm gives a practical numerical algorithm to
detect the breakdown of analyticity (these
algorithms were implemented
for Lagrangian rotational tori and models
in statistical mechanics in \cite{CallejaL10}, \cite{CallejaC10}).
For conformally symplectic systems, the paper
\cite{LocatelliS15} includes
a comparison of the method presented here with the transformation method
from the point of view of applications.

\subsection{Main results of the paper}
The proof of our main result on the existence of the whiskered tori, Theorem~\ref{whiskered},
consists in describing
an iterative Newton method  to solve the invariance equation
when started on an approximate solution. The Newton step
takes advantage of  the geometry of the system
and of  the existence of hyperbolic directions.

A step of the Newton method involves solving a linearized
invariance equation. To solve this equation,
the linearized invariance equation is projected on the hyperbolic and
center subspaces. The equations projected on
the hyperbolic equations can be solved taking advantage of
the contractions (in the future or in the past).
As for the equations projected
to the center subspace, we use the so--called \emph{automatic reducibility}
which, near an approximate invariant torus, constructs a
system of coordinates in which the linearized equation along
the center directions takes a particularly simple form, which allows
the use of Fourier methods.
\smallskip

Our second result,
Theorem~\ref{thm:domain}, is concerned with the
limit of small dissipation.
The low dissipation limit is very natural in Celestial Mechanics and Astrodynamics since
many celestial bodies experience weak frictions, due to tidal forces or
the atmospheric drag (\cite{CCbasin}, \cite{Chao}).
In this case, the friction is indeed proportional
to the velocity, which makes the system conformally symplectic.

We consider systems that depend on a small
parameter\footnote{Since we are considering analyticity properties, it is natural
to take $\eps$ complex and all the objects considered
to be complex as well. This does not make
any difference in the proof, since the Newton step
involves just algebraic manipulations, derivatives
and solutions of cohomology equations, which are
the same for complex maps. Of course, if $f_{\mu,\eps}$ is such
that for real values of the arguments it gives real values, the
parameterizations $K_\eps$ and $\mu_\eps$ will also
be real for real values of the arguments.} $\eps$ and
the conformal factor is $\lambda(\eps)=1+\alpha
\eps^a+O(|\eps|^{a+1})$ for $a\in\integer_+$,
$\alpha\in\complex\setminus\{0\}$.  When $\eps=0$ we recover the
symplectic case and small $\eps$
means small dissipation.
We study the analyticity domains in $\eps$ near
$\eps=0$ of the whiskered tori (as well as the domains of the
drift and of the stable/unstable bundles). We note that the domains
we obtain do not contain any ball centered at $\eps = 0$, and we
conjecture that this is optimal, hence we conjecture that indeed
the formal power series are divergent.  For full dimensional
tori, this has been studied numerically by \cite{Bustamante}.

Inspired by \cite{CCLdomain}, we
start by constructing Lindstedt series for the
problem.
Even if the Lindstedt series, in this case,
probably do not converge, a finite order truncation
provides an
approximate solution, which can be used as the approximate solution
taken as input in  the a-posteriori
theorem, Theorem~\ref{whiskered}, for some (complex) values of
$\eps$ for which  we can verify  the quantitative conditions
of Theorem~\ref{whiskered}.

In this way, we obtain that the parameterization and the
drift are analytic when $\eps$ ranges in a domain which
we describe very explicitly.  The domain of $\eps$ where
we show that $K_\eps, \mu_\eps$ are analytic is
obtained by removing from a ball centered at the origin
a sequence of smaller balls with centers on a curve.
The radii of the balls decrease exponentially fast with the distance
of the centers to the origin.

It is interesting to remark that the Lindstedt expansions
produce approximate solutions for all  sufficiently small
values of $\eps$. What determines the shape of
the analyticity domain is the values of $\eps$ for which we
can  verify the non-degeneracy conditions (notably the
Diophantine conditions) of Theorem~\ref{whiskered}.  This emphasizes
the importance of the non-degeneracy conditions rather than just
the smallness of the error.

The study of Lindstedt series followed by a validation step for
low--dimensional tori in Hamiltonian systems was developed
in \cite{MR1720891}. The paper \cite{MR1720891}
considers also  the case of weak hyperbolicity: systems which
are perturbations of an integrable one, which does not have any
hyperbolicity, so that the hyperbolicity is generated by the
perturbation.  We also mention \cite{Masdemont}, where Lindstedt
series expansions for whiskered tori have been constructed in a problem
of Celestial Mechanics. More recently, \cite{Bustamante} has
computed Lindstedt series for Lagrangian tori  for
the dissipative case and determined the
domain of analyticity of the tori, as well as several other
properties of the series (monodromy, Gevrey properties).  The
numerical results of \cite{Bustamante} present several interesting
conjectures for the singular problem.

\subsubsection{Organization of the paper}
This paper is organized as follows. In Section~\ref{sec:def} we provide
basic notions, such as  conformally symplectic systems, Diophantine
frequency vectors, function spaces, cocycles,  invariant bundles, dichotomies.
In Section~\ref{sec:statement}, we state the first main
result, Theorem~\ref{whiskered}.
In Section~\ref{sec:geometry} we present some interactions
of the hyperbolicity assumptions and the geometry.
In particular,
in Section~\ref{sec:automatic} we present the automatic reducibility,
a key ingredient of the proof of Theorem~\ref{whiskered}, and in Section
~\ref{sec:triviality} we prove that the center bundle has to be trivial.
In Section~\ref{sec:whiskered} we provide the proof of Theorem~\ref{whiskered},
while the proof of the local uniqueness of the solution is given
in Section~\ref{sec:uniqueness}.
The study of the analyticity domains in the symplectic
limit is presented in Section~\ref{sec:domain}.
We also call attention to \cite{CallejaCL18b}, which can serve
as a reading guide for this paper.

\section{Some preliminary definitions and standard
results}\label{sec:def}
In this Section, we collect some definitions
and some elementary lemmas  that will
be used in the formulation (and proof) of the results.
Of course, this Section may be considered mainly as reference
and could be skipped in a first reading.

\subsection{Conformally symplectic systems}\label{sec:CS}

We start by introducing the definition of conformally symplectic
mappings and flows (see, e.g., \cite{Banyaga02,CallejaCL13a,DettmannM96,WojtkowskiL98}).

Let $\M = \torus^n \times B$ be a symplectic manifold of dimension
$2n$ with $B \subseteq \real^n$ an open, simply connected domain with
smooth boundary. We assume that $\M$ is endowed with the standard
scalar product and a symplectic form $\Omega$. We do not assume that
$\Omega$ has the standard form. In the study of the small dissipation limit
we will assume that $\Omega$ is exact, but Theorem~\ref{whiskered} does not need
the  assumption of exactness.

\begin{definition}\label{defCS}
We say that $f:\M\rightarrow \M$ is a conformally symplectic diffeomorphism, when there exists a constant $\lambda$ such that
\begin{equation} \label{conformallysympmap}
f^* \Omega = \lambda\, \Omega\ .
\end{equation}
\end{definition}

\begin{remark}\label{lambdaconstant}
When $n=1$, any orientable manifold is
symplectic for any non-degenerate 2-form.
If we allow that $\lambda$ is a function,
any diffeomorphism
is conformally symplectic.

When $n\geq 2$, any function $\lambda$
satisfying \eqref{conformallysympmap} has to be a constant
for a connected manifold $\M$ (see \cite{Banyaga02}).
It suffices to observe that
\[
0 = f^* d \Omega =  d (f^* \Omega)  = d( \lambda \Omega) =
d \lambda \wedge \Omega \ .
\]

It is shown in \cite{Banyaga02}
that, when the dimension of the space is $4$ or higher,
the above implies that $d\lambda = 0$ (a simple argument is
just to use locally the Darboux form of $\Omega$).

Throughout this paper we always consider $\lambda$ constant
since the whiskered tori we are concerned with only appear
when $n \ge 2$.
\end{remark}

\begin{remark}
The constant $\lambda$ will be real when we consider
real maps. It will be a complex
constant in Section~\ref{sec:domain} devoted
to the analyticity properties in $\eps$, where it is natural
to consider complex maps.
In the physical applications,  the complex maps are
such that they take real values for real arguments.
\end{remark}

\subsection{Expressions in coordinates}\label{sec:expressions}
In some computations later, we will find it convenient to
use matrix notation for the computations (recall that we are
assuming that the phase space we are considering is Euclidean).

We will consider the tangent of the phase space endowed with the
standard inner product, which does not depend on the base point
(later in Section~\ref{sec:triviality}, we will find it useful
to use metrics depending on the point).

Given a linear operator $A: \T_x \M \rightarrow \T_y \M$
(where $\T_x\M$ denotes the tangent space of $\M$ at $x$),
we denote the
adjoint $A^T$ as the linear operator  from $\T_y \M$ to $\T_x\M$
 that satisfies
$$
< u, A^T v >  = < Au, v>  \quad \forall u  \in \T_x \M, v \in \T_y\M\ .
$$
Once we fix the inner product, we can identify the symplectic
form $\Omega$ with an operator:
\begin{equation}\label{identification}
\Omega_x(u, v) = \langle u, J_x v\rangle  \quad \forall u, v \in \T_x \M\ .
\end{equation}
The asymmetry of the symplectic form means that
\[
J^T_x = - J_x\ .
\]
Using the identification \eqref{identification}  between the symplectic
form and the  operator $J_x$, we see that a map $f$ is conformally symplectic
if and only if
\begin{equation}\label{conformallysym}
Df^T(x) J_{f(x)} Df(x)  = \lambda J_{x}\ .
\end{equation}
In this paper we will not assume that $J_x$ is constant or that it
has the standard form $J_x = \begin{pmatrix} 0 & \Id \\ -\Id &0 \end{pmatrix} $.
Non-constant symplectic forms appear naturally in the study of
neighborhoods of fixed points, in the study of PDE's, etc.

\begin{remark}
The relation between the operator $J$, the symplectic form
and the metric
is a very useful tool in modern symplectic geometry.
We mention the books (\cite{Berndt01, Cannas01,  McDuffS17}). A much deeper study of
the applications of $J$ as a complex structure is in
\cite{McDuffS12}.

In this paper, we will indeed use the relation between
the operator $J$ and the metric to obtain a result on the global structure of the center bundle,
namely, that it is diffeomorphic to a product bundle,
see Lemma~\ref{lem:istrivial}.
\end{remark}

\subsubsection{Exactness}
\label{sec:exactness}
We say that the symplectic form is exact when there exists $\alpha$ such that
\[
\Omega = d\alpha\ .
\]
A map  $f$  is exact when there exists a single-valued function $G$, such that
$$
f^* \alpha - \lambda \alpha = d G\ .
$$
Note that the fact that the mapping is conformally symplectic
can be written as $d( f^* \alpha) = f^*\Omega = \lambda d \alpha$, which gives
$d(f^*\alpha - \lambda \alpha) = 0 $. Hence, exact maps are conformally
symplectic.

When $\lambda \ne 1$,  the paper \cite{CallejaCL13a} does
not need exactness of the map to produce invariant tori.
Nevertheless, in the symplectic case ($\lambda = 1$) it is well known
that the exactness is a necessary condition to have Lagrangian invariant tori
which are homotopically non-trivial (the case of maximal
homotopically trivial tori is discussed in \cite{FoxL15}).

Lower dimensional tori can exist in non-exact symplectic systems,
but if the tori have some non-trivial homotopy, one needs
that some cohomology of $f^*\alpha - \alpha$ vanishes.
The considerations of exactness come into play only when we
consider the symplectic limit in Section~\ref{sec:domain}.

For simplicity, in this
paper,
we will assume that $f$ is exact, even if parts of
this assumption can be weakened  to the vanishing of
some cohomology class of $f^*\alpha - \alpha$ depending on the
topology of the embedding of the  torus.

\subsection{Diophantine properties}

We will assume that  frequency vectors of the whiskered tori
 satisfy the following Diophantine inequality.

\begin{definition}\label{def:DC}
Let $\omega \in\real^d$, $d\leq n$, $\tau\in\real_+$.
Let the quantity $\nu(\omega;\tau)$ be defined as
\beq{DC}
\nu(\omega;\tau) \equiv \sup_{k\in\integer^d\setminus\{0\}}
\Big(|e^{2\pi ik\cdot\omega}-1|^{-1}\ |k|^{-\tau}\Big)\ ,
\eeq
where $\cdot$ denotes the scalar product and $|k|\equiv |k_1|+...+|k_d|$ with $k_1$, ..., $k_d$
the coordinates of $k$.
We say that $\omega$ is Diophantine of class $\tau$ and constant
 $\nu(\omega;\tau)$, if
$$
\nu(\omega;\tau) < \infty\ .
$$
We denote by $\D_d(\nu,\tau)$ the set of Diophantine vectors of class $\tau$ and constant $\nu$.

For $\lambda\in\complex$, we define the quantity $\nu(\lambda;\omega,\tau)$ as
\begin{equation}\label{nudefined}
\nu(\lambda;\omega,\tau) \equiv \sup_{k\in\integer^d\setminus\{0\}}
\Big(|e^{2\pi ik\cdot\omega}-\lambda|^{-1}\ |k|^{-\tau}\Big)\ .
\end{equation}
We say that $\lambda$ is Diophantine with respect to $\omega$ of
class $\tau$ and constant $\nu(\lambda;\omega,\tau)$ if
$$
\nu(\lambda;\omega,\tau)<\infty\ .
$$
\end{definition}

Note that the quantity \eqref{nudefined} makes sense for
any complex number $\lambda$. In Theorem~\ref{whiskered} we will need
only to consider  $\lambda \in \real$, but in Theorem~\ref{thm:domain}, we
will use complex $\lambda$. Note that when $|\lambda| \ne 1$,
we have that $\nu(\lambda; \omega, \tau) < \infty$. But, as $\lambda$
approaches the unit circle, depending on the limit on the unit
circle,  the $\nu$ could
either become unbounded or remain bounded.

Note that the definition of $\lambda$ Diophantine with respect to
$\omega$ makes sense even if $\omega$ itself is not Diophantine
(in particular if $|\lambda| \ne 1$, $\lambda$ is Diophantine with
respect to $\omega$ for all $\omega$, even rational ones).
Of course, the Diophantine exponents of $\omega$ and of
$\lambda$ with respect to $\omega$ are independent.

In our applications, however, it is natural to assume both that
$\omega$ is Diophantine and that $\lambda$ is Diophantine with respect to
$\omega$. For simplicity of the notation we will only use
a common exponent that works both for the Diophantine exponent
of $\omega$ and for the exponent of $\lambda$ with respect to $\omega$.

\begin{remark}
In the study of analyticity domains in Section~\ref{sec:domain},
we will consider $\omega$ fixed, but $\lambda$ will change.
Hence, we prefer to think of $\nu(\lambda; \omega, \tau)$
mainly as a function of $\lambda$ and think of $\omega$ as a
parameter that remains fixed.
\end{remark}

\subsection{Invariant rotational tori}\label{sec:inv}
Let $\Upsilon \subset \M$ be diffeomorphic to a torus.
We say that $\Upsilon$ is a rotational
torus for a map $f$, when $f(\Upsilon) = \Upsilon$ and
the dynamics of $f$ restricted to $\Upsilon$ is conjugate to
a rotation. Precisely, we start with the following definition
of rotational invariant tori, non necessarily of maximal dimension.

\begin{definition}
Let $f$ be a differentiable diffeomorphism of  $\M$.
For $0<d\leq n$, let $K: \torus^d \rightarrow \M$ be a
differentiable embedding.
Let $\omega \in \real^d$, denote
by $T_\omega:\torus^d \rightarrow \torus^d$ the rotation of vector $\omega$,
namely $T_\omega(\th) = \th + \omega$.

We say that $K$ parameterizes a rotational invariant torus,
if the following invariance equation holds:
\begin{equation}\label{invariance}
f \circ K  = K \circ T_\omega\ .
\end{equation}
\end{definition}

The relation \eqref{invariance} appears frequently in ergodic
theory and it is described as \emph{``The rotation $T_\omega$ is
a factor of $f$''} or \emph{``The rotation $T_\omega$ is semiconjugate to
$f$"}.
\medskip

As it is well known (\cite{Moser67}),
to find an invariant torus with prescribed frequency
of a conformally symplectic
system, we will need to introduce a \emph{drift} parameter (\cite{CallejaCL13a}) and,
precisely, we will consider a family $f_\mu$ of conformally symplectic mappings. Then, we
will try to find a parameter vector $\mu\in\real^d$ and an embedding of the torus
$K$ in such a way that the following invariance equation is satisfied:
\begin{equation}\label{inv}
f_\mu \circ K = K \circ T_\omega\ .
\end{equation}

Note that the equation \eqref{inv} is an equation for both $\mu$ and $K$.
The equation \eqref{inv} will be the centerpiece of our analysis.
We will develop a quasi-Newton method for it, using the geometric
properties of the map to analyze the linearization of \eqref{inv}.

\subsubsection{Normalization}\label{sec:normalization}
The equation \eqref{inv} is underdetermined.
Note that if $(\mu, K)$ is a solution of \eqref{inv},
then, so is  $(\mu, K\circ T_\sigma)$ for any $\sigma \in \real^d$.
Note that this lack of uniqueness can be interpreted as choosing the origin
in the space of the parametrization.

We will show that, in many cases, the choice of
the origin of the parameterization  is the only source of
local non-uniqueness. By choosing some normalization that fixes
the origin,  we will show
that one obtains  local uniqueness (i.e. the
normalized parameterizations of tori which are close
enough coincide) and, hence, one can discuss
smooth dependence on parameters, etc.  Of course, global uniqueness
is false unless one makes global assumptions.

In Section~\ref{sec:uniqueness} we give a precise statement of the local uniqueness of
the solution of the invariance equation \equ{inv} under some
normalization that fixes the origin of coordinates.

Before starting with the main results let us prove the
easy result given by Proposition~\ref{prop:easy}.

If $\omega$ is nonresonant
(i.e., for any $k \in \integer^{d} \setminus \{0\}$, we have
$|\omega\cdot k| \not \in \nat$), then $\{T^n_\omega (\th)\}$ is
dense in the torus for any $\th$ and
the embedding $K$ is almost uniquely determined as stated in the following Proposition.

\begin{proposition}
\label{prop:easy}
Let $f$ be a differentiable diffeomorphism of $\M$.
Assume that $\omega\in\real^d$ is nonresonant.
Let $K_1, K_2$ be continuous mappings satisfying \eqref{invariance}
and their ranges have a non-empty intersection.
Then, there exists $\sigma \in \real^d$ such that
\[
K_1 = K_2 \circ T_\sigma\ .
\]
\end{proposition}

\begin{proof}
We can find $\sigma_1, \sigma_2 \in \torus^d$ such that
$K_2(\sigma_1)= K_1(\sigma_2)$. Applying $f$ to both sides and using
\eqref{invariance}, we obtain $K_2(\sigma_1 + \omega) = K_1(\sigma_2 + \omega)$.
Repeating the argument, we obtain that for all $j \in \nat$,
$K_2(\sigma_1  + j \omega) = K_1( \sigma_2 + j \omega)$, which gives
\begin{equation} \label{coincidence}
K_2( \th ) = K_1(\th + \sigma_2 - \sigma_1 )\ ,\qquad \th \in \{j \omega + \sigma_1\}_{j \in \nat}\ .
\end{equation}

Since $\{\sigma_1+j\omega \}_{j \in \nat}$ is dense in the torus and $K_1, K_2$
are continuous,
we have the equality \eqref{coincidence}  for all $\th$ in the torus.
\end{proof}

\subsection{Definition of function spaces}
To make precise estimates on the quantities involved in the proof, most notably
the error associated to an approximate solution of \equ{inv}, we need to fix a function space and a norm, precisely
the space and norm of analytic functions as in the definition given below.

\begin{definition}\label{def:spaces}
For $\rho >0$ we denote by $\torus_\rho^d$ the set
\[
\torus^d_\rho = \{ z \in \complex^d/\integer^d \, :\ {\rm Re}(z_j)\in\torus\ ,\quad |\Im(z_j)| \leq \rho\ ,\quad j=1,...,d\}\ .
\]

Given $\rho>0$ and a Banach space $X$,
we denote by $\A_\rho(X)$ the set of functions
from $\torus^d_\rho$ to $X$  which are analytic in the interior of
$\torus_\rho^d$ and that extend continuously to the boundary
of $\torus^d_\rho$.

 We endow $\A_\rho$ with the norm
$$
\|f\|_\rho = \sup_{z\in\torus_\rho^d}\ |f(z)|\ ,
$$
which makes it into a Banach space.
\end{definition}

For later use, we introduce the norm of a vector valued function $f=(f_1,\ldots,f_n)$ as
$\|f\|_\rho=\sqrt{\|f_1\|^2_\rho+\ldots+\|f_n\|^2_\rho}$ and the norm of
an $n_1\times n_2$ matrix valued
function $F$ as $\|F\|_\rho=\sup_{\chi\in\real^{n_2}_+,|\chi|=1}
\sqrt{\sum_{i=1}^{n_1}(\sum_{j=1}^{n_2}\|F_{ij}\|_\rho\, \chi_j)^2}$.
In this way, we have the customary inequalities for the norm of
the product of a matrix and a vector.

\subsection{Cocycles and invariant bundles}
\label{sec:cocycle}

In this Section, we recall the standard definitions on growth
properties of the products \eqref{product} below. These properties are quite
standard (\cite{SackerS74, Coppel78,MitropolskySK03}).
The hyperbolicity of cocycles -- over more
general systems than rotation --  is
treated specially in \cite{ChowL95}.
For a pedagogical treatment, see also the expository
chapters of  \cite{deViana17}.

 Later on, in
Section~\ref{sec:geometry}, we will see that there is a deep
relation between the growth properties of the cocycle and the
conformal symplectic properties of the map.
These properties are somewhat surprising since
they show an interplay between the dynamics and the geometric
properties of the tori. Quite notably, we will establish that the
center bundle of a rotational invariant torus is trivial
in the sense of bundle theory. That is, it can be written as a product
bundle.

For the sake of notation, we will assume
that $\M$, the phase space, is an Euclidean manifold.
In that way, we can identify the different tangent spaces.

Since we are identifying the spaces, we can think of
all the factors $Df_\mu \circ K \circ T_{j\omega}$
as analytic functions from $\torus^d_\rho$ to $n\times n$ matrices.

We will see that in our
study of corrections to the invariance equation,
to reduce the error we are led to the study of
products of the form
\begin{equation} \label{product}
\Gamma^j \equiv Df_\mu \circ K \circ T_{(j-1)\omega} \times
Df_\mu \circ K \circ  T_{(j-2)\omega} \times  \cdots \times  Df_\mu\circ K\ .
\end{equation}
This is a particular case of products of
the form
\begin{equation}\label{cocyclerotation}
\Gamma^j = \gamma\circ T_{(j-1) \omega} \times \gamma\circ T_{(j-2) \omega} \times \cdots \times \gamma\ ,
\end{equation}
when we take:
\begin{equation}\label{gammaform}
\gamma(\th) = Df_\mu\circ K(\th)\ .
\end{equation}
Note that if the torus was invariant, then \eqref{product} would become
\[
\Gamma^j = Df^j_\mu\circ K\ ,
\]
which makes it clear that the cocycle $\Gamma^j$ has a dynamical interpretation for invariant tori.

An important property of products  of the form \eqref{cocyclerotation} is:
\begin{equation}\label{cocycle}
\Gamma^{j+m} = \Gamma^j\circ T_{m\omega}\ \Gamma^m .
\end{equation}

Products of the form \eqref{cocyclerotation} have been studied extensively
in the mathematical literature under the name of
\emph{``quasiperiodic cocycles''} or
\emph{``cocycles over a rotation''}. The property \eqref{cocycle} is
called the \emph{cocycle} property.
In this Section, we will collect some properties of
the cocycles, specially the properties of asymptotic growth
and exponential trichotomies (\cite{SackerS74,Coppel78}).

\subsubsection{Exponential Trichotomies}
\label{sec:trichotomies}

The asymptotic growths of the products \eqref{cocyclerotation} are
important for the study of the invariance equation.

\begin{definition} \label{def:trichotomy}
We say that the cocycle \eqref{product} admits an exponential trichotomy when
we can find a decomposition
\begin{equation}\label{splitting}
\real^n =  E_{\th}^s \oplus E_{\th}^c \oplus E_{\th}^u\ ,
\qquad \theta\in\torus^d
\end{equation}
and rates of decay
\beqano
&\lambda_- < \lambda_c^- \le \lambda_c^+ <  \lambda_+\ , \nonumber\\
& \lambda_- < 1 < \lambda_+
\eeqano
and a constant $C_0 > 0$ that characterize the decomposition:
\begin{equation}\label{growthrates}
\begin{split}
& v\in  E_{\th}^s\ \iff   |\Gamma^j(\th) v | \le C_0 \lambda_-^j |v|,\quad j \ge 0 \\
& v\in  E_{\th}^u\ \iff   |\Gamma^j(\th) v | \le C_0 \lambda_+^j |v|,\quad j \le 0 \\
&v\in  E_{\th}^c\ \iff
\begin{matrix}
|\Gamma^j(\th) v | \le C_0 (\lambda_c^-)^j  |v|, \quad j \ge 0 \\
|\Gamma^j(\th) v | \le  C_0 (\lambda_c^+)^j |v|, \quad  j \le  0\ . \\
\end{matrix}
\end{split}
\end{equation}
\end{definition}

We will refer to $C_0$ and the rates
$\lambda_- , \lambda_c^- , \lambda_c^+,  \lambda_+$
as the constants characterizing the splitting. We will not introduce
the dependence of these quantities on the splitting in the notation.

Even if we assume that the space is Euclidean, we allow that the
invariant sub-bundles are not trivial in the sense of bundle theory.
That is, they are not isomorphic to a product bundle.

It will be convenient for notation to consider trichotomies as
just a pair of dichotomies
\begin{equation}\label{dichotomies}
\begin{split}
&\real^n =  E^s \oplus E^{\hat s} \\
&\real^n =  E^{\hat u} \oplus E^{u}\ ,
\end{split}
\end{equation}
where we denote by $E^{\hat s} = E^c \oplus E^u$
($\hat s$ stands for the symbols that are not $s$).
We use similar notations for other symbols.

The  dichotomies  in \eqref{dichotomies} are also characterized by
rates of growth.  If we have both dichotomies, we can reconstruct
the spaces in the trichotomy taking
$E^c_\th = E^{\hat s}_\th \cap E^{\hat u}_\th$.

\subsubsection{General properties of splittings and their distances}
\label{sec:splittings}

Given a splitting  of the Euclidean space as in \eqref{splitting}, we denote by
$\Pi^s_\th, \Pi^c_\th, \Pi^u_\th$ the projections corresponding to
the spaces in the splitting. Note that the projections
$\Pi^{s/c/u}_\th$ depend
on the whole splitting, but we do not include this in the
notation.

We can think of $\Pi^\sigma_\th$ both as mappings from
$\real^n$ to $\real^n$ or as mappings from $\real^n $ to
$E^\sigma_\th$. The advantage of the former is that we can
compose mappings corresponding to different splittings.
We will not always emphasize the interpretation that is taken.

The space of possible splittings is related to the Grassmanians
(\cite{MilnorS74}).
In this paper, in an iterative step, we will need to refine
the splittings. Hence, it will be important for us to
give an efficient description of the splittings close to
another one and also give precise definitions of
how close splittings are and define the convergence.
Of course, for our applications, some smoothness considerations will be needed
and the norms involved in the measure of distances will be smooth norms.

\subsubsection{Explicit description of a neighborghood of
splittings}

Given a splitting $E$ and another splitting $\widetilde E$ close
to it, we can write uniquely  each of the spaces in $\widetilde E$ as
the graph of a function from the corresponding space in
$E$ to the complementary spaces in $E$. That is,
we can find linear  functions
$A^\sigma_\th: E^\sigma_\th \rightarrow E^{\hat \sigma}_\th$
(recall that $E^{\hat \sigma}$ denotes the
sum of the spaces in the splitting which
are not indexed by $\sigma$) in such a way that

\begin{equation} \label{graph}
\widetilde E^\sigma_\th = \{ v \in \real^n, v = x + A^\sigma_\th x \, | \,
x \in E^\sigma_\th \}\ .
\end{equation}
In the sequel we fix once and for all an analytic splitting that we
will call  the
\emph{``reference splitting''}, and we will describe
all the splittings we use in terms of the $A^\sigma$
as in \eqref{graph}.

In our application, we  will see that we can take as the
reference splitting the initial approximate splitting in the hypothesis
of Theorem~\ref{whiskered}. Under the assumption that the  initial
error is small enough, we will see that all the splittings that
appear during the iterative procedure are inside the neighborhood
of maps that can be described as in \eqref{graph}.

It is useful to think of \eqref{graph} as providing a system of
coordinates of a neighborhood of the reference splitting in the
(highly nonlinear) space of all splittings.

\subsubsection{Several notions of distance among bundles and splittings}
\label{sec:distancesplitting}

In Grassmanian geometry (\cite{MilnorS74}) it is customary to define
the distance between subbundles of an ambient bundle,
by fixing a metric in the ambient bundle (in our case, the ambient
bundle will be the tangent space to the phase space)
and introducing the orthogonal projection $P^\perp_{E_\th}$, where $E_\th$ is
the fiber of a bundle $E$ over $\th$.

For $\rho>0$, the distance between two
subbundles $E, \tE$ of the ambient bundle is defined as:
\begin{equation} \label{distancebundle}
dist_\rho(E, \tE)  = \| P^\perp_{E_\th} - P^\perp_{\tE_\th} \|_\rho\ .
\end{equation}

Note that the intrinsic distance between two bundles depends
just on the bundles themselves and they do not depend on whether
they are part of a  splitting.

The distance between two splittings of the tangent space
into bundles can be measured by the
maximum of the distances between the corresponding bundles that give
the splitting.  Later we will discuss other notions of
distance among splittings.

In our case, we will consider bundles based on
a parameterization as $\| \cdot \|_\rho$, analytic norms in a complex
extension of the torus as in Definition~\ref{def:spaces}. This will
induce an analytic distance between bundles based on a parameterization.

\vskip.1in

\leftline{\emph{Other equivalent ways of measuring the
distance among splittings}}

It is clear that \eqref{distancebundle} satisfies the triangle inequality
and it is indeed a distance. As we will see later, for
our application there will be other quantities (not distances)
which bound \eqref{distancebundle} from above and below, but
which are easier to compute. We will also need to study
upper bounds for the distance between splittings and how do
they change when the problem changes.
We will try to work with the projections corresponding to
the splitting $\Pi^\sigma$  and not with the orthogonal projections.

If we have two splittings $E^\sigma$ and $\widetilde E^\sigma$ with projections
$\Pi^\sigma$ and $\widetilde \Pi^\sigma$, we can measure the distance  between the
splittings by:
$$
\left(
 \max_{\sigma, \sigma' \in \{s,c,u\} \atop \sigma \ne \sigma'}
\max(
\| {\widetilde  \Pi}^\sigma \Pi^{\sigma'} \|,
\| \Pi^\sigma {\widetilde \Pi}^{\sigma'} \| )
\right)^{-1}\ ,
$$
where $\| \cdot \|$ can stand for any  norm. As before, we will take
analytic norms in a neighborhood of the torus.

We can also use as a measure of the distance between the reference
splitting and a splitting given by \eqref{graph}, the quantity
\begin{equation} \label{distance2}
\max_\sigma \| A_\theta^\sigma\|\ ,
\end{equation}
where again $\| \cdot \|$ stands for a smooth norm.
Note that we are thinking of the linear  maps $A_\theta^\sigma$ as mappings
from $E^\sigma_\th$ to $\real^n$.

More generally, if we fix a reference splitting, given two
splittings $E_1^\sigma$, $E_2^\sigma$  that can be parameterized
by giving the  mappings $A_{\theta,1}^\sigma$, $A_{\theta,2}^\sigma$
as above, the quantity
\begin{equation} \label{distance3}
\max_\sigma \| A_{\theta,1}^\sigma - A_{\theta,2}^\sigma\|
\end{equation}
gives a measure of the distance of the splittings.

Note that \eqref{distance2} and \eqref{distance3} are
not distances. Nevertheless, they can be used in
place of a distance in the sense that, in a small
neighborhood in the space of splittings,
there
are constants that bound one in terms of any of the others.
The reason for this equivalence  is that there are  algebraic expressions
giving each of  $P^\perp_{E^\sigma}$, $\Pi^\sigma$, $A_\theta^\sigma$  in terms of
the others, see \cite{deViana17}.

\subsubsection{Standard properties of splittings satisfying \eqref{growthrates}}

One of the consequences of
\eqref{growthrates} (see \cite{Coppel78})
 is that the splittings
depend continuously on $\th$ (actually in H\"older fashion)
 and, hence that the projections $\Pi^\sigma$,
$\sigma =  s,u,c$, are  uniformly bounded, see \cite{SackerS74}.
Using the fact that the dynamics on the base is a rotation, we
will later bootstrap the regularity of the splittings to analytic
(see \cite{HaroL18}).
The H\"older continuity remains valid (and is optimal) if
the dynamics on the base is more complicated than a rotation.

Another (slightly non-trivial\footnote{
The result  would  be trivial if the characterization
was valid for any $C_0$ rather than a different number.  The  key
of the proof is to show that one can redefine the norms
in such a way that $C_0 =  1$ and that the norm of
the operators is also small.}, see \cite{SackerS74})
 consequence of \eqref{growthrates} is that the bundles characterized by
\eqref{growthrates} are invariant
in the sense that:
$$
\gamma(\th) E^\sigma_\th = E^{\sigma}_{\th + \omega}\ .
$$

\subsubsection{Approximately invariant splittings}
Given a splitting $E^s_\th  \oplus E^c_\th \oplus E^u_\th$ and
a cocycle, $\gamma(\th)$, we
define
\beq{sigmasigma}
\gamma_\theta^{\sigma, \sigma'} =
 \Pi^\sigma_{\th +\omega}  \gamma(\th) \Pi^{\sigma'}_\th\ .
\eeq
The splitting is invariant under the cocycle if
and only if
\[
\gamma_\theta^{\sigma, \sigma'} \equiv 0\ , \qquad \sigma \ne \sigma'\ .
\]
Again, we note that we can think of the
$\gamma^{\sigma, \sigma'}$ either as linear  maps in $\real^n$
or, more geometrically, as maps from $E^\sigma$ to $E^{\sigma'}$.
Thinking of them as maps in $\real^n$ allows us
to compare maps in different spaces and will be useful in
perturbative calculations.

Hence, it is natural to measure the lack of invariance
of the splitting under the cocyle $\gamma$ by
\begin{equation}\label{approximatenorm}
\I_\rho(\gamma, E)  \equiv  \max_{\sigma, \sigma' \in \{s,c,u\} \atop \sigma \ne \sigma'}
\| \gamma_\theta^{\sigma, \sigma'} \|_\rho\ ,
\end{equation}
where $\| \cdot \|_\rho $ is the supremum on $\torus^d_\rho$.
The choice of a smooth norm is consistent with what
we will do in the paper, but of course, other smooth norms
could also be considered.

We note that given a splitting specified with some finite error
and some $\gamma$ also specified with a finite error, it is
possible to obtain error bounds for \eqref{approximatenorm}
with finite calculations. In particular, they can be verified
with computer assisted proofs.

The following definition taken from \cite{deViana17},
formulates  a notion of
approximately invariant splittings and approximately
hyperbolic cocycles.

\begin{definition}\label{approximatehyperbolic}
Given a cocycle $\gamma$ and a 3-splitting
$E = E^s_\th \oplus E^c_\th \oplus E^u_\th$,
we  can write the  cocycle in blocks as in
\eqref{sigmasigma}.

We say that the splitting is $\eta$ approximately invariant
if
\begin{equation}\label{approximatelyinvariant}
 \max_{\sigma, \sigma' \in \{s,c,u\} \atop \sigma \ne \sigma'}
\| \gamma_\theta^{\sigma, \sigma'} \|_\rho \le \eta\ .
\end{equation}

We say  that the cocycle is approximately  hyperbolic
with respect to the splitting $E$,
if the cocycle
\[
\tilde \gamma_\theta =
\begin{pmatrix}
& \gamma^{s,s}_\theta & 0 & 0 \\
&0 &  \gamma^{c,c}_\theta & 0  \\
&0 & 0 & \gamma^{u,u}_\theta
\end{pmatrix}
\]
satisfies the trichotomy properties in Definition~\ref{def:trichotomy},
with $\gamma_\th^{\sigma, \sigma}$ defined in \equ{sigmasigma}.
\end{definition}

We again note that the notion of measurements of
approximate invariance involves the use of a smooth norm
to measure the distance.

\subsection{The closing lemma for approximately invariant
splittings}\label{sec:closinglemma}

In this Section, we present and prove a result generalizing slightly
and making more precise
Proposition 5.2 of \cite{FontichLS}.
The main result in this Section is Lemma~\ref{lem:closing}
that roughly says that if a splitting is approximately invariant
(with a sufficiently small error)
for a cocycle and the cocycle is approximately hyperbolic
for this splitting in the sense of Definition~\ref{approximatehyperbolic},
then there is a true invariant splitting.
The proof of Lemma~\ref{lem:closing} will be given in Appendix~\ref{app:closing}, where
we present very detailed estimates.

For the applications we have in mind, it will be important that the
size of the corrections of the bundles needed
to make them invariant  can  be bounded by the error in the
invariance and that the constants
involved in the lemma can be chosen uniformly in a neighborhood of
splittings and cocycles. We have included these precisions
in the statement of Lemma~\ref{lem:closing}.

The ideas of the proof are  standard among the specialists
in hyperbolic dynamical systems.   The basic idea
is that we use the parameterization of splittings
by the mappings $A_\theta^\sigma$ as in \eqref{graph}, formulate
an invariance equation and manipulate it in a form that can
be shown to be a contraction  in appropriate spaces.
After that, we will need to estimate the changes in
the hyperbolicity characteristics.

If we fix a reference  splitting $E = E^s \oplus E^c \oplus E^u$,
we can parameterize
all the splittings close to it by a triple of linear functions
as in \eqref{graph}. Note that we are not assuming that the reference
splitting is invariant, but we will assume it is
approximately invariant. The results described below will hold for a
sufficiently small neighborhood of the splittings in the Grasmannian
space. More precisely, we will consider splittings that can be
expressed as in \eqref{graph} with sufficiently small $\|A_\theta^\sigma\|$.

\begin{remark}
In our application to KAM theorem,  we can take as the reference splitting the
one in the first iterative step. As we will show,
if the error in the first
approximation is small enough, all the splittings remain in this neighborhood.
This smallness condition in the first approximation will
be one of the conditions that appear in the final
smallness conditions of our main result.
\end{remark}

\begin{lemma}\label{lem:closing}
Assume that we have fixed an analytic reference 3-splitting $E_0$
defined on $\torus^d_\rho$. Assume that this
splitting is approximately hyperbolic and
$\eta_0$ approximately invariant with $\eta_0$ sufficiently small.

We denote by $\U$ a
sufficiently small neighborhood of this splitting, so
that all the splittings can be parameterized as graphs of
linear
maps $A^\sigma_\th$ as in \eqref{graph}
with $\|A^\sigma_\th\|_\rho  < M_1$ for some $M_1>0$.

Let $\gamma$ be an analytic cocycle over a rotation defined on
$\torus^d_\rho$.
Assume that $\| \gamma \|_\rho  < M_2$ for some $M_2>0$.

Let $E$ be an analytic 3-splitting  in the neighborhood $\U$.

Assume that $E$ is $\eta$ approximately invariant under $\gamma$
and that $\gamma$ is approximately hyperbolic for
the reference splitting in the sense of Definition~\ref{approximatehyperbolic}.

Assume that $\eta$ is sufficiently small (depending only on the
neighborhood $\U$ and $M_2$).

Then, there is a locally unique splitting $\widetilde E$ invariant under
$\gamma$ close to $E$ in the sense that
$$
\dist_\rho( E, \widetilde E) \le C \eta
$$
for some constant $C>0$.

The splitting $\widetilde E$ satisfies a trichotomy in the sense
of Definition~\ref{def:trichotomy}.

The constant $C$ can be chosen uniformly
depending only on $M_1$, $M_2$.
\end{lemma}

We stress  that Lemma~\ref{lem:closing} does not require
any non-resonance condition on the frequency $\omega$,
but it is quite important that the dynamics in the base is
a rotation.

Note that  there is no domain loss. The
changes required are in  same spaces
where we assumed the error. This is because
the proof of Lemma~\ref{lem:closing} will be just
an application of the contraction mapping theorem, so it does
not incur in any loss of domain. If the dynamics on the base of
the cocycle was not a rotation, we could not be using analytic
regularity.

The proof of Lemma~\ref{lem:closing} is given in Appendix~\ref{app:closing}.
It is based on formulating a functional equation for the quantities $A^\sigma$
given in \equ{graph}. For future use, it is convenient to mention now that
after some manipulation one is led to solve the following equations
for the dichotomy between $s, \hat s$ spaces:
\begin{equation} \label{invariance2-s}
\begin{split}
& (\gamma_\th^{\hat \st, \hat \st } )^{-1} \left[ A^\st_{\th + \omega}  \big(
\gamma^{\st, \st}_\th  +
\gamma^{\st, \hat \st}_\th  A^\st_\th \big)-
\gamma^{\hat \st, \st}_\th \right]  = A^\st_\th\ ,\\
&\left[ \gamma^{\st, \hat \st}_{\th - \omega} +
 \gamma^{\st, \st}_{\th-\omega}A^{\hat \st}_{\th -
\omega}
- A^{\hat \st}_\th \gamma_{\th -\omega}^{\hat s, s} A^{\hat \st}_{\th -\omega}\right]
 \left( \gamma^{\hat \st, \hat \st}_{\th -\omega}\right)^{-1}
= A^{\hat \st}_\th
\end{split}
\end{equation}
and the following equations for the dichotomy between $\un, \hat \un$ spaces:
\begin{equation} \label{invariance2-u}
\begin{split}
& \left[-A^\un_{\th} \gamma^{\un, \hat \un}_{\th -\omega}A^\un_{\th -\omega}
+ \gamma^{\hat u, u}_{\th - \omega}
+ \gamma^{\hat u, \hat u}_{\th - \omega} A^\un_{\th - \omega} \right]
(\gamma^{\un,\un}_{\th - \omega})^{-1} =
 A^\un_{\th}  \\
& (\gamma_\th^{\un, \un } )^{-1} \left[ A^{\hat \un}_{\th + \omega}  \big(
\gamma^{\hat \un, \hat \un}_\th  +
\gamma^{\hat \un, \un}_\th  A^{\hat \un}_\th \big)-
\gamma^{\un, \hat \un}_\th \right]  = A^{\hat \un}_\th\ .
\end{split}
\end{equation}

\subsubsection{Estimating the change of hyperbolicity properties in terms of the change in
  the cocycles}

In Lemma~\ref{lem:closing} we showed that given an approximately invariant splitting satisfying
the hyperbolicity conditions in Definition~\ref{approximatehyperbolic} there is a truly
invariant splitting; we estimated the distance between the approximately invariant
splitting and the truly invariant one and the diagonal blocks of the cocycle.

The goal of this Section is to obtain estimates on the hyperbolicity properties
of the new cocycles in terms of the size of the changes. Estimates are, of course,
not unique and we find it interesting to provide several versions which may be
useful in different circumstances. In our case, hyperbolicity properties involve
rates and constants and there are trade--offs among them. Note that in contrast to
\cite{SackerS74} we do not study optimal rates, but only bounds on the rates.

In Lemma~\ref{general} we provide simple, but generally applicable, estimates and
in Lemma~\ref{precise} we provide more quantitative estimates. We will use
Lemma~\ref{general} to justify the repeated application of our Newton step and to
obtain estimates of the change of the splitting and the diagonal blocks of the
whole process. Then, we can apply Lemma~\ref{precise} to obtain information
on the change experienced in the whole Newton process.

We now need to introduce the following notation. Given
$\gamma$ as in \eqref{gammaform}, we denote $\gamma_j (\theta) = \gamma(\theta +
j\omega)$ so that
\beq{iteration}
\Gamma_m^k \Gamma_j^m(\theta) = \Gamma_j^k(\theta)\ ,
\eeq
where $\Gamma_0^k=\Gamma^k$ and $\Gamma_j^m(\theta) = \Gamma_0^{m-j}\circ T_{j
\omega}(\theta)$.

We will assume that for some $C_0, \xi \in \real_+$
\[
\| \Gamma^m\|_\rho \leq C_0 \xi^m\ .
\]
Note that multiplying  the generator of
the cocycle by a constant number (hence the $\Gamma$ is
multiplied by an exponential in $n$),
we can arrange that the quantity $\xi$ which measures the growth of
the cocycle satisfies $\xi < 1$. The multiplication by a constant
does not change the invariant spaces and only changes the smallness
conditions by a constant.

We want to investigate conditions on $\| \gamma - \tilde \gamma\|_\rho $ so that
we can ensure that $\widetilde \Gamma$ obtained by iterating as in \eqref{iteration} satisfies
\eqref{goodestimates} for other constants $\tilde C_0$, $\tilde \xi$.

A general estimate on changes of hyperbolicity properties is given
by the following result, whose proof is given in
Appendix~\ref{app:closing}.

\begin{lemma}
  \label{general}
  Assume that $\Gamma$ obtained from a cocycle $\gamma$ satisfies
\eqref{goodestimates}. Let $\eps^*>0$
  and let $\tilde\gamma$ be a cocycle such that
  \[\| \gamma - \tilde \gamma\|_\rho \leq \eps^*\ ;\]
  then, there exist $\tilde C_0$ and $\txi$ such that
  $\|\widetilde \Gamma^m\|_\rho \leq \tilde C_0 \txi^m$.
\end{lemma}

A more quantitative estimate on changes of hyperbolicity
properties is given by the following result, whose proof is given
in Appendix~\ref{app:closing}.

\begin{lemma}
  \label{precise}
  With the notations of Lemma~\ref{general}, let $a\equiv\|\gamma-\tilde\gamma\|_\rho$ be
  small enough, say $ a \leq {1 \over 4 C_0} $. Then, we can take
  $\txi = \xi + C a$ with an explicit constant $C$.

Furthermore, the quantity $\tilde C_0 = \tilde C_0 (C_0, \mu,a)$ in Lemma~\ref{general} can be bounded as
\beq{C0main}
\tilde C_0\leq 4 C_0^2 a \xi^{-1} \frac{\txi}{\txi - \xi}\ .
\eeq

\end{lemma}

An important corollary of Lemma~\ref{lem:closing} is the following.

\begin{cor}
If the reference splitting is approximately hyperbolic
in the sense of Definition~\ref{approximatehyperbolic}
for some cocycle $\gamma^0_\th$, then for
all the $\gamma$'s in a neighborhood of $\gamma^0_\th$, there is
an (locally unique)
 invariant splitting which is hyperbolic and the hyperbolicity
constants can be chosen uniformly.
\end{cor}

\begin{remark}
  The estimates on the hyperbolicity constants in
  \eqref{C0main} of Lemma~\ref{precise}
involve choices. One can make $C_0$ change or $\lambda$'s change.

Much of the theory (e.g. \cite{SackerS74})
is concerned with the optimal $\lambda$'s.
Note that, even for constant $2 \times 2$ cocycles, the
optimal $\lambda$ can change with a fractional power of
the perturbation.  Once we choose a slightly less optimal $\lambda$'s
we can make the change to be linear in the perturbation. Of course,
the range of validity, may be smaller if the chosen upper bound is
close to the optimal one.

Note that, the $C_0$ depends on the choice of metrics -- but the rates
$\lambda_* $ do not.  Indeed, it is customary in the theory that deals
with perturbations to observe that we can choose an \emph{``adapted metric''} so
that $C_0 = 1$.  Of course, the size of the perturbations allowed
is measured in this metric and, when $C_0$ increases,  the adapted metric
becomes more inequivalent to the original one and the perturbations allowed
may decrease.

Numerical explorations (\cite{HaroL07, HLlverge})
suggest that if one fixes a metric and studies the optimal rates
$\lambda$'s  and the optimal $C_0$, there is  a very interesting scenario
for the loss of hyperbolicity called
\emph{``bundle collapse''}. In this scenario,
the rates $\lambda$'s remain bounded
and the $C_0$ explodes. This scenario empirically presents remarkable
scaling properties.   The bundle collapse scenario is particularly
important in the breakdown of  KAM tori in conformally symplectic
systems (see \cite{CallejaF11}). The papers \cite{HaroL07, HLlverge}
presented numerical conjectures of the blow-up of the optimal values of
the rates and of the geometric properties of the bundles. These
conjectures were recently proved  in several cases
in \cite{BjerklovS08, Ohlson17, FiguerasT18, Timoudas18}.

We think it would be interesting to study the breakdown
of hyperbolicty in conformally symplectic systems. It seems
possible that the limit of zero dissipation will have some
interactions with the previously studied phenomena.
For these numerical implementations, the a-posteriori results and
the fast algorithms developed here are likely to be useful.
\end{remark}

\subsection{Whiskered tori}

The main result of this paper concerns whiskered tori, which are defined as follows.

\begin{definition}
\label{def:whiskered}
Let $f_\mu$ be a conformally symplectic system with
conformal factor $\lambda$ of a symplectic manifold $\M$.

We say that $K: \torus^d \rightarrow \M$ is a whiskered torus when:
\begin{enumerate}
\item
$K$ is the embedding of a rotational torus, that is,
$f_\mu \circ K = K \circ T_\omega$.
\item
The cocycle $Df_\mu\circ K$ over the rotation $T_\omega$
admits a trichotomy as in  Definition~\ref{def:trichotomy} with the
rates
$\lambda_-, \lambda_c^-, \lambda_c^+, \lambda_+$.
\item
The rates satisfy
$\lambda_c^- \le \lambda \le \lambda_c^+$.
\item
The spaces $E^c$ have dimension $2d$.
\end{enumerate}
\end{definition}

Somewhat surprisingly
there are relations between  the conformal symplectic properties,
the rates and the properties of the bundles. They will be explored in
Section~\ref{sec:geometry}. Notably we will show
that if an embedding satisfies Definition~\ref{def:whiskered}, then
it is isotropic, there are relations between the rates of
growth and, more surprisingly, the $E^c$ bundle is trivial.

\section{Statement of the main result, Theorem~\ref{whiskered}}
\label{sec:statement}

The main result of this paper is an a-posteriori result
about solutions of a parameterized version of \eqref{inv}.

As motivation for the
hypothesis of
Theorem~\ref{whiskered},
 assume that $K$, $\mu$ satisfy approximately \eqref{inv} with a small
error term $e$, i.e.
$$
f_{\mu}\circ K - K\circ T_\omega = e\ .
$$
If we want that $K + \Delta$, $\mu + \beta$ for some
corrections  $\Delta$, $\beta$ is a better solution, the Newton-Kantorovich method would
prescribe to choose $\Delta$, $\beta$ satisfying
\begin{equation}\label{Newtonheuristic}
Df_{\mu} \circ K \Delta - \Delta \circ T_\omega
+ (D_\mu f_{\mu}) \circ K \beta  = -e\ .
\end{equation}

If one tries to solve \eqref{Newtonheuristic} by iterating,
one is quickly led to the cocycles that were discussed in
Section~\ref{sec:cocycle}. Hence, it is clear
that the asymptotic growth of the cocycles plays a role.

As it turns out, the geometry of the problem plays also a very
important role and one of the most surprising facts
is that the conformal symplectic geometry leads to constraints on
the rates of growth.  These interactions of the geometry
with the dynamics will be explored in Section~\ref{sec:geometry}. We
anticipate that the most important results will be a surprising
triviality result for the bundle of vectors with intermediate
slow decay and the \emph{``automatic reducibility''} that
constructs  a natural system of coordinates in which the
linearized equations are very simple.  In this paper we go
beyond the results in previous papers and show in
Lemma~\ref{lem:istrivial} that the center bundle is trivial.

Using the geometry, we will show that the equations
\eqref{Newtonheuristic} can be solved. As mentioned in Section~\ref{sec:intro}, the result is a very efficient
algorithm. Of course, the a-posteriori format of
the theorem gives an analytical support to the results.

Given the important role  played by the geometry, it is clear
that the limit when the geometry changes from conformally symplectic
to symplectic is very singular. In Section~\ref{sec:domain} we
will study this singular limit in which the dissipation becomes weak.

The following Theorem~\ref{whiskered}  on the persistence of whiskered tori is
the main result of this paper.
Later, we will use it to obtain information on the analyticity
properties of the tori under dissipative perturbations (see Theorem~\ref{whiskered}).

We will consider specially the case $0<|\lambda|<1$, but Theorem~\ref{whiskered} can be stated as well for
$|\lambda|>1$, just taking the inverse of the mapping. In the discussion of
analyticity properties with respect to perturbations, we will need
to consider even complex values.  We will also consider the case
$\lambda = 1$, but, as pointed out in \cite{CCLdomain}, the case
of complex $\lambda$ with $|\lambda| = 1$, $\lambda \ne 1$
requires special considerations.
Indeed, when $\lambda$ is a root of the identity, we do not
expect that the solutions persist in general. Indeed, for generic perturbations,  it is impossible to find even formal asymptotic expansions.

\begin{theorem}\label{whiskered}
Let $\omega \in \D_d(\nu, \tau)$, $d\leq n$, as in \eqref{DC},
let $\M$ be as in Section~\ref{sec:CS} and let
$f_\mu:\M\rightarrow\M$, $\mu \in \real^d$, be a family of real analytic, conformally symplectic mappings
as in \equ{conformallysympmap} with $0<\lambda<1$. We make the following assumptions.

$(H1)$ Approximate solution:

Let $(K_0,\mu_0)$ with $K_0 :\torus^d \to \M$, $K_0 \in\A_\rho$,
and $\mu_0 \in \real^d$
define an approximate whiskered torus with frequency $\omega$ for $f_{\mu_0}$,
 so that
\begin{equation} \label{approximately-invariant}
\|f_{\mu_0}\circ K_0  - K_0 \circ T_\omega\|_\rho\leq \E
\end{equation}
for some $\E>0$.

To ensure that the composition of
$f_\mu$  and $K$ can be defined, we  will assume that the range of
$K_0$ is well inside the domain of $f_\mu$ for all $\mu$ sufficiently close to $\mu_0$.

We will assume that there is a domain
$\U \subset \complex^n/\integer^n \times \complex^n $ such that
for all $\mu_0$ such that $| \mu -\mu_0| \le \eta$,
$f_\mu$ has domain $\U$. Moreover, we assume that the range of
$K_0$ is inside the domain $\U$:
\begin{equation}
\label{compositions}
\dist( K_0(\torus^d_\rho), \complex^n/\integer^n \times \complex^n
\setminus \U) \ge \eta\ .
\end{equation}

$(H2)$
Approximate splitting:

For all the points in the torus,
there  exists a splitting of the tangent
space of the phase space, depending analytically on the angle
$\th\in\torus^d_\rho$.

These bundles are approximately invariant under the
cocycle $\gamma(\th) = Df_{\mu_0} \circ K_0(\th)$, namely
the quantity in \eqref{approximatenorm} is
smaller than $\E_h$, for some $\E_h>0$.

$(H3)$ Spectral condition for the bundles (exponential trichotomy):

For all  $\th\in\torus_\rho^d$ the spaces in $(H2)$
are approximately hyperbolic for the cocycle
$\gamma(\th)$ (see Definition
~\ref{approximatehyperbolic}). We recall  that this
just entails that the diagonal cocycles have different rates of
growth  and  hyperbolicity constant that satisfy \eqref{growthrates}.

$(H3')$ Since we are dealing with conformally symplectic
systems and are interested in the almost symplectic limit,
we will also assume\footnote{As we will show in Section~\ref{sec:geometry}, the interaction between
conformally symplectic systems and the exponential trichotomy implies
further restrictions which follow from the present assumptions.}\footnote{ Note
that we have used $\lambda$ for the conformal factor and
$\lambda_\sigma$ for the different bounds on  rates. Even if these
are conceptually very different things, we will show that they are
related. This justifies using similar letters.}:
$$
\lambda_- < \lambda \lambda_+< \lambda_c^-\ ,\qquad \lambda_c^-\leq \lambda\leq\lambda_c^+\ .
$$

$(H4)$ We assume that the dimension of the center subspace\footnote{The content of this assumption is that the dimension of
  the center bundle is exactly twice the dimension of the invariant torus.
  As we will show later, the dimension has to be at least twice. }
is $2d$.

$(H5)$ Non--degeneracy:
\footnote{The idea of the condition is that a very explicit $2d\times 2d$ matrix
is invertible. We will formulate it here in detail,
but the main point is that the condition can be verified
with a finite computation on the approximate solution and
the approximate bundles given in $(H1)$, $(H2)$.}

Denote by $J_c$ the operator $J$ restricted to the center
space (we will show in Lemma~\ref{lem:nondegenerate}
that $J_c$ is a non-degenerate matrix).

Let
\beq{N}
N(\th)=(DK(\th)^T DK(\th))^{-1}\ ,
\eeq
$$
P(\th)=DK(\th) N(\th)\ ,
$$
$$
\chi(\th)=DK(\th)^T(J^c)^{-1}\circ K(\th) DK(\th)\ .
$$
Let $M$, $S$ be auxiliary quantities defined as
\beq{M}
M(\th) = [ DK(\th)\ |\  (J^c)^{-1}\circ K(\th)\ DK(\th) N(\th)]
\eeq
and
\beq{torsionW}
S(\th)\equiv P(\th+\omega)^T Df_\mu \circ K(\th) (J^c)^{-1}\circ K(\th)P(\th)
- N(\th+\omega)^T \chi(\th + \omega) N(\th+\omega)\ \lambda\,\Id_d\ .
\eeq

We assume that the following non--degeneracy condition is satisfied, precisely that
the matrix $\S$ defined below is invertible:
\begin{equation}
\label{non-degeneracyW}
\S \equiv
\left(
\begin{array}{cc}
  {\overline S} & {\overline {S(W_b^c)^0}}+\overline{\widetilde A_1^c} \\
  (\lambda-1)\Id_d & \overline{\widetilde A_2^c} \\
 \end{array}%
\right)\ ,\qquad \det\S \ne 0\ ,
\end{equation}
where the bar denotes the average,
$\widetilde A_1^c$, $\widetilde A_2^c$ denote the first $d$ and
the last  $d$ rows of the $2d\times d$ matrix
$\widetilde A^c\equiv[\widetilde A_1^c|\widetilde A_2^c]=M^{-1}\circ T_{\omega} D_{\mu} f_{\mu} \circ K$,
$(W_b^c)^0$ is the solution of $\lambda (W_b^c)^0-(W_b^c)^0\circ T_\omega=-(\widetilde A_2^c)^0$,
where $(\widetilde A_2^c)^0=\widetilde A_2^c-\overline{\widetilde A_2^c}$.

Let $\alpha(\tau)$ be an explicit number
(see the discussion later for the values that come from the proof).
Assume that for some $\delta$, $0< \delta < \rho$, we have
$$
\E_h  \le \E^*_h\ ,\qquad \E \le \delta^{2 \alpha} \E^*\ ,
$$
where
$\E_h^*$, $\E^*$ are
explicit functions given along the proof and depending on the following quantities:
\begin{equation}\label{degeneracylist}
  \begin{split}
&\nu, \tau,
 C_0, \lambda_+, \lambda_-, \lambda_c^+, \lambda_c^-,
 \| \Pi^{s/u/c}_{\th}\|_\rho \\
 &\| DK_0\|_\rho, \| (DK_0^{T} DK_0)^{-1}\|_\rho,
|\S^{-1}|, \max_{ j = 0, 1,2}
\sup_{|\mu - \mu_0 | \le \eta_0} \| D^j f_\mu\|_{\U}\ .
\end{split}
\end{equation}

Then, there exists an exact solution $(K_e,\mu_e)$, such that
$$
f_{\mu_e}\circ K_e-K_e\circ T_\omega=0
$$
with
\beq{Kmu}
\|K_e-K_0\|_{\rho-2\delta}\leq C\E\delta^{-\tau} ,\qquad |\mu_e-\mu_0|\leq C \E\ ,
\eeq
where $C$ is a constant whose explicit expression can be obtained from
the proof and which depends on the same variables as $\E_h^*, \E^*$.

Furthermore, the invariant torus $K_e$ is
hyperbolic in the sense that there exists an invariant
splitting
\[
\T_{K_e(\th)}\M = E_{\th}^s \oplus
 E_{\th}^c \oplus E_{\th}^u\ ,
\]
that satisfies Definition~\ref{def:trichotomy}.

The splitting of the invariant torus is close to
the original one in the sense that, for some constant $C>0$, one has
\beq{proj}
\| \Pi^{s/u/c}_0 -
\Pi^{s/u/c}_f   \|_{\rho - 2 \delta} \le C (\E\delta^{-\tau} + \E_h)
\eeq
(as remarked above, this is equivalent to
the analytic Grassmanian distance).

Moreover, the hyperbolicity constants corresponding to the invariant splitting
of the invariant torus (which we denote by a tilde)
can be taken to be close to those of the approximately
invariant splitting of the approximate invariant torus
assumed to exist in (H1), (H2):
\beqa{lamest}
| \lambda_\pm - \widetilde\lambda_\pm| &\le&
C ( \E\delta^{-\tau}  + \E_h),\nonumber\\
| \lambda^\pm_c  - \widetilde\lambda^{\pm}_c | &\le&
C ( \E \delta^{-\tau}  + \E_h)\ .
\eeqa
\end{theorem}

The proof of Theorem~\ref{whiskered} is postponed to Section~\ref{sec:whiskered},
since we devote Section~\ref{sec:geometry} to discuss some properties stemming
from the geometry of conformally symplectic systems.
Later in this paper, we will present other results. Notably, we will
study the domain of analyticity of the tori for the
small dissipation regime (see Theorem~\ref{thm:domain}).

We also obtain explicit estimates on  the  new hyperbolicity constants
$C_0$, but they are too cumbersome
to state now (see Lemma~\ref{precise} for more detailed
estimates on the new hyperbolicity constants).

\subsection{Some remarks and comments on the statement of
Theorem~\ref{whiskered}}

We collect in this Section some useful comments on the content of Theorem~\ref{whiskered}
and comparisons with other results in the literature.

\medskip

$\bullet$ Note that  the non-degeneracy quantities in
\eqref{degeneracylist} are quantities that can be
estimated just on the approximate solution.
The only global property of  the function needed is
an estimate on $\sup_{|\mu - \mu_0 | \le \eta_0} \| D^j f_\mu\|_{\U}$
and we do not need delicate global properties of the map such as a
global twist condition.

\medskip

$\bullet$ Note that Theorem~\ref{whiskered} is stated without any reference
to an integrable system. We just need an approximate solution of the invariance
equation.

\medskip

$\bullet$ We are not assuming that any of the invariant bundles are trivial (but we will show that the center bundle is trivial as
a consequence of the other hypotheses).

\medskip

$\bullet$ The twist  non-degeneracy condition  $(H5)$
is just that  a very explicit $2d\times 2d$
matrix is non-degenerate. This matrix is formed by the derivatives
of the approximate solution, performing algebraic operations and averages.
It can be computed with a finite number of
computations from the approximate solution.
\medskip

$\bullet$ The above formulation gives a  very transparent proof
of several \emph{``small twist results''}. One can construct
perturbative expansions that satisfy the invariance equation
to arbitrarily high powers of the perturbation  parameter.
At the same time (performing calculations) one can prove
that the twist, hyperbolicity, etc., start to grow like a finite power
of the perturbation.  Then, the theorem will imply the existence of
a solution.

Another application included in this paper is that we will
prove small hyperbolicity assumptions, see Theorem~\ref{thm:domain}.

\medskip
Another non-degeneracy assumption we will need is that
the matrix $M$ introduced in \eqref{M} is invertible if
the initial error is small enough. We will also show
that the iterative procedure maintains the uniform bounds in $M^{-1}$.
In computer assisted proofs, and more explicit treatments, it is
advantageous to obtain precise estimates for $M^{-1}$ at the initial
step.

\medskip

$\bullet$ The hypothesis $(H5)$ is analogue to the Kolmogorov non-degeneracy
condition.We note that if $\lambda = 1$ -- the symplectic case --
then, the condition just becomes $\overline S$ being invertible.
For an integrable system, this is the Kolmogorov non-degeneracy
condition.

\medskip

$\bullet$ The condition $(H5)$ is not a global property of
the map. It is only a numerical condition evaluated on the
approximate solution. It can be readily computed by taking derivatives,
performing algebraic operations and taking averages.

\medskip

$\bullet$ It is possible to use the method of \cite{Moser67}
or the method of \cite{Yoccoz92b,Sevryuk99} to obtain
the result under much
weaker non--degeneracy conditions than $(H5)$ such as R\"ussmann non-degeneracy
conditions.

The proof of this result is particularly transparent taking
advantage of the a-posteriori format which gives very easily the dependence
on parameters.

\medskip

$\bullet$
We note that, thanks to Lemma~\ref{lem:closing}, instead of
the approximate invariance of the splitting included
in $(H2)$, we could  have
assumed that the approximately invariant torus has an invariant splitting.

We have chosen the present formulation to emphasize that all the
hypotheses of Theorem~\ref{whiskered} can be verified from
an approximate solution with just a finite precision computation.

\medskip

$\bullet$ We have treated separately the smallness conditions in the
invariance of the torus $\E$ and the smallness condition in the
invariance of the hyperbolic splitting $\E_h$.  As we will see,
the error in the invariance of the hyperbolic splitting
can be eliminated with a contraction point argument.
Eliminating the error $\E$ requires a Nash-Moser iteration to
beat  the small divisors that appear.

\medskip

$\bullet$ The proof of Theorem~\ref{whiskered} will be based on describing an
iterative process which leads to a very efficient algorithm.
To obtain an algorithm from the proof of Theorem~\ref{whiskered},
one needs to present also descriptions of the discretizations of
the bundles and finite calculations that allow to verify the
hypotheses. These algorithmic details  are presented in \cite{HuguetLS11}
for the symplectic case and they do not need to be modified in
our case.

\medskip

$\bullet$  The error in the hyperbolicity plays a very different role
than the error in the invariance in the iterative process.
We could think of the hyperbolicity as a preconditioner for
the Newton method for the invariance equation. As we will see,
the iterative step has an upper triangular structure. The
error in the invariant splittings
can be eliminated without affecting the embeddings. On the other hand, if
we modify the embedding $K$ and the drift $\mu$, we modify
$Df_\mu \circ K $ and have to correct for the invariant embedding.
This elementary remark will be important for the study of
Lindstedt series in Section~\ref{sec:domain}.

\medskip

$\bullet$ The conformal symplectic properties  of the map imposes many relations between the
properties of the invariant splitting. These will be discussed in Section~\ref{sec:geometry}.

\medskip

$\bullet$
Notice that when $\lambda \ne 1$, the invariant torus is normally hyperbolic
since the center direction, as remarked above, has the conformal factor
$\lambda$ as the multiplier.

This observation allows one to obtain several results, slightly weaker than
Theorem~\ref{whiskered}.

\begin{itemize}
\item[(i)]
Using a-posteriori formulations of the theory of
normally hyperbolic invariant manifolds (\cite{BatesLZ08, CapinskiZ11}), we
obtain from the hypotheses on the approximate invariance and the
approximate invariant splitting that there are smooth invariant tori
for all perturbations (no need to adjust the drift!). Of course,
we do not know that the motion in the manifold will be conjugate
to a rotation.

If we change the drift, using the theory of \cite{Moser66a} we obtain,
under some non-degeneracy conditions that, for
appropriate choices of the drift, the motion on the torus is
conjugate to a rotation. We refer to \cite{CCFL14} for more
details on the argument  and for an application of this strategy to discuss
phase locking and other situations when the motion is not conjugate
to a rotation. Notice that this method produces only finitely
differentiable objects and not analytic ones as the present method.
Also, the algorithms they give rise are very different.

\item[(ii)]
An alternative approach is in~\cite{CanadellHaro}, which deals with
normally hyperbolic tori using the fact that in the stable and unstable
directions we can use an  iterative  method to solve the linearized invariance equation.
For the tangent directions
one needs to adjust parameters to solve the conjugacy equation.
Notice that this method is different from the normally hyperbolic
method, since it produces analytic manifolds but needs to
adjust parameters. This technique leads to very efficient numerical
methods that have been implemented in~\cite{CanadellHaro2}.

\item[(iii)]
Note that the above mentioned approaches do not require
(and do not take advantage of) the conformally symplectic
geometry. This generality is useful for some
models of friction in which the friction does not lead
to a conformally symplectic system. The limit of weak dissipation
in such cases seems a challenging problem.

\item[(iv)]
On the other hand, we note that these methods have estimates that
blow up as $\lambda$ goes to $1$, whereas the method of
this paper leads to a comfortable study of the
small dissipation limit. Indeed, one of the main
results of this paper is the study of the analyticity
domain in the zero dissipation limit, see Section~\ref{sec:domain}.
One of our motivations was precisely the studies in celestial
mechanics where the dissipation is indeed small and the zero dissipation
limit is very relevant.
\end{itemize}

\medskip

$\bullet$ It is important to remark that $\lambda_+$, $\lambda_-$, $\lambda_c^+$, $\lambda_c^-$
appearing in $(H3)$ are only upper bounds. Hence, they are not uniquely defined.
When we consider such values optimal, we obtain the Sacker-Sell spectrum
\cite{SackerS74}. The pairing rule provides relations between the bounds,
but they become equalities for the optimal values.

\section{Some consequences of the geometry}
\label{sec:geometry}

\subsection{Introduction}

In this Section, we present consequences of
the conformally symplectic systems and the trichotomy assumptions
for an (approximately) invariant torus with an (approximately)
invariant splitting.

The geometrically natural arguments (leading to
the sharper results)  happen when the torus is invariant
and the bundle is invariant. The main reason is that we
need to compare vectors and forms in $f(K(\th))$ and in
$K(\th + \omega)$. An alternative to invariance, is
that $\Omega$ is constant.  Of course, our iterative process to improve
approximate solutions, needs to take advantage of the geometry
for the approximate solutions, which are slightly weaker than the
geometrically natural ones.

Hence we will introduce provisionally the hypothesis (HI) below
to be able to carry out geometrically natural arguments.
In Section~\ref{sec:noninvariant}, we show how
to remove this assumption.

(HI) Assume either:
\begin{description}
\item[$\rm{(HI.A)}$] $K$ is an embedding of the torus satisfying
\eqref{invariance} and $E$ is a splitting of the tangent bundle  to
the phase space invariant under the cocycle $Df\circ K (\th)$
under the rotation $\omega$.
\item[$\rm{(HI.B)}$] The phase space is Euclidean and the symplectic
form $\Omega$ is constant.
\end{description}
In many practical applications the case  B) in the alternative
above holds.

Of course, in the iterative step of the KAM theorem, we cannot assume
(HI.A), that the torus is invariant. The assumption (HI.A) is natural
in the development of Lindstedt series that will take place in Section~\ref{sec:domain}.

The removal of (HI.A) in
Section~\ref{sec:noninvariant} will be obtained just by
examining carefully the naturally geometric argument
and adding some extra terms that are controlled by the
invariance error and its derivatives.

\subsection{The results in this Section}
The first result we will present
is the well known pairing
rule (\cite{DettmannM96, WojtkowskiL98}), which relates the
stable/unstable exponential rates (see Section~\ref{sec:georates}).

In Section~\ref{sec:isotropic} we show the isotropic property of invariant tori,
namely that the symplectic form restricted to the torus is zero.

We will also show that, because of the conformally symplectic
structure, we have that the symplectic form restricted to the center
is non-degenerate, see Lemma~\ref{lem:nondegenerate} in Section~\ref{sec:coisotropic}.

A rather remarkable result obtained here
is that the center bundle has to be
trivial, see Section~\ref{sec:triviality}. This solves a question
raised in \cite{FontichLS}, which constructed examples where
the stable and unstable bundles were non-trivial.
The automatic reducibility is discussed in Section~\ref{sec:automatic} and some
consequences in Section~\ref{sec:consequences}. Geometric identities for approximately
invariant tori are presented in Section~\ref{sec:noninvariant}.

\subsection{Geometry and rates}\label{sec:georates}
The conformal symplectic properties of the maps imply constraints on
the rates assumed in $(H3)$. In this Section, we develop two of them:
the pairing rule and the rigidity of rates in the center. These properties
will not be used in the proof of Theorem~\ref{whiskered} and, hence, can be
omitted, but we include them since the method of proof is useful in
other parts of the paper.

\subsubsection{The Pairing rule}\label{sec:pairing}

The paper \cite{DettmannM96} studies the effect on the geometry of
eigenvalues of periodic orbits; the paper \cite{WojtkowskiL98} studies
the effect on the Lyapunov exponents of cocyles. In our case, we want to study
the relation with the optimal rates appearing in $(H3)$, which are known also
as Sacker-Sell spectrum (\cite{SackerS74}). We note that the paper
\cite{WojtkowskiL98}, since it worked for
general cocycles,  did not take advantage of the fact that, for
diffeomorphisms, the factor has to be a constant
when $n\geq 2$ (see the argument
after Definition~\ref{defCS}). Therefore, some of the formulas
in \cite{WojtkowskiL98} can be simplified for the applications
of this paper. We will revisit a more detailed comparison with
these papers in \cite{CallejaCL18b}.

The key observation is that since
\[
\Omega(Df^n(x)u,Df^n(x)v)=\lambda^n \Omega(u,v)\ ,
\]
then, if $|Df^n(x)v|\leq C\lambda_-^n |v|$ for any $n\geq 0$, we should have
\[
|Df^n(x)u|\geq \tilde C\ \lambda^n\ \lambda_-^{-n}\ \Omega(u,v)\ |v|^{-1}
\]
for some positive constant $\tilde C$.
That is, if there is a vector that decreases
exponentially fast, there should be others which grow faster than the rates.

Hence, if there is a vector with Lyapunov multiplier smaller than $\lambda$,
there should be another one with Lyapunov multiplier bigger than ${\lambda_-}^{-1}\lambda$.
By reversing the argument we obtain that the set of Lyapunov multipliers $\{\lambda_i\}_{i=1}^{2d}$
should satisfy the pairing rule
\beq{pairing}
\lambda_i\ \lambda_{i+d}=\lambda
\eeq
(compare with the corresponding formula in \cite{WojtkowskiL98}, which involves
an integral of $\log\lambda$). We remark that our desired result is different than that of \cite{WojtkowskiL98},
since we want to obtain uniform bounds rather than Lyapunov exponents.

In our case, we can obtain uniform bounds on the growth, using
Corollary~\ref{newcorollary} below and other elementary arguments.

\begin{remark}
There are more general arguments in Sacker-Sell theory, relating the edge
of the Sacker-Sell spectrum and the supremum of Lyapunov exponents of all measures
(see \cite{SackerS74}). We will not emphasize those arguments, since
we will not use them.
\end{remark}

We note that for every $x$ and any $u\in E_x^u$, $|u|=1$,
there exists $v\in E_x^s$, $|v|=1$, such that $\Omega(u,v)\neq0$
 (see Corollary~\ref{newcorollary} below).
Using the continuity and compactness
of $\torus^d$ and the spheres in the unit bundle, we obtain
$$
\inf_{x\in K(\torus^d_\rho)}\inf_{u\in E_{K(x)}^u,|u|=1}\inf_{v\in E_{K(x)}^s,|v|=1}\ |\Omega_x(u,v)|\geq\zeta
$$
for some positive constant $\zeta$.
Therefore, given $u\in E_{K(x)}^u$, we can choose $v\in E_{K(x)}^s$ so that
\beq{uniform}
|Df^n(K(x))u|\geq {\tilde C}\zeta(\lambda\lambda_-^{-1})^n|u|\ .
\eeq
Using that the bounds \equ{uniform} are uniform, we see that given $u\in E_{K(x)}^u$,
we can apply them to $Df^{-n}(K(x))u$ and obtain
\[
|Df^{-n}(K(x))u|\leq {\tilde C^{-1}}\zeta^{-1}(\lambda^{-1}\lambda_-)^n|u|\ .
\]
Therefore we obtain that we should have
\[
\lambda^{-1}\lambda_-\leq \lambda_+\ .
\]
By applying a similar argument
to the bounds along the stable direction, we obtain the other
inequality, $\lambda^{-1}\lambda_-\geq \lambda_+$, which leads to
\equ{pairing}, that is $\lambda_-\lambda_+^{-1}= \lambda$ for the
optimal values.

\subsection{Isotropic properties of rotational invariant tori}\label{sec:isotropic}
The isotropic property means that the symplectic form restricted to the invariant torus is zero.

To establish that rotational tori are indeed isotropic, we
note that by the invariance equation \equ{inv} we have
$$
K^\ast f_\mu^\ast\Omega = T_\omega^\ast K^\ast\Omega\ .
$$
Since $f_\mu$ is conformally symplectic, according to \equ{conformallysympmap} we have
\begin{equation}\label{formtransform}
\lambda\ K^\ast\Omega=T_\omega^\ast K^\ast\Omega
\end{equation}
and if  $|\lambda| \ne 1$, we obtain by iterating the
relation \eqref{formtransform} (either in the future or in the past) that
$$
K^\ast\Omega=0\ ,
$$
thus proving that the tori are isotropic.

In the case that $\lambda = 1$ (this is a case that has
been discussed in  \cite{Zehnder75}), it is required that $\omega$ is
non-resonant and that the map is exact. We note that, under
the non-resonant hypothesis, we obtain that $K^*\Omega$
is given by a constant matrix. Moreover, if $\Omega = d \alpha$,
we have $K^*\Omega = K^* d \alpha = d K^* \alpha$. The only exact
form with a constant matrix is $0$. Note that in the case that $|\lambda| \ne 1$
we do not need that the symplectic form is exact, nor that $\omega$ is
nonresonant to conclude that the rotational tori are isotropic.

In Section~\ref{sec:coisotropic}, we will see that approximately invariant tori
are also approximately isotropic. In the $\lambda =1$ case
it requires that $\omega$ is Diophantine.

\medskip

In coordinates, the isotropic property of
the invariant torus, using the matrix $J$
defined in Section~\ref{sec:expressions}, is written as
\begin{equation}
\label{isotropic}
DK^T(\th) J_{K(\th)} DK(\th) = 0\ .
\end{equation}
The equation \eqref{isotropic} can be interpreted geometrically
as saying that any vector in the range of $DK(\th)$ is orthogonal
to any vector in the range of $J_{K(\th)}DK(\th)$.

\subsection{Non degeneracy of the symplectic form restricted
  to the center bundle of a rotational invariant torus}\label{sec:coisotropic}

The following result shows that there are many cases where
the symplectic form $\Omega$ has to vanish. As a corollary, we
will deduce that the symplectic form is non-degenerate
when restricted to the center bundle $E^c_\th$.

\begin{lemma}\label{lem:nondegenerate}
Let $E^s_\th$, $E^c_\th$, $E^u_\th$ be an invariant splitting
around a rotational torus  with growth rates as in $(H3)$. Then,
\begin{equation}\label{orthogonality}
\begin{split}
& \Omega (s, c) = 0 \quad \forall\, s \in  E^s_\th, c \in E^c_\th \\
& \Omega (u, c) = 0 \quad \forall\, u \in  E^u_\th, c \in E^c_\th \\
& \Omega (s_1, s_2) = 0 \quad \forall\, s_1, s_2 \in  E^s_\th\\
& \Omega (u_1, u_2) = 0 \quad \forall\, u_1, u_2 \in  E^u_\th\ .\\
\end{split}
\end{equation}

\end{lemma}

\begin{proof}
Let $s\in E^s_\th$, $c\in E^c_\th$; then, one finds that
$$
\Omega(s,c)=0\ ,
$$
since the following bounds hold for a suitable constant $C$ and for $j\geq 1$:
\beqano
|\Omega(s,c)|&=&\left |{1\over \lambda}\ \Omega(Df_\mu\ s,Df_\mu\ c)\right |\nonumber\\
&=&{1\over \lambda^j}\ \left |\Omega(Df^j_\mu\ s,Df^j_\mu\ c) \right|\nonumber\\
&\leq&C\ \left({{\lambda_-\, \lambda_c^-}\over \lambda}\right)^j\
, \eeqano whose limit tends to zero as $j$ goes to infinity for
$\lambda_-$ and $\lambda_c^-$ as in $(H3)$ and $(H3')$, i.e. using that
$\lambda_-\lambda_c^-<(\lambda_c^-)^2\leq\lambda^2$.

A similar argument
holds for the unstable bundle,
\beqano
|\Omega(u,c)|&=&\left |\lambda\Omega(Df_\mu^{-1}\ u,Df_\mu^{-1}\ c)\right |\nonumber\\
&=&\lambda^j\ \left |\Omega(Df_\mu^{-j}\ u,Df_\mu^{-j}\ c) \right|\nonumber\\
&\leq&C\ (\lambda\, \lambda_+\, \lambda_c^+)^j\ ,
\eeqano
whose limit goes to zero as $j$ tends to infinity under the condition $(H3')$,
i.e. using that $\lambda\lambda_+\lambda_c^+<\lambda\lambda_+\lambda_+<\lambda_c^-\lambda_+
<\lambda\lambda_+<\lambda_c^-\leq\lambda$.

Next we prove the third of \equ{orthogonality}; for any $s_1$, $s_2\in E_\th^s$, we have:
\beqano
|\Omega(s_1,s_2)|&=&\left |{1\over \lambda}\ \Omega(Df_\mu\ s_1,Df_\mu\ s_2)\right |\nonumber\\
&=&{1\over \lambda^j}\ \left |\Omega(Df_\mu^{j}\ s_1,Df_\mu^{j}\ s_2) \right|\nonumber\\
&\leq&C\ \left({{\lambda_-^2}\over \lambda}\right)^j\ ,
\eeqano
which goes to zero for $j\to\infty$ due to $(H3)$ and $(H3')$, since
$\lambda_-^2/\lambda<\lambda_-\lambda_c^-/\lambda<1$.
The fourth equation in \equ{orthogonality} holds under the assumption $\lambda\lambda_+^2<1$, which is
guaranteed by $(H3)$ and recalling that $\lambda<1$. In fact,we have:
\beqano
|\Omega(u_1,u_2)|&=&\left |\lambda\ \Omega(Df_\mu^{-1}\ u_1,Df_\mu^{-1}\ u_2)\right |\nonumber\\
&=&\lambda^j\ \left |\Omega(Df_\mu^{-j}\ u_1,Df_\mu^{-j}\ u_2) \right|\nonumber\\
&\leq&C\ (\lambda\, \lambda_+^2)^j\ .
\eeqano

\end{proof}

As we will see, the above results lead to some automatic non-degeneracy
conclusions which will be important to develop structures on the theorem.

\begin{cor}
\label{cor:nondegenerate}
In the hypotheses of Lemma~\ref{lem:nondegenerate}
we have that $\Omega$ restricted
to $E^c_\th$ is
non-degenerate.
\end{cor}

\begin{proof}
To conclude that $\Omega$ restricted to $E^c_\th$ is non--degenerate, we
observe that if  for some $w\in E^c_\th$, we have that
$$
\Omega(c,w)=0\qquad\forall c\in E^c_\th\ ;
$$
then, using that
$$
\Omega(s,w)=0\ ,\quad \Omega(u,w)=0\ ,\quad \forall s\in E^s_\th\ ,\ \
\forall u\in E^u_\th\ ,
$$
we obtain that $\Omega(v,w)=0$ for any $v\in \T_\th\M$. Therefore, since
$\Omega$ is non-degenerate in the whole space, we conclude that $w=0$.
\end{proof}

\begin{cor}\label{newcorollary}
If $v\in E_x^s$ and for any $u\in E_x^u$ we have that $\Omega(v,u)=0$, then $v=0$.
If $\tilde u\in E_x^u$ and for any $\tilde v\in E_x^s$ we have that
$\Omega(\tilde u,\tilde v)=0$, then $\tilde u=0$.
\end{cor}

\begin{proof}

The proof of Corollary~\ref{newcorollary} is identical with that of
Corollary~\ref{cor:nondegenerate}.  We note that the hypotheses of
Corollary~\ref{newcorollary} and the results of
Lemma~\ref{lem:nondegenerate} imply that $\Omega(v,u)=0$ for any
$u\in \T_xM$ which, by the non-degeneracy of $\Omega$, implies the conclusion
of Corollary~\ref{newcorollary}.
\end{proof}

Corollary~\ref{newcorollary} can be interpreted as saying that some of the matrix
elements giving $\Omega$ are not degenerate. This will be useful later
when we discuss pairing rules for exponents.

\subsection{Triviality of the center bundle}
\label{sec:triviality}

The main goal of this Section is to show that the bundle
$E_\th^c$ based on a rotational invariant torus satisfying
our hypotheses (notably that the dimension of the fibers of the bundle
is $2d$) is trivial in the sense of bundle theory.
That is, we will show that $E_\theta^c$ is isomorphic to
a product bundle (namely, a trivial bundle in the language of
bundle theory).

Furthermore, we show that there is a natural system of
coordinates on $E_\theta^c$, see Lemma~\ref{lem:istrivial}.
In this system of coordinates,
the linearization of the invariance equation
\eqref{inv} restricted to the center space becomes a constant
coefficient equation and, hence, can be solved by using Fourier
methods, see Lemma~\ref{neutral} in Section~\ref{sec:noninvariant}.

Note that the triviality of $E^c_\th$ is in contrast
with the stable and unstable bundles, which can be nontrivial
(see examples in \cite{FontichLS}).
Note also that the proof works when the phase space is a manifold
and it applies a-fortiori for symplectic systems.

\begin{lemma}\label{lem:istrivial}
Assume that $K$ is an approximate solution of
\eqref{inv}. Then, we can find a linear operator
\[
B_\th: \Range( DK(\th))  \rightarrow E^c_\th\ ,
\]
such that the center bundle is given by
\beq{E}
E^c_\th = \{ v + B_\th v :\  v \in \Range( DK(\th))\}\ .
\eeq
\end{lemma}

Notice that \equ{E} shows that $E_\th^c$ is the range under $\Id+B_\th$ of the
tangent bundle of the torus. This shows that $E_\th^c$ is a trivial bundle.

\begin{proof}
We start by remarking that, as it is standard, if we fix a
Riemannian metric $g$, we can identify the two form  $\Omega_x$ with
a linear operator $J(x): \T_x \M \rightarrow \T_x \M$ by requiring
\begin{equation} \label{identification2}
g_x (u, J(x)v) = \Omega_x (u,v) \quad \forall u,v \in \T_x \M\ .
\end{equation}

Of course, the operator $J$ depends on the metric  chosen (we omit the
dependence on the metric from
the notation, unless it can lead to error). It will be advantageous
for us to choose the metric so that the operator $J$ has extra properties.

We will choose a metric $g_x$ such that the spaces $E^c_x, E^s_x, E^u_x$
are orthogonal under $g_x$. A possibly degenerate (i.e., assigning
zero length to non-zero vectors)  metric
can be easily constructed in
coordinate patches.  By adding constructions in different coordinate patches,
we can ensure that the resulting metric is not degenerate.

We denote by  $\tJ_x$  the operator corresponding via
\eqref{identification2} with $\Omega_x$ using the metric
constructed above, which makes the splitting orthogonal.
The properties established in \eqref{orthogonality} imply  that
if we decompose the operator $\tJ_x$ in blocks corresponding to
the decomposition $\T_x \M = E^c_x \oplus E^s_x \oplus E^u_x$, then
we have  the block structure:
\begin{equation}\label{Jblock}
\tJ_x =
\begin{pmatrix}
\tJ^{cc}_x   & 0 & 0 \\
0 & 0 &\tJ^{su}_x \\
0 &  \tJ^{us}_x  & 0 \\
\end{pmatrix}\ .
\end{equation}
The inverse of the operator $\tJ_x$ also has the same structure
as \eqref{Jblock}.

The key of the construction is that the metric $g$ is globally defined
in a neighborhood of the approximately invariant torus and, therefore,
so are the operators $J$ and $J^{-1}$.

We also note that we established in Section~\ref{sec:coisotropic}
that the form restricted to the tangent space vanishes for
invariant tori (we will see  that it is small for
approximately invariant tori in many cases).
Thus, we obtain that
\[
\Omega_{K(\th)}\left( \Range(DK(\th) ), \Range( DK(\th))\right)=
( \tJ^{cc}_{K(\th)}
  \Range( DK(\th)), \Range( DK(\th)))
\]
is very small (identically zero for exactly invariant tori).
In particular, we obtain that the operator
$\tJ^{cc}_{K(\th)}$ maps $\Range( DK(\th))$ into a linearly independent space.

Using that the dimension of the center manifold is $2d$ as in assumption $(H4)$,
we obtain that:
\[
\Range ( DK(\th))  \oplus \tJ_{K(\th)}^{cc}
\Range ( DK(\th))  = E^c_{\th}\ .
\]
Since $\tJ_{K(\th)}^{cc}$ is a linear operator, we obtain that the center bundle
can be expressed as in \equ{E}.
\end{proof}

\subsection{Automatic reducibility}\label{sec:automatic}
A key ingredient in the proof of our main result on whiskered tori is the so--called \emph{automatic reducibility}:
in a neighborhood of an invariant torus, one can construct a change of coordinates such that the linearization of the
invariance equation \equ{inv} becomes a constant coefficient equation.

This technique is presented in full detail in \cite{CallejaCL13a} for conformally symplectic systems
(see \cite{LlaveGJV05} for symplectic systems), but we will present the
details again. An important reason is that, by examing the proof
carefully we will discover the surprising global result that the
center bundle has to be trivial in the sense of bundle theory.

It will be important to note that  there is also
a version of approximate reducibility
when the torus is only approximately invariant,
see Section~\ref{sec:noninvariant}.  The proof of
the results in Section~\ref{sec:noninvariant} will be based on
walking through the arguments in this Section and checking how
they are affected by the error in the invariance equation.

We will assume that the tangent space of $\M$ at $K(\th)$, say $\T_{K(\th)}\M$ with $\th\in\torus^d$,
admits an invariant splitting as
$$
\T_{K(\th)} \M=E_\th^s \oplus E_\th^c \oplus E_\th^u\ .
$$
Taking the derivative of \equ{inv} we obtain
\beq{derinv}
Df_\mu\circ K(\th)\, DK(\th) - DK\circ T_\omega(\th) =0\ .
\eeq

This implies  that the range of $DK(\th)$ is contained in $E^c_\th$.

Let $\Omega_{K(\th)}^c$ denote the symplectic form
$\Omega$ restricted to $E_\th^c$ with
\[
\Omega_{K(\th)}^c (u,v)=(u,J^c_0\circ K(\th) v)\ , \quad \forall \, u, v\in E_\th^c\ ,
\]
where $J^c_0$ is the $2d\times 2d$ matrix representing $\Omega_{K(\th)}^c$ on the
center space. Let $J^c$ be the $2n\times 2n$ matrix of the embeddings of the center space into
the ambient space.

As indicated above, we have that $\Range(DK(\th)) \subset
E^c_\th$. Hence, we can write \eqref{isotropic} as
\begin{equation}
\label{isotropiccenter}
DK^T(\th) J^c\circ K(\th)  DK(\th) = 0\ .
\end{equation}

Let us introduce the $2d \times 2d$ matrix $M(\th)$ on
$E_\th^c$, obtained juxtaposing the two matrices $DK(\th)$,
$(J^c)^{-1}\circ K(\th)\ DK(\th) N(\th)$:
\begin{equation}\label{Mdefined}
M(\th) = [ DK(\th)\ |\  (J^c)^{-1}\circ K(\th)\ DK(\th) N(\th)]\ ,
\end{equation}
where we have introduced the normalization factor $N$ as in \equ{N}.
For typographic reasons, we will write
\[
v(\th) = (J^c)^{-1}\circ
K(\th)\ DK(\th) N(\th)\ .
\]
Note that, because of \eqref{isotropiccenter} we have that the
range of $M$ has dimension $2d$ and, due to our assumption on
the dimension of the center, we obtain that
\begin{equation}\label{rangeM}
\Range(M(\th)) = E^c_\th\ .
\end{equation}
Because of \eqref{rangeM}, we know that there exists a matrix $\B(\th)$ such that
\beq{red}
Df_\mu \circ K(\th) M(\th) = M(\th +\omega)\ \B(\th)\ ,
\eeq
where $\B(\th)$ is required to be upper triangular with constant matrices on the diagonal.

The goal now is to identify the matrix $\B$.
We observe that  \equ{derinv} identifies the first column of
$\B$ to be $\begin{pmatrix} \Id_d  \\  0 \end{pmatrix}$.

To identify the second column of $\B(\th)$, by \eqref{rangeM}, we
know that
\begin{equation}\label{secondcolumn}
Df_\mu\circ K(\th) \, v(\th) = DK(\th+\omega) S(\th) +  v(\th + \omega) U(\th)\ ,
\end{equation}
for some function $U=U(\th)$ that we compute as follows.
According to \cite{CallejaCL13a},
we multiply \eqref{secondcolumn}  on the right by
$DK^T(\th + \omega) J^c\circ K(\th+\omega)$.
Using \eqref{isotropiccenter}, we obtain
\begin{equation} \label{identities}
\begin{split}
DK^T(\th + \omega ) J^c\circ K(\th +\omega)  \,
&Df_\mu\circ K(\th) \, v(\th)=\\
& = DK^T(\th + \omega ) J^c\circ K(\th +\omega)(J^c(\th + \omega))^{-1}  DK(\th+\omega) N(\th + \omega)  U(\th) \\
& = DK^T(\th + \omega ) DK(\th+\omega) N(\th + \omega)  U(\th)  \\
& = U(\th)\ .
\end{split}
\end{equation}

Working on the other side, using the conformally symplectic property
\eqref{conformallysym} and the invariance property of
the center foliation, we obtain:
\[
Df^T_\mu(x) J_{f(x)}^c Df_\mu(x)  = \lambda J_{f(x)}^c\ .
\]
Therefore,  $J_{f(x)}^c  Df_\mu(x) (J_{f(x)}^c)^{-1}  = \lambda Df_\mu^{-T}(x)$.

Hence, we see that the left hand side of \eqref{identities} can be computed as
\[
\begin{split}
DK^T(\th + \omega ) & J^c\circ K(\th +\omega)  \,
Df_\mu\circ K(\th) \, (J^c)^{-1}\circ K(\th) DK(\th) N(\th)
\\
& =\lambda  DK^T(\th + \omega ) Df_\mu^{-T} \circ K(\th) DK(\th) N(\th)  \\
& = \lambda DK^T(\th) DK(\th) N(\th)\ ,
\end{split}
\]
where we have used \equ{derinv}.

Therefore, we conclude that $U(\th) = \lambda\Id_d$.

The matrix $S$ can be  computed  in similar way (it just
suffices to multiply in the right to compute the projections):
it does not require any change from the calculations in
\cite{CallejaCL13a}. The result is given by \equ{torsionW}.

In conclusion, we can  write \equ{red} as
\beq{red2}
Df_\mu \circ K(\th) M(\th) = M(\th +\omega)
\begin{pmatrix} \Id_d & S(\th)\\ 0&\lambda\Id_d \end{pmatrix}\ .
\eeq

We note that the average of the matrix $S(\th)$ computed here is
the matrix $S$ appearing in $(H2)$ in Theorem~\ref{whiskered}.
Hypothesis $(H2)$ is just that the average of the matrix $S$
-- which is a $d\times d$ matrix -- is invertible. Again, we
emphasize that this is a condition that is computed
out of the approximate solution taking derivatives, performing
algebraic operations and taking averages.

\subsection{Consequences of automatic reducibility}\label{sec:consequences}
In Section~\ref{sec:automatic} we showed that the preservation of the geometric structure yields
that we can find a matrix $M(\th)$ in such a way that
$$
M^{-1}(\th+\omega) \Pi_{\th + \omega}^u Df_\mu\circ K(\th)\ \Pi_{\th}^u\ M(\th)=
\begin{pmatrix}
\Id_d  & S(\th) \\
0 & \lambda\Id_d\\
\end{pmatrix}\ .
$$
This shows that we can choose $\lambda^-_c$, $\lambda^+_c$ as close as desired to $|\lambda|$
(at the price of choosing an appropriate proportionality constant).

For some $\lambda$'s it is possible to do a further linear change of variables
$A(x)=\begin{pmatrix}
\Id_d & B(\th) \\
0 & \Id_d\\
\end{pmatrix}$ in such a way that the matrix is even simpler. Computing
$$
\begin{pmatrix}
\Id_d & -B(\th+\omega) \\
0 & \Id_d\\
\end{pmatrix}
\begin{pmatrix}
\Id_d & S(\th) \\
0 & \lambda \Id_d\\
\end{pmatrix}
\begin{pmatrix}
\Id_d & B(\th) \\
0 & \Id_d\\
\end{pmatrix}=
\begin{pmatrix}
\Id_d & S(\th)-\lambda B(\th+\omega)+B(\th) \\
0 & \lambda \Id_d\\
\end{pmatrix}\ ,
$$
one is led to solve the following equation for $B$ given $S$:
\beq{reduction}
S(\th)-\lambda B(\th+\omega)+B(\th)=0\ .
\eeq
The equation \equ{reduction} can be solved when $|\lambda|\not=1$ or, for $\lambda\in\complex$,
when $\lambda$ is Diophantine with respect to $\omega$.
In such a case, we can reduce the cocycle derivative to
$\begin{pmatrix}
\Id_d & 0 \\
0 & \lambda\Id_d\\
\end{pmatrix}$
and, hence, we can take $\lambda^+_c=\lambda^-_c=|\lambda|$.

Nevertheless, when $\lambda$ is close to one (or a root of unity), the $B$ appearing
in the last change of variables may be very large.

This means that if we take $\lambda^+_c=\lambda^-_c=|\lambda|$, in $(H3)$ we can take
a very large constant. Geometrically,  this means
that the center direction (which is a weak center direction)
and the tangent span the symplectically complement to the tangent bundle.
We can obtain the stable direction by taking the simplectic conjugate and
add to it vectors on the stable direction.

The limit of $\lambda$ close to one appears naturally in many physical problems
and will be considered in great detail in Section~\ref{sec:domain} (see also \cite{CCLdomain}).
We note that this is a singular
limit because some part of the normal hyperbolicity is lost. It can be
controlled precisely because the geometry forces that this loss of
hyperbolicity is done in a very specific way.

\subsection{Geometric properties for approximately invariant tori}\label{sec:noninvariant}

Of course, in the iterative procedure, we will not be dealing
with invariant tori but with approximately invariant tori.
Hence, it will be important for us to show
that the geometric identities we developed for invariant tori
-- notably the automatic reducibility -- hold approximately.

The main result of this Section is to show that indeed, this is
the case, see Lemma~\ref{approximate-reducibility}.
The reason is that, to obtain the main equation \eqref{red}, we just took derivatives
of the invariance equation and applied algebraic transformations.
Hence, if the invariance equation holds up to an error, we obtain that
\eqref{red} will hold up to errors which can be estimated by derivatives
of the error in the invariance equation. A subtle point in the derivation
is the use of the isotropic properties of the torus. We will also show
that if the torus is approximately invariant, then it has
to be approximately isotropic (with quantitative bounds).

As a preliminary result, we recall the following classical lemma which
gives the solution of a cohomological equation and which will be needed
in the proof of Lemma~\ref{approximate-reducibility}.

\begin{lemma}\label{neutral}
Let $\lambda \in [A_0, A_0^{-1}]$ for some
$0 < A_0< 1$ and let $\omega\in \D_d(\nu,\tau)$.

Consider a cohomological equation of the form
\begin{equation}
w(\varphi + \omega) -\lambda w(\varphi) = \eta (\varphi)
\label{difference1}
\end{equation}
for some functions $w$ and $\eta$ with $\eta \in \A_\rho$, $\rho>0$, and with zero average:
$$
\int_{\torus^d} \eta(\th)\, d\th =0\ .
$$

Then, there is one and only one solution of \eqref{difference1}
with zero average. Moreover, if $\varphi \in \A_{\rho-\delta}$ for some
$0<\delta<\rho$, then we have
$$
\|\varphi\|_{\rho-\delta} \le C\, \nu\, \delta^{-\tau} \|\eta\|_\rho\ ,
$$
where $C$ is a constant that depends on $A_0$ and the dimension of the
space, but it is uniform in $\lambda$ and it is independent of the
Diophantine constant $\nu$.
\end{lemma}

The proof of Lemma~\ref{neutral} can be found, e.g., in \cite{CallejaCL13a},
see also \cite{Russmann75,Russmann76b,Russmann76a}.

\begin{lemma}\label{approximate-reducibility}
Consider an approximately invariant torus,
satisfying \eqref{approximately-invariant}
where $f_{\mu_0}$ is a family of conformally symplectic maps.

Assume that  the cocycle
over $T_\omega$ given by $\gamma(\th) = Df_{\mu_0} \circ K(\th)$
admits an invariant splitting
which is hyperbolic and whose center dimension is $2d$ dimensional.

Assume furthermore that for some constant $C>0$:
\begin{equation}
\label{inductiveass}
\begin{split}
& \| DK\|_\rho, \| N\|_\rho , \| (J^c \circ K )^{-1} \|_\rho    \le C\ ,\\
& C \|e\|_\rho \delta^{-1} < 1\ .
\end{split}
\end{equation}

Then, defining the matrix $M$ as in \eqref{M}, we have
in the center direction
\begin{equation}
\label{approximatereducibility}
Df_{\mu_0}\circ K(\th) M(\th) = M(\th + \omega)
\begin{pmatrix} \Id_d & S(\th)\\ 0&\lambda\Id_d \end{pmatrix} + \E_R\ ,
\end{equation}
where
\begin{equation}
\label{approxredestimates}
|| \E_R\|_{\rho -\delta} \le C \delta^{-1} \|e\|_\rho\ .
\end{equation}
\end{lemma}

\begin{proof}
The proof is basically walking though the proof  of the
invariant case.

The first step is to study how the approximate invariance
modifies the approximately invariance properties.

We note that \eqref{conformallysym}
gives  that, defining $a(\th) = DK^T(\th) J_{K(\th)}  DK(\th) $,
we have:
\begin{equation}\label{errorisotropy}
\begin{split}
a(\th + \omega) - \lambda a(\th)  & =
DK^T(\th+\omega) J_{K(\th + \omega)}  DK(\th + \omega)
- \lambda  DK^T(\th) J_{K(\th)}  DK(\th) \\
& = \left( DK^T(\th)Df_{\mu_0}^T\circ K(\th) - De^T(\th) \right)
 J_{K(\th + \omega)}
\left(  Df_{\mu_0}\circ K(\th) \, DK(\th) - De(\th) \right)\\
&\ \ \ - \lambda  DK^T(\th) J_{K(\th)}  DK(\th) \\
& = \left( DK^T(\th)Df_{\mu_0}^T\circ K(\th) - De^T(\th) \right)
\left[J_{f_{\mu_0}(K(\th))}  +
(J_{K(\th + \omega)} -J_{f_{\mu_0}(K(\th)}) \right]\\
&\ \ \ \left(  Df_{\mu_0}\circ K(\th) \, DK(\th ) - De(\th)\right)
- \lambda  DK^T(\th) J_{K(\th)}  DK(\th)\\
&= DK^T(\th)Df_{\mu_0}^T\circ K(\th)
J_{f_{\mu_0}(K(\th)} Df_{\mu_0}\circ K(\th) DK(\th)\\
&\ \ \ - \lambda  DK^T(\th) J_{K(\th)}  DK(\th) + e_I(\th)  \\
&= e_I(\th)\ ,
\end{split}
\end{equation}
where the expression for $e_I$ is just products of
derivatives of the invariance equation (with other terms).

Hence, we can bound $e_I$ by using the Cauchy estimates for $De$ and
the smallness assumptions and obtain
$$
\|e_I\|_{\rho  -\delta} \le C \delta^{-1} \| e \|_\rho\ .
$$
We note that \eqref{errorisotropy} shows
that $a$ satisfies a cohomology equation of the form considered
in Lemma~\ref{neutral}. Therefore, we can obtain estimates on $a$ as
$\|a\|_{\rho -2 \delta} \le C \delta^{-\tau}\ \nu\  \| e \|_\rho$.

Again, we can interpret the estimates on $a$ as approximate
orthogonality relations.

We now walk through the calculations used in the computation
leading to \eqref{red2}.

The first column in \eqref{red2} is just the derivative of
the invariance equation and we can use Cauchy estimates to get the
estimates \eqref{approxredestimates} for the first column.

To study the second column, we see that \eqref{secondcolumn} is
still true, since it only depends on the property that $\Range( DK(\th +\omega)),
\Range( v(\th + \omega))$ span $E^c_{\th + \omega}$, which is an easy
consequence of the approximate orthogonality.

As before, we multiply \eqref{secondcolumn} by
$DK^T(\th + \omega ) J^c\circ K(\th +\omega) $. We can follow the
calculations used in \eqref{identities} adding and subtracting the
terms that we have estimated.
\end{proof}

\section{Proof of Theorem~\ref{whiskered}}\label{sec:whiskered}

We now proceed to the proof of Theorem~\ref{whiskered}.
As standard in KAM theorem in an a-posteriori format,
 the proof can be divided into  two parts. The first one
is an algorithm that, given an approximate solution, produces a
much more approximate one (in a slightly weaker sense of approximation)
and with slightly worse quality properties.
In a second part, we show
that, if we start with a sufficiently small error, we can
repeat the procedure indefinitely and that the solution indeed converges
in some appropriate sense, and that the limit inherits properties
such as the hyperbolicity and twist. See
the following Proposition~\ref{pro:iterative} in Section~\ref{sec:step}.

The iterative algorithm will be discussed in
Section~\ref{sec:step}. At the end of the step,
the error will be roughly  the square of the original error,
but measured in a norm corresponding to functions
in a smaller domain than that of the original approximation
 (there is also a factor depending
on the loss of domain). This  iterative
algorithm  will be affected by condition numbers
(hyperbolicity properties, twist, etc.) and we need to
estimate how do they deteriorate.

We note that the step we will discuss in this paper will be
numerically very efficient. It does not require that the system is close
to integrable, it only requires to handle functions of
the dimension of the torus, the storage requirement is small and
the operation count is small.

We anticipate that  to be able to carry the step and obtain estimates, we will
need to introduce \emph{inductive assumptions}. One -- very standard in KAM
theory -- requires that $\Delta$, the correction to $K$, is small
enough in its domain so that the range of $( K + \Delta) $  is
well inside the domain of $f$ (so that we can define
$f\circ(K + \Delta)$ and study Taylor expansions in $\Delta$).  In our case, we
will also need another inductive assumption that guarantees that the
hyperbolicity constants are still bounded. As it is well known,
once we fix the domain loss, the first inductive assumption can
be guaranteed by the requirement that the error is small enough
(so that the correction $\Delta$ is small enough). As for
the assumption on the uniformity of the hyperbolicity, using
Lemma~\ref{general}, it will amount to the block diagonal cocycles
remaining in a neighborhood.

The iterative process is discussed in
Section~\ref{sec:convergence}.  The main, well known, idea is
that if we fix a sequence of domain losses that goes to zero
not too fast (e.g. exponentially fast), if
the original error of the invariance is small enough,
the error of the invariance decreases very fast in the iterative step so
that the ranges of the $K$ do not get close to the boundary of
the domain of $f$ and the block diagonal cocycles do not move out
of the neighborhood specified in Lemma~\ref{general}.

\subsection{Results on the iterative step}
\label{sec:step}

The main result of this Section is Proposition~\ref{pro:iterative}
that specifies how, given an approximate solution with some non-degeneracy
properties (if some quantitative assumptions
are satisfied),  we produce a more approximate one with only slightly worse
non-degeneracy assumptions.  The quantitative assumptions, that allow
to perform the step, are standardly called \emph{inductive assumptions},
since we will use an inductive argument to show that they can
be satisfied for all steps of
the iteration.

\begin{proposition}\label{pro:iterative}
Let $\omega\in \D_d(\nu,\tau)$, $d\leq n$, and let $f_\mu:\M\rightarrow\M$, $\mu\in\real^d$,
be a family of real-analytic, conformally symplectic maps as in Theorem~\ref{whiskered} with
$0<\lambda<1$.

Let $(K,\mu)$, $K :\torus^d \to \M$, $K\in\A_\rho$, be an approximate solution of the invariance equation
\eqref{invariance}:
\beq{invw}
f_\mu\circ K(\th)-K\circ T_\omega(\th)=e(\th)
\eeq
for some function $e=e(\th)$. Denote
$\E = \| e\|_\rho$.

Let $E_\th^s\oplus E_\th^c \oplus E_\th^u$ be an approximately invariant,
hyperbolic splitting
based on $K$.
Denote by $\E_h$ the quantity appearing on \eqref{approximatenorm}.

Assume that $(K,\mu, E_\th^s,E_\th^c, E_\th^s)$ satisfy assumptions $(H2)$-$(H3)$-$(H3')$-$(H4)$-$(H5)$ of Theorem~\ref{whiskered}.

We will assume that $\E, \E_h$ are sufficiently small depending on
the quantities in \eqref{degeneracylist} and on $\eta$ with $\eta$ as in $(H1)$.
The constants $C$ denote expressions
depending only the quantities in \eqref{degeneracylist}, and the formulas for
$C$ will be made explicit along the proof.

Then, we have the following results.

1)
There exists an exact invariant splitting
based on $K$, say
 $\tE^{s/c/u}_\th$.
Denote by $\tilde \gamma_\th^{\sigma, \sigma}$  the projections of
$Df_\mu\circ K(\th)$ corresponding to the invariant splitting.
Then, we have:
\beqa{Eh}
\dist_\rho(E^{s/c/u}_\th,\tE^{s/c/u}_\th)&\leq& C\E_h\ ,\nonumber\\
\|\gamma_\th^{\sigma, \sigma} - \tilde \gamma_\th^{\sigma, \sigma}\|_\rho  &\le& C \E_h\ .
\eeqa

2) Assume that $\delta$ is such that
\begin{equation}
\label{inductive}
C \delta^{-\tau} \|e\|_\rho + \dist( K(\torus^d_\rho)), \complex^{2n} \backslash\D) > \eta/2\ .
\end{equation}

Then, we have that $K'=K+\Delta$, $\mu'=\mu+\beta$ for suitable corrections $\Delta$, $\beta$, satisfy
\[
f_{\mu'}\circ K'(\th)-K'\circ T_\omega(\th)=e'(\th)
\]
with
$$
\|e'\|_{\rho-\delta}\leq C\ \delta^{-2\tau}\ \|e\|_\rho^2\ .
$$
Moreover, the corrections can be bounded as
\begin{equation} \label{correctionbounds}
\begin{split}
&\| \Delta\|_{\rho - \delta} \le C \delta^{- \tau} \E \\
&\| D \Delta\|_{\rho - \delta} \le C \delta^{-1 - \tau} \E \\
& |\beta | \le C \E\ .
\end{split}
\end{equation}

3) Furthermore, the splitting $\tE^{s/c/u}_\th$
is approximately invariant for $Df_{\mu + \beta} \circ ( K + \Delta)$.

3.1) The error in the change of the invariance is smaller
than $C \delta^{-\tau}\E$.

3.2)  The block diagonal cocycles
corresponding to $Df_{\mu + \beta} \circ ( K + \Delta)$
(which we denote by $\hat\gamma_\th^{\sigma, \sigma}$) satisfy
\[
\| \hat\gamma_\th^{\sigma, \sigma}  -  \tilde\gamma_\th^{\sigma, \sigma}  \|_{\rho -\delta}
\le C \delta^{-\tau}\| e\|_\rho + \E_h  \ .
\]
\end{proposition}

\subsubsection{Overview of the argument}
\label{sec:overview}
We look for a correction $(\Delta,\beta)$, such that
$K'=K+\Delta$, $\mu'=\mu+\beta$ satisfy \equ{invw} with an error quadratically smaller.

Expanding the composition to first
order in $\Delta, \beta$ we obtain:
\beqano
f_{\mu'}\circ K'(\th)-K'(\th+\omega) &= & f_\mu\circ K(\th)+Df_\mu\circ K(\th)\,\Delta(\th)\\
&&  +D_\mu f_\mu \circ K(\th)\,\beta\nonumber\\
&&-K(\th+\omega)-\Delta(\th+\omega)\nonumber\\
&& +O(\|\Delta\|^2)+O(|\beta|^2)\ .
\eeqano

Taking into account \equ{invw}, the new error is quadratically smaller if $(\Delta,\beta)$ satisfies
\beq{EW}
Df_\mu\circ K(\th)\ \Delta(\th)+D_\mu f_\mu\circ K(\th)\ \beta-\Delta(\th+\omega)  = -e(\th)
\eeq
or, more generally, if
\beq{quasi-EW}
Df_\mu\circ K(\th)\ \Delta(\th)+D_\mu f_\mu\circ K(\th)\ \beta-\Delta(\th+\omega)  +e(\th) = O( \| e\|_\rho^2) \ .
\eeq

Finding corrections by solving
\eqref{EW} can be thought  of as an infinite dimensional version of Newton method.  The small modification \eqref{quasi-EW} is called a quasi-Newton method,
since as we will see, the errors are also reduced quadratically.

To establish part 1) in Proposition~\ref{pro:iterative},
we just invoke the Lemma~\ref{lem:closing} to
change the assumed approximately invariant splitting into an exactly invariant
splitting.

We project \equ{EW} on the hyperbolic and center spaces,
using the invariant splitting \equ{splitting}.

Due to the invariance of the splitting, we have that
$$
\Pi_{\th+\omega}^{s/c/u}\Big(Df_\mu\circ K(\th)\ \Delta(\th)\Big)=
Df_\mu\circ K(\th)\ \Delta^{s/c/u}(\th)\ ,
$$
where $\Delta^{s/c/u}(\th)\equiv \Pi_{\th+\omega}^{s/c/u}\Delta(\th)$.
 Therefore, \equ{EW} is equivalent to the three equations:
\beq{WB}
Df_\mu\circ K(\th)\ \Delta^{s/c/u}(\th)+
\Pi_{\th+\omega}^{s/c/u} D_\mu f_\mu\circ K(\th)\beta-\Delta^{s/c/u}(\th+\omega)
=-e^{s/c/u}(\th)\ ,
\eeq
where we have defined $e^{s/c/u}(\th)\equiv\Pi_{\th+\omega}^{s/c/u} e(\th)$.

Note that in the three linear equations \equ{WB}, we have four unknowns given by
$\Delta^s$, $\Delta^c$, $\Delta^u$ and $\beta$, which we solve using a substitution method.
We will first  solve (approximately)
the equation for $\Delta^c$, which will determine both $\Delta^c$ and $\beta$.

Then, we will solve the equations for $\Delta^s$, $\Delta^u$.

It is convenient for us to start by solving the equation in
the center space since it is the equation that allows to determine
the $\beta$, which also enters in the equations in the stable/unstable
directions. The alternative of solving the stable/unstable equations
with a floating parameter seems more cumbersome.

The equation in the center can be solved approximately,
using the automatic reducibility  established
 in Section~\ref{sec:inv}.  Note that the automatic reducibility
depends on the geometry. We will also use a
non--degeneracy condition (as in $(H5)$); from
the analysis point of view, it is the most delicate equation since it
involves small divisors (here we use the assumption  that the frequency
is Diophantine) and it entails a loss of domain. In contrast, the
equations along the stable/unstable directions can be solved by soft
methods (iteration and contraction) and do not involve any loss of
domain.

Along the argument, we will make some side remarks about an efficient
numerical implementation, which just need
to implement the correction step. Of course, to be convincing
one needs to monitor also the condition numbers (which we make
explicit)  to ensure
that the numerical solutions produced correspond to the true ones.
We anticipate that all the operations required are algebraic operations
on the approximate solutions, taking derivatives, shifting and solving
cohomology equations. If we discretize the parameterization with $N$ terms,
a Newton step requires $O(N \ln(N))$ operations,
either in a grid discretization or in a Fourier discretization.
Of course, one can generate a grid from a Fourier series by using FFT. The procedure gets
quadratic convergence, but does not need to store (much less solve)
an $N \times N$ matrix. Indeed, the storage required  only $O(N)$.

Once we have the estimates on the correction, the nonlinear
estimates for the error can be obtained by elementary
methods such as adding and subtracting terms and applying Taylor's theorem
to first order (the corrections have been chosen precisely to
cancel out the first order approximation).

\subsubsection{Approximate solution of the linearized
invariance equation in the center space}
\label{sec:centerspace}

As for the center subspace,
we will  construct an approximate solution of the Newton equation in the center direction. That is, we will construct a function that solves
the projection of the linearized equation up to
a quadratically small error (which does not affect the
quadratic convergence of the method).
The construction of this approximate solution, which will
take the rest of the Section, is somewhat subtle, since it
requires taking advantage of some geometric properties
and of the Diophantine properties of the frequency.
This solution will also involve a loss of domain
and we will obtain estimates for the correction only in
a slightly smaller domain than the domain of the error.

The linearized
equation \equ{EW} on the center subspace is
\beq{deltac}
Df_\mu\circ K(\th)\ \Delta^c(\th)+\Pi_{\th+\omega}^c D_\mu f_\mu \circ K(\th)\beta-\Delta^c(\th+\omega)=-e^c(\th)\ .
\eeq

We will take advantage of the geometry, see Section~\ref{sec:noninvariant},
to find an explicit linear change of variables
that approximately reduces   the equation in \eqref{deltac} to
constant coefficients difference equations.   These equations can
be solved using Fourier coefficients (but they need the Diophantine
conditions on the frequency). Let us introduce $W^c$ such that we can write $\Delta^c$ as
$$
\Delta^c=M\ W^c
$$
with $M$ as in \equ{M} and satisfying \equ{approximatereducibility}.

Using \equ{deltac} and \equ{approximatereducibility}, we obtain
that the Newton equation projected in the center is
equivalent to:
\[
\begin{split}
M(\th+\omega)&\left(
\begin{array}{cc}
  \Id_d & S(\th) \\
  0 &
\lambda\Id_d \\
 \end{array}%
\right)\ W^c(\th)-M(\th+\omega)\ W^c(\th+\omega) \\
&+ \E_R(\th) W^c(\th)+
\Pi_{\th+\omega}^c D_\mu f_\mu\circ K(\th)\beta=-e^c(\th)\ .
\end{split}
\]

Since the above equation is hard to solve, we argue heuristically
that the term $\E_R$ is comparable to $\E$ because of
\eqref{approxredestimates}, hence the term $\E_R W^c$ is second order.
Hence, we omit it and consider the following equation
\eqref{quasi-newton}. As we will see, the equation
\eqref{quasi-newton} is readily solvable, admits tame estimates
and, indeed,  omitting the term $\E_R W^c$ does not change the fact that the
error remaining  after the iterative step is quadratic in the original error
(we will obtain estimates for the $W^c$ and we have estimates for
$\E_R$ in \eqref{approxredestimates}):
\begin{equation} \label{quasi-newton}
M(\th+\omega)\left(%
\begin{array}{cc}
  \Id_d & S(\th) \\
  0 & \lambda\Id_d \\
 \end{array}%
\right)\ W^c(\th)-M(\th+\omega)\ W^c(\th+\omega)+
\Pi_{\th+\omega}^c D_\mu f_\mu\circ K(\th)\beta=-e^c(\th)\ .
\end{equation}
Multiplying \eqref{quasi-newton} on the left by $M^{-1}(\th + \omega)$,
we obtain:
\beq{CW}
\left(%
\begin{array}{cc}
  \Id_d & S(\th) \\
  0 & \lambda\Id_d \\
 \end{array}%
\right)W^c(\th)-W^c\circ T_\omega(\th)=-\widetilde e^c(\th)-\widetilde A^c(\th)\beta\ ,
\eeq
where $\widetilde e^c(\th)\equiv M^{-1}\circ T_\omega(\th) e^c(\th)$,
$\widetilde A^c(\th)\equiv M^{-1}\circ T_\omega(\th)\
\Pi_{\th+\omega}^c D_\mu f_\mu\circ K(\th)$.

Writing \equ{CW} in components we obtain
\beqa{FW}
W_1^c(\th)-W_1^c\circ T_\omega(\th) &=& -S(\th)\, W_2^c(\th)-\widetilde e_1^c(\th)
-\widetilde A^c_1(\th)\beta\nonumber\\
\lambda W_2^c(\th)-W_2^c\circ T_\omega(\th) &=& -\widetilde e_2^c(\th)-\widetilde A^c_2(\th)\beta\ ,
\eeqa
where $\widetilde A^c\equiv [\widetilde A^c_1|\widetilde A^c_2]$ with $\widetilde A^c_1$, $\widetilde A^c_2$
denoting the first $d$ and the last $d$ rows of the $2d\times d$ matrix $\widetilde A^c$.

Denoting by $\overline{W^c}$ the average of $W^c$ and setting $(W^c)^0\equiv W^c-\overline{W^c}$, we obtain
\beqa{DW}
\noaverage{W_1^c}(\th) -
\noaverage{W_1^c}\circ T_\omega(\th)
 &=& - \noaverage{S W_2^c}(\th)
 -\noaverage { \widetilde e_1^c}(\th) - \noaverage{\widetilde A_1^c}(\th) \beta\nonumber\\
\lambda \noaverage{W_2^c}(\th)  -  \noaverage{W_2^c} \circ T_\omega(\th)
&=& - \noaverage{\widetilde e_2^c}(\th) - \noaverage{\widetilde A_2^c}(\th)\beta\ .
\eeqa
Since $\noaverage{W_2^c}$ is an affine function of $\beta$, we can write
$\noaverage{W_2^c}=\noaverage{W_a^c}+\beta\noaverage{W_b^c}$ for some
functions $W_a^c$, $W_b^c$. Therefore, the second equation in \equ{DW} can be split as
\begin{equation}
\label{LW}
\begin{split}
&\lambda \noaverage{W_a^c}(\th)  -  \noaverage{W_a^c} \circ T_\omega(\th)
= - \noaverage{\widetilde e_2^c}(\th)\\
&\lambda \noaverage{W_b^c}(\th)  -  \noaverage{W_b^c} \circ T_\omega(\th)
= - \noaverage{\widetilde A_2^c}(\th)\ .
\end{split}
\end{equation}
On the other hand, taking the average of \equ{FW} we obtain
\begin{equation}
\label{HW}
\begin{split}
&\overline{S}\, \overline{W_2^c}
+ (\overline{ S (W_b^c)^0   }
+\overline{\widetilde A_1^c})\beta=-\overline{ S (W_a^c)^0   }- \overline{\widetilde e_1^c}\\
&(\lambda -1) \overline{W_2^c}+\overline{\widetilde A_2^c}\beta
= -\overline{\widetilde {e_2^c}} \ .
\end{split}
\end{equation}
Provided that the non--degeneracy condition \equ{non-degeneracyW} in $(H5)$ is satisfied,
equations \equ{HW} yield $\overline{W_2^c}$ and $\beta$ as the solution of the finite
dimensional system:
\[
\left(\begin{array}{cc}
  \overline{ S} & \overline{S (W_b^c)^0}+\overline{\widetilde A_1^c} \\
  (\lambda -1)\Id_d & \overline{\widetilde A_2^c} \\
\end{array}\right)
\left(\begin{array}{c}
  \overline{W_2^c} \\
  \beta \\
\end{array}\right)=
\left(\begin{array}{c}
  -\overline{ S (W_a^c)^0   }- \overline{\widetilde e_1^c} \\
  -\overline{\widetilde {e_2^c}} \\
\end{array}\right)\ .
\]

Using Lemma~\ref{neutral}, we can solve
\equ{LW} to get $(W_a^c)^0$, $(W_b^c)^0$, which provide
$\noaverage{W_2^c}$ and finally we solve the first of \equ{DW} in
order to compute $\noaverage{W_1^c}$. This yields the solution
of \equ{CW}, which allows to find the correction
$(\Delta^c,\beta)$ on the center subspace.

Note that this correction ($\Delta^c, \beta$) does
not eliminate completely the error in the center direction,
but reduces it to $W^c \E_R$.

Note that we obtain $W_2^c$ solving a small divisor equation
and then we obtain $\beta$, $W^c_1$ by performing algebraic equations and applying
a contraction argument, leading to the following estimates:
\begin{equation}\label{Wcestimates}
 \| W^c \|_{\rho - \delta} \le C \delta^{-\tau}\| e\|_\rho\ ,
\end{equation}
\begin{equation} \label{betaestimates}
|\beta| \leq C\|e\|_\rho\ .
\end{equation}

Then, $\Delta^c$ is obtained multiplying $M$ and $W^c$ and,
hence, under the assumption that $\| M\|_{\rho - \delta}$ (which
we will prove holds inductively) is
bounded by a constant, we obtain
that:
\begin{equation} \label{Deltacbounds}
 \| \Delta^c \|_{\rho - \delta} \le C \delta^{-\tau}\| e\|_\rho\ .
\end{equation}

\subsubsection{Uniqueness properties of the approximate
solutions in the center direction}
\label{sec:uniqueness-center}
It will be important for future studies (e.g., for
the local uniqueness in Theorem~\ref{thm:uniqueness})  to note that the
$W^c$, out of which we construct $\Delta^c$, is obtained
applying Lemma~\ref{neutral}.

We observe that, by following the procedure indicated in
the previous paragraph, we obtain that the $W^c$ solving the
equation \eqref{DW} is unique up to adding a constant to $W^c_1$.
The estimates claimed in \eqref{Wcestimates} correspond to taking
$W^c_1$ with zero average.

This lack of uniqueness of the corrections has a geometric
interpretation related to the underdetermination of \eqref{inv}
remarked at the beginning of Section~\ref{sec:normalization}.
Since $\Delta^c  = M W^c$, with $M$ introduced in
\eqref{Mdefined},  we have that  adding a  constant $\sigma$  to
$W^c$ is tantamount to adding to $\Delta^c$ the quantity $DK \sigma$. That is,
we are changing the corrections by adding to them a movement
along the space of
solutions.

Note that the above statement of uniqueness refers only to
the solutions of \eqref{WC}. If our goal was to improve
the accuracy of the solutions, we have a flexibility of
changing the $\Delta^c$ by other terms that are much smaller than
the error.

\subsubsection{Solutions of the linearized equation in the
hyperbolic directions}
\label{sec:stable}

Let us now consider the stable subspace.
Let $\th'=T_\omega(\th)$, so that we can write \equ{WB} as
\beq{WC}
Df_\mu(K\circ T_{-\omega}(\th'))\ \Delta^{s}(T_{-\omega}(\th'))+
\Pi_{\th'}^s D_\mu f_\mu(K\circ T_{-\omega}(\th'))\beta
-\Delta^{s}(\th')=-\widetilde e^{s}(\th')\ ,
\eeq
where
\[
\widetilde e^{s}(\th')\equiv \Pi_{\th'}^{s} e\circ T_{-\omega}(\th')\ .
\]
We proceed now to solve \equ{WC} for the stable subspace. For any $\beta$, we can write
\beq{stableq}
\Delta^s(\th')=\widetilde e^s(\th')+
Df_\mu(K\circ T_{-\omega}(\th'))\ \Delta^s(T_{-\omega}(\th'))+
\Pi_{\th'}^s D_\mu f_\mu(K\circ T_{-\omega}(\th'))\beta\ ,
\eeq
which leads to the solution for $\Delta^s$ in the form
\beqa{Deltas}
\Delta^s(\th')&=&\widetilde e^s(\th')+\sum_{k=1}^\infty \Big(Df_\mu(K\circ T_{-\omega}(\th'))\times \cdots \times
Df_\mu(K\circ T_{-k\omega}(\th'))\Big)\ \widetilde e^s(T_{-k\omega}(\th'))\nonumber\\
&+&\Pi_{\th'}^s D_\mu f_\mu(K\circ T_{-\omega}(\th'))\beta\nonumber\\
&+&\sum_{k=1}^\infty \Big(D f_\mu(K\circ T_{-\omega}(\th'))\times...\times
D f_\mu(K\circ T_{-k\omega}(\th'))\
\Pi_{\th'}^s D_\mu f_\mu(K\circ T_{-(k+1)\omega}(\th')\Big)\ \beta\ .\nonumber\\
\eeqa
By the variations of parameters formula, we guess a solution of the form \equ{Deltas}.
Then, we observe that the series in \equ{Deltas} converges uniformly in $\A_\rho$, because of the
bounds assumed in \equ{growthrates}. Hence, we can substitute \equ{Deltas} in \equ{stableq} and
rearrange the terms, so that we can verify that \equ{stableq} is satisfied by \equ{Deltas}.

Due to the growth conditions \equ{growthrates} and due to the fact that $D_\mu f_\mu$ is a bounded operator, we obtain that
\beq{deltas2}
\|\Delta^s\|_\rho\leq C\ \Big(\|\widetilde e^s\|_\rho+|\beta|\
\|\Pi_{\th+\omega}^s D_\mu f_\mu \circ K\|_\rho
\Big)\cdot \sum_{k=0}^\infty \lambda_-^k
\eeq
for some constant $C>0$.

Concerning the unstable subspace, from \equ{WB} we can write in a similar way for any $\beta$:
$$
\Delta^u(\th)=(Df_\mu)^{-1}(K(\th))\ [- e^u(\th)-
\Pi_{\th+\omega}^u D_\mu f_\mu(K(\th))\beta+\Delta^u(T_\omega(\th))]\ ,
$$
which leads to the expression
\beqa{Deltau}
\Delta^u(\th)&=&-\sum_{k=0}^\infty \Big((Df_\mu)^{-1}(K(\th))\times...\times
(Df_\mu)^{-1}(K\circ T_{k\omega}(\th))\Big)\  e^u(T_{k \omega}(\th))\nonumber\\
&-&\sum_{k=0}^\infty \Big((Df_\mu)^{-1}(K(\th))\times...\times
(Df_\mu)^{-1}(K\circ T_{k\omega}(\th))\
\Pi_{\th+\omega}^u D_\mu f_\mu(K\circ T_{k\omega}(\th)\Big)\beta\ .\nonumber\\
\eeqa
{From} \equ{Deltau} and \equ{growthrates} we obtain:
\beq{deltau}
\|\Delta^u\|_\rho\leq C\ \Big(\|\tilde e^u\|_\rho+|\beta|\
\|\Pi_{\th+\omega}^u D_\mu f_\mu \circ K\|_\rho\Big)\cdot  \sum_{k=0}^\infty \lambda_+^k
\eeq
for some constant $C>0$. Expressions \equ{deltas2} and \equ{deltau} lead to
$$
\|\Delta^s\|_\rho\leq C\|e\|_\rho\ ,\qquad \|\Delta^u\|_\rho\leq C\|e\|_\rho\ .
$$

\begin{remark}
In numerical implementations, it may be advantagous to
solve \eqref{WC} by adding an equation that makes the cocycle
diagonal.  The effort required is not that much, since
the linearizing transformation can be computed iteratively.
In this case, we can use the Fourier methods to
compute the solutions (see \cite{HLlnum}).
\end{remark}

\subsubsection{Nonlinear estimates for the step}
\label{sec:nonlinear}
Notice that, in all the cases above, we have obtained
estimates for the corrections in terms of
the error. The most delicate ones are those in
the center direction since they involve
the small divisors.

We have that  for all $0 < \delta < \rho$, the
corrections satisfy \equ{correctionbounds}, namely:
\begin{equation}
\label{goodestimates}
\begin{split}
&\| \Delta\|_{\rho - \delta}  \le C  \delta^{-\tau}  \| e\|_\rho\\
&\| D \Delta\|_{\rho - \delta}  \le C  \delta^{-\tau-1}  \| e\|_\rho\\
&|\beta|\le C \| e\|_\rho\ .
\end{split}
\end{equation}

The first line of \eqref{goodestimates} is wasteful in the
stable and unstable directions. The second line can be obtained from
the first line estimates for
$\| \Delta\|_{\rho - \frac{1}{2}\delta}$ and then, using Cauchy estimates for
the derivative.

\subsubsection{Nonlinear estimates}
Let us now conclude the proof of part 2) of Proposition~\ref{pro:iterative}.

We see that, under the assumption
\eqref{inductive}, we have  that we can define the composition
$f_{\mu + \beta}(K + \Delta)$ and that
$(K + \Delta)(\torus^d_{\rho - \delta})$ is away from the boundary of
the domain of the function $f$.

We can write the error of $K'  = K + \Delta$, $\mu'  = \mu + \beta$ as
\[
\begin{split}
f_{\mu + \beta} \circ (K + \Delta) - & (K + \Delta) \circ T_\omega  =
f_\mu \circ K +  Df_\mu \circ K \Delta + D_\mu f_\mu \circ K \beta + \E_T \\
& \phantom{AAAAAA}
- K \circ T_\omega - \Delta \circ T_\omega   \\
&\phantom{AAAAAA} = \E_T + \E_R W\ ,
\end{split}
\]
where $\E_T$ is the Taylor reminder of the expansion of
the composition and we note that the correction $\Delta$ has
been chosen precisely to cancel the  error.
The Taylor estimates are elementary
\[
\| \E_T \|_{\rho - \delta} \le C ( \| \Delta\|_\rho^2 + |\beta|^2)
\le C \delta^{-2\tau} \|e\|_\rho^2
\]
and the estimates for
$\E_R$ -- the error in the approximate reducibility -- are in
\eqref{approxredestimates}, while the estimates for $W$ can be obtained from
\eqref{goodestimates}.

\begin{remark}
The algorithm we have implemented uses an invariant
splitting to solve the linearized equation. Of course,
after the step, we need to apply  the closing Lemma~\ref{lem:closing} to
obtain a new invariant splitting.
The algorithm we presented is
very efficient from the theoretical point of view. Nevertheless
in a numerical implementation, it is wasteful to obtain an exact
splitting at each step.  In numerical implementations it suffices
to obtain invariant splittings up to  an error comparable
to the error in the invariance equation. Refining the
splittings to more accuracy does not lead to significant improvements
of the step.

Hence, in numerical implementations, it would be advantageous to
develop a Newton method for the invariant splittings and for
the invariance equation, simultaneously.

A rigorous presentation and
 estimates for the  Newton method for the invariance of
the torus, and the splittings (which also allows some elliptic directions)
can be found in \cite{deViana17}.
\end{remark}

\subsubsection{Change in the non-degeneracy conditions after
an iterative step}\label{sec:changes}

Now, we estimate the changes on the non-degeneracy conditions
induced by the change of the $\mu, K$ in the iterative step.
Since we start with an invariant splitting, we have that $\gamma_\th^{\sigma,\sigma'}=0$. Hence, denoting by
$\hat \gamma_\th^{\sigma,\sigma'}$ the cocycle associated to $K'=K+\Delta$, $\mu'=\mu+\beta$, then
$$
\|\hat \gamma_\th^{\sigma,\sigma'}-\gamma_\th^{\sigma,\sigma'}\|_\rho<C(\|D^2f_\mu\circ K\|_\rho\,\|\Delta\|_\rho+
\|D_\mu D f_\mu\circ K\|_\rho\,|\beta|)<C\delta^{-\tau}\E\ ,
$$
thus yielding 3.1 of Theorem~\ref{pro:iterative}.

The results of Proposition~\ref{pro:changes} below give part 3.2 of Proposition~\ref{pro:iterative}, beside
implying the estimate \equ{proj} of Theorem~\ref{whiskered}.

\begin{proposition}\label{pro:changes}
With the notations and assumptions of Proposition~\ref{pro:iterative}, we have the following bounds.

$(i)$ Let $\S'$ be the quantity $\S$ in \equ{non-degeneracyW} after one iterative step,
namely with $K$, $\mu$ replaced by $K'=K+\Delta$, $\mu'=\mu+\beta$ as in
Proposition~\ref{pro:iterative}. Let $0<\delta<\rho$; then:
\beq{Sbound}
\|\S'\|_{\rho-\delta}\leq  \|\S\|_\rho+C\delta^{-\tau}\|e\|_\rho
\eeq
for a suitable constant $C>0$.\\

$(ii)$ The changes in the projections are bounded as
\beq{projbound}
\|\widetilde \Pi^{s/c/u}_{\th}-\Pi^{s/c/u}_\th\|_{\rho -\delta}\leq C\|K'-K\|_\rho \le C \delta^{-\tau} \|e\|_\rho\ .
\eeq

$(iii)$ The changes in the diagonal cocycles are bounded
by
\beq{lambound}
\| \hat\gamma_\th^{\sigma, \sigma} - \tilde\gamma_\th^{\sigma, \sigma} \|_{\rho - \delta}
\le C( \delta^{-\tau}\E + \E_h )\ .
\eeq
\end{proposition}

\begin{proof}
We note that the matrix $\S$ is an algebraic expression
of the derivatives of $K$ and the derivatives of
$f_{\mu}$ evaluated at $K$. The projections of the above quantities
are taken on the center directions. We can estimate all
the changes by adding and subtracting, so that we get
only terms in which one changes.
The derivatives of $K$ change
by the derivatives of $\Delta$, which can  be bounded using
Cauchy estimates from the estimates for $\Delta$. Hence, the change of
these terms can be bounded by $C \delta^{- \tau}\| e\|_\rho$, thus
leading to \equ{Sbound}.

The change in the projections is also estimated already and come
from the application of Lemma~\ref{lem:closing}, thus leading to \equ{projbound}.

The changes of the term $Df_{\mu +\beta}$ can be bounded by the
size of $\beta$ and the size of $\Delta$ multiplied by the second derivatives
of $f$ with respect to its arguments.  The terms change
by quantities that are smaller than the previous one.
The estimate in \equ{lambound} is bounded by the sum of the norms of
$D^2f_\mu\circ K\, \Delta$ and $D_\mu D f_\mu\circ K\, \beta$ and using \equ{Eh}.
\end{proof}

\subsubsection{Iteration of the iterative step}\label{sec:convergence}
It is a classical argument in KAM theory that,
if the initial error is small
enough,  the inductive step can be iterated indefinitely
and that it converges to the true solution.  We can also
estimate the difference between the limit and the initial
approximate solutions.

\begin{remark}
\label{rem:uniformrho}
It is important to note that assumptions of smallness in
Lemma~\ref{general} are independent of the $\rho$ considered.
In the applications, we will be considering a sequence of corrections in a
sequence of decreasing domains. We will just need that
the norm of the corrections (in the smallest domain) are
sufficiently small.
\end{remark}

\begin{remark}
\label{rem:uniformtwist}
If we assume that
\[
\|D K_0 -D K_j \|_{\rho_j} \le \alpha
\]
for some $\alpha > 0$ we obtain that
\eqref{non-degeneracyW}
and that the  $\| (DK_j^\T DK_j)^{-1} \|_{\rho_j} $ are bounded by  a
value slightly bigger than the original one, say twice.

Similarly, if $\|K_0 - K_j\|_{\rho_j}$ are sufficiently small, we
can ensure that the range of $K_j$ is inside the domain of $f_\mu$ and at a distance  bigger than $\eta/2$
from the boundary of the domain.

Provided that during the iteration the $DK$ does
not leave the neighborhood more than the distance $\alpha$, we can
use the bounds corresponding to the chosen values of
the condition numbers.

Recalling the definition of $M$ in \equ{Mdefined}, we notice that $M_j-M_0$ is
an algebraic expression of $DK_j$ and $DK_0$. By the mean value theorem, $M_j-M_0$ is
bounded by $DK_j-DK_0$ and hence it is small, if $DK_j$ is close to $DK_0$.
\end{remark}

We will just repeat the standard argument.
We start by fixing the sequences of domains where we will be
doing estimates. Let $\delta_j$, $\rho_{j+1}$ be defined as
$$
\delta_j={\delta_0\over {2^{j+2}}}\ ,\qquad \rho_{j+1}=\rho_j-\delta_j\ .
$$
We take $\delta_0$ to be $1/2$ of the total loss of domain
that we allow in the conclusions of Theorem~\ref{whiskered}.
The $\rho_j$ will be the domains where we will carry estimates in the
$j$-th step.

We recall that to perform the iterative step, we need to make
two inductive assumptions that ensure that we can define
the composition. We will also need to assume that we
are in the region when all the hyperbolicity constants are uniform
(slightly worse than those in the original problem)
and we will also assume that other non-degeneracy conditions are
satisfied.

The important thing to observe is that, if we can carry out
$j$ steps, the conditions for being able to carry out the next step
and remaining in the neighborhood are implied by
\begin{equation} \label{uniformestimatescondition}
\begin{split}
& \| K_0 - K_j\|_{\rho_j} \le \alpha\ ,\\
& \| D K_0 - D K_j\|_{\rho_j} \le \alpha
\end{split}
\end{equation}
for some $\alpha > 0$, which is independent of $j$, see
Remark~\ref{rem:uniformrho} and Remark~\ref{rem:uniformtwist}.

We will assume that for $j$ steps, we have been able to carry out the
step and remain in the set of analytic functions where we have
\eqref{uniformestimatescondition} and, hence, we can perform the iteration.
We will show that, if the initial error
is small enough,  the error has decreased  so much after the
$j$-th step  that we
can apply again the result. Note that this is very similar
to the estimates that one carries out in the elementary Newton Method.

Denoting by $\eps_h\equiv\|e_h\|_{\rho_h}$, we have:
\beqano
\eps_h&\leq&C \delta_h^{-\tau} \eps_{h-1}^2\nonumber\\
&=&(C\delta_0^{-\tau} 2^{3\tau})\ 2^{(h-1)\tau}\ \eps_{h-1}^2\nonumber\\
&\leq&(C\delta_0^{-\tau} 2^{3\tau)})^{1+2+...+2^{h-1}}\
2^{\tau\ [(h-1)+2(h-2)+\ldots+2^{h-2}]}\ \eps_0^{2^h}\nonumber\\
&\leq&(C\delta_0^{-\tau} 2^{3\tau}\, 2^{\tau} \eps_0)^{2^h-1}\ \eps_0\nonumber\\
&=&(\kappa_0 \eps_0)^{2^h-1}\ \eps_0\ ,
\eeqano
where
$$
\kappa_0 = C \delta_0^{-\tau}\, 2^{4\tau}\ .
$$

Starting from $K_h=K_{h-1}+\Delta_{h-1}$, a bound on $K_h-K_0$ and its derivative is obtained as follows:
\beqano
\|K_h-K_0\|_{\rho_h}&=& \|\sum_{j=0}^{h}\Delta_j\|_{\rho_h} \le
\sum_{j=0}^{h} \| \Delta_j\|_{\rho_h}\nonumber\\
&\leq& C\ \sum_{j=0}^{h} \delta_0^{-\tau} 2^{\tau(j+2)}\ (\kappa_0\eps_0)^{2^j-1}\eps_0\ , \\
\eeqano

\beqano
\|DK_h-DK_0\|_{\rho_h}&=& \|\sum_{j=0}^{h}D \Delta_j\|_{\rho_h} \le
\sum_{j=0}^{h} \| D \Delta_j\|_{\rho_h}\nonumber\\
&\leq& C\ \sum_{j=0}^{h} \delta_0^{-\tau-1} 2^{\tau(j+2) +1 }\ (\kappa_0\eps_0)^{2^j-1}\eps_0\ , \\
\eeqano

\begin{equation}
\label{muestimates}
\begin{split}
|\mu_h-\mu_0|&= |\sum_{j=0}^{h}\beta_j|\le
\sum_{j=0}^{h} |\beta_j|
\\
&\leq C\ \sum_{j=0}^{h} \eps_j\leq C\ \sum_{j=0}^{h} (\kappa_0\eps_0)^{2^j-1}\leq 2C\ \eps_0\ ,
\end{split}
\end{equation}
provided that $\kappa_0 \eps_0$ is smaller than $1/2$.

The important points of the above estimates are:
\begin{enumerate}
\item
The bound for $\eps_h$  can be made as
small as we want by making $\eps_0 $ sufficiently small.
Hence, by imposing some smallness assumption in $\eps_0$, we can ensure
that $K_h$ does not leave a neighborhood of $K_0$.
\item
In particular, under suitable
smallness assumptions on $\eps_0$ (independent of $h$),
  the range of $K_h(\torus_{\rho_h}^d)$ is
$\eta/2$ away from the boundary of the domain of $f_\mu$.
\item
Under suitable
smallness assumption on $\eps_0$ (independent of $h$),
 $\| DK_h  -DK_0\|_{\rho_h}$ is so small that we have the
estimates on the hyperbolic splitting assumed in
Lemma~\ref{general}. The $K_h$ are also in the region where we have
uniform bounds on the non-degeneracy constants.
\item
Since
$\| \Delta_h \|_{\rho_j} \le \eps_h \delta_h $ decreases very fast,
much faster than an exponential, we obtain that by
taking $\eps_0$ sufficiently small, we recover the assumption
\eqref{inductiveass}.
\item
Putting together the above two remarks, we obtain that under a finite
number of smallness assumptions on $\eps_0$, we obtain that we
can iterate the procedure indefinitely and remain in the neighborhood
where we have uniform estimates for all the non-degeneracy assumptions.
\item
Since $\sum_h \| \Delta_h \|_{\rho_j} \le \eps_h \delta_h \le C \eps_0$,
we obtain that the $K_h$ converge in $\A_{\rho_\infty}$ and
they satisfy the conclusions of Theorem~\ref{whiskered},
including the estimates \equ{Kmu}.
\end{enumerate}

We conclude by noticing that \equ{lamest} come from the fact that, due to Lemma~\ref{precise},
the estimates $|\lambda_\pm-\tilde\lambda_\pm|$, $|\lambda_c^\pm-\tilde\lambda_c^\pm|$ are bounded by a
constant time $\|\gamma-\tilde\gamma\|_\rho$; with the same argument as in \equ{lambound},
we obtain \equ{lamest}.

\section{Local uniqueness of the solution}\label{sec:uniqueness}

In this Section we give a result on the local uniqueness of the solution of the invariance
equation, see Theorem~\ref{thm:uniqueness}.

\begin{theorem}\label{thm:uniqueness}
Let $\omega\in\D_d(\nu,\tau)$, $d\leq n$; let $\M$ be as in Section~\ref{sec:CS} and let
$f_\mu:\M\rightarrow\M$, $\mu\in\real^d$, be a family of real analytic, conformally symplectic
maps with $0<\lambda<1$.

Let $(K^{(1)},\mu^{(1)})$, $(K^{(2)},\mu^{(2)})$ be solutions of \equ{inv} and assume that
\beq{normaliz}
\int_{\torus^d} M^{-1}(\theta)\ [K^{(2)}(\theta)-K^{(1)}(\theta)]_1\ d\theta=0\ ,
\eeq
Recall that we defined $M$ on the central bundle using  $J^c$,
where $M$ is as in \equ{M} and $[\cdot]_1$ means that we take the first component.

Assume that the non-degeneracy condition $(H5)$ is satisfied at $(K^{(1)},\mu^{(1)})$.
Let $W^c$, $W^s$, $W^u$
be the projections of $M^{-1}(\theta)(K^{(2)}(\theta)-K^{(1)}(\theta))$ on
the center, stable, unstable subspaces. Let $0<\delta<\rho$ and assume that the following
inequalities are satisfied:
\beqa{ineq}
C\ \nu\ \delta^{-\tau}\ \max(\|W^c\|_{\rho+2\delta},|\mu^{(2)}-\mu^{(1)}|)\ <\ 1\nonumber\\
C\ (\|W^s\|_{\rho+\delta}+\|W^u\|_{\rho+\delta})\ <\ 1\ .
\eeqa
Then, we have:
$$
K^{(1)}=K^{(2)}\ ,\qquad \mu^{(1)}=\mu^{(2)}\ .
$$
\end{theorem}

\begin{remark}
The normalization condition \eqref{normaliz} has a very transparent
geometric interpretation. Remember that the $M$ is a linear change of
variables in the torus that selects the tangent and the symplectic
conjugate.  The normalization chosen roughly  imposes that
 the average of the increase  in
phase of $K^{(2)}$ is the same as that of  $K^{(1)}$, since the average over
the angles of the difference of the two solutions in the parametric coordinates is zero (compare
with Section~\ref{sec:normalization}).
\end{remark}

\begin{remark}
Note that if we take $K^{(2)}(\theta) = K^{(1)}(\theta + \sigma)$,
then
\[
\frac{d}{d \sigma}  \int_{\torus^d} M^{-1}(\theta)\ [K^{(1)}(\theta + \sigma)-K^{(1)}(\theta)]_1\ d\theta=0\ \big|_{\sigma = 0} = {\rm Id}\ .
\]

Therefore, the  finite dimensional implicit function theorem shows that
given any $K^{(2)}$ close to $K^{(1)}$, there is a unique $\sigma$ such
that $K^{(2)} \circ T_\sigma$ satisfies the normalization \eqref{normaliz}.
The estimates on $\sigma$ show that if there existed a
solution $K^{(2)}$ close to $K^{(1)}$, then,
we could, without loss of generality, get a normalized solution
by composing it with a translation. This normalized translation
satisfies similar hypothesis of proximity to $K^{(1)}$ as the
original $K^{(2)}$.

Hence, the normalization can be interpreted as fixing the element in the
family of solutions mentioned at the beginning of
Section~\ref{sec:normalization}.
\end{remark}

\begin{proof}
Assume that $(K^{(1)},\mu^{(1)})$, $(K^{(2)},\mu^{(2)})$ satisfy the invariance equations
$$
f_{\mu^{(1)}}\circ K^{(1)}=K^{(1)}\circ T_\omega\ ,\qquad f_{\mu^{(2)}}\circ K^{(2)}=K^{(2)}\circ T_\omega\ .
$$
Let us define the quantity $\widetilde R$ as
\beq{tildeR}
\widetilde \R=f_{\mu^{(2)}}\circ K^{(2)}-f_{\mu^{(1)}}\circ K^{(1)}-Df_{\mu^{(1)}}\circ K^{(1)}\ (K^{(2)}-K^{(1)})-D_\mu
f_{\mu^{(1)}}\circ K^{(1)}\ (\mu^{(2)}-\mu^{(1)})\ .
\eeq
By Taylor's theorem, setting the constant $\widetilde C$ as
$$
\widetilde C={1\over 2}\ \max(\|D^2f_{\mu^{(1)}}\circ K^{(1)}\|_{\rho-\delta},\|D^2_\mu f_{\mu^{(1)}}\circ K^{(1)}\|_{\rho-\delta})\ ,
$$
one has the quadratic estimates
\beq{estR}
\|\widetilde \R\|_{\rho-\delta}\leq \widetilde C\ (\|K^{(2)}-K^{(1)}\|^2_\rho+|\mu^{(2)}-\mu^{(1)}|^2)\ .
\eeq
Projecting \equ{tildeR} on the center subspace, one obtains
\beqa{tildeRC}
&&\Pi_{K^{(2)}(\theta)}^c\{K^{(2)}\circ T_\omega-K^{(1)}\circ T_\omega-Df_{\mu^{(1)}}\circ K^{(1)}\ (K^{(2)}-K^{(1)})-D_\mu
f_{\mu^{(1)}}\circ K^{(1)}\ (\mu^{(2)}-\mu^{(1)})\}\nonumber\\
&-&\Pi_{K^{(2)}(\theta)}^c \widetilde \R=0\ .
\eeqa

Let $W=W(\theta)$ be defined\footnote{Notice that the formulas are the same as the iterative step.
  Therefore, even if we use the same letters as in the study of the iterative
step the interpretation of the letters is slightly different.  In
the iterative step, $W$ was an unknown to be found. Here, it is
a known quantity and our goal is to show that it satisfies some equations
and that we have estimates for it.}
by:
\[
K^{(2)}(\theta)-K^{(1)}(\theta)=M(\theta)W(\theta)\ ,
\]
where we intend that $M(\theta)$ is computed for $K^{(1)}$, i.e.
\[
M(\th) = [ DK^{(1)}(\th)\ |\  (J^c)^{-1}\circ K^{(1)}(\th)\ DK^{(1)}(\th)\ (DK^{(1)}(\th)^TDK^{(1)}(\th))^{-1}]\ .
\]

Let $W^c(\theta)=\Pi_{K^{(2)}(\theta)}^c W(\theta)$; similarly to \equ{estR}, one has that
\beq{PiR}
\|\Pi_{K^{(2)}(\theta)}^c \widetilde \R\|_{\rho-\delta}\leq \widetilde C\ (\|W^c\|_\rho^2
+|\mu^{(2)}-\mu^{(1)}|^2)\ .
\eeq
Using the automatic reducibility\footnote{Note that because $(K^{(1)},\mu^{(1)})$ is an
exact solution of \eqref{inv}, the approximate reducibility is
exact.} \equ{red} on the center subspace, we obtain:
\beqa{aux}
&&\B(\theta)W^c(\theta)-W^c(\theta+\omega)+M^{-1}(\theta+\omega)\
\Pi_{K^{(2)}(\theta)}^c (D_\mu f_{\mu^{(1)}}\circ K^{(1)}(\theta))(\mu^{(2)}-\mu^{(1)})\nonumber\\
&+&M^{-1}(\theta+\omega)\Pi_{K^{(2)}(\theta)}^c  \widetilde \R(\theta)=0\ .
\eeqa

Hence, the $W^c$ is a solution of the equation \equ{aux}
By the remarks in Section~\ref{sec:uniqueness-center} we
obtain that the solutions are unique up to adding a constant vector
to $W^c_1$.  The estimates for the zero average solution are
obtained from \eqref{aux}.  We note that using \eqref{tildeRC},
we obtain that the average of $W^c_1$ should be bounded by the
size of $R$. Hence, we obtain  that  the $W^c$ used here
satisfies:
\beq{est}
\|W^c\|_{\rho-\delta}\leq C\nu\delta^{-\tau}\ \|\Pi_{K^{(2)}(\theta)}^c  \widetilde \R\|_\rho\ , \qquad
|\mu^{(2)}-\mu^{(1)}|\leq C\ \|\Pi_{K^{(2)}(\theta)}^c  \widetilde \R\|_\rho\ .
\eeq
Next, we recall Hadamard's three circle theorem (\cite{Ahlfors}), which gives
\beq{Had}
\|W^c\|_\rho\leq C\ \|W^c\|_{\rho-2\delta}^{1\over 2}\
\|W^c\|_{\rho+2\delta}^{1\over 2}\ .
\eeq
Hence, from \equ{PiR}, \equ{est}, \equ{Had}, we obtain
\beqano
\max(\|W^c\|_{\rho-2\delta},|\mu^{(2)}-\mu^{(1)}|)&\leq&C\nu\delta^{-\tau}\ \|\Pi_{K^{(2)}(\theta)}^c  \widetilde \R\|_{\rho-\delta}\nonumber\\
&\leq&C\nu\delta^{-\tau}\ \max(\|W^c\|_\rho^2,|\mu^{(2)}-\mu^{(1)}|^2)\nonumber\\
&\leq&C\nu\delta^{-\tau}\ \max(\|W^c\|_{\rho+2\delta},|\mu^{(2)}-\mu^{(1)}|)
\ \max(\|W^c\|_{\rho-2\delta},|\mu^{(2)}-\mu^{(1)}|)\nonumber\\
&<&\max(\|W^c\|_{\rho-2\delta},|\mu^{(2)}-\mu^{(1)}|)\ ,
\eeqano
if the first in \equ{ineq} holds. This allows to conclude that $W^c=0$ and $\mu^{(1)}=\mu^{(2)}$.

\vskip.1in

We now project on the hyperbolic subspaces; let
$$
W^h=\Pi_{K^{(2)}(\theta)}^h M (K^{(2)}-K^{(1)})\ ,
$$
where $h=s$ or $h=u$. From Section~\ref{sec:stable}, we have that
$$
\|W^h\|_{\rho-\delta}\leq C\ \|\widetilde \R\|_\rho\ .
$$
From \equ{estR}, knowing that $W^c=0$ and $\mu_1=\mu_2$, we have
$$
\|\widetilde \R\|_{\rho-\delta}\leq \widetilde C\ (\|W^s\|_\rho^2+\|W^u\|_\rho^2)\ .
$$
From Hadamard's three circle theorem, we have:
\beqano
\|W^s\|_\rho^2+\|W^u\|_\rho^2&\leq&C(\|W^s\|_{\rho+\delta}\|W^s\|_{\rho-\delta}+
\|W^u\|_{\rho+\delta}\|W^u\|_{\rho-\delta})\nonumber\\
&\leq&C(\|W^s\|_{\rho+\delta}+\|W^u\|_{\rho+\delta})\ \|\widetilde \R\|_{\rho-\delta}\nonumber\\
&\leq&C(\|W^s\|_{\rho+\delta}+\|W^u\|_{\rho+\delta})\ (\|W^s\|_\rho^2+\|W^u\|_\rho^2)\ .
\eeqano
This relation, together with the second inequality in \equ{ineq}, leads to $W^s=W^u=0$.
\end{proof}

\section{Domains of analyticity of Lindstedt expansions of
whiskered tori}\label{sec:domain}

In this Section we investigate the domains of analyticity of
whiskered tori in conformally symplectic systems in the limit of
small dissipation. The discussion below proceeds
along the lines of \cite{CCLdomain}. We develop an
asymptotic expansion (Lindstedt series) and use
this as a starting point for the application of
Theorem~\ref{whiskered}.

The perturbative expansions for the parameterization of
the torus are not very different from the treatment in \cite{CCLdomain}
-- the main difference is that  the equation that needs
to be studied in the iterative step requires to consider the hyperbolic
directions, but this will be very similar to the
treatment in Section~\ref{sec:stable}.

The construction of expansions of invariant bundles will be based on a perturbative
treatment of the equations \eqref{invariance2-s}, \eqref{invariance2-u}.

It is interesting to note that the Lindstedt series  have a triangular
structure. The series of the parameterization of the torus
solve the invariance  equation
by themselves; on the other hand the series for the parameterization of
the invariant
spaces require the series of parameterization of the torus.

The formal expansions of $K$, $\mu$,
$ A^{\sigma}$ constructed to order $N$ by the Lindstedt method produce
objects  that satisfy the invariance equations up to
an error which has norm less than $C_N |\eps|^{N+1}$. We apply
Theorem~\ref{whiskered} for $\eps$ in a domain such that
the Diophantine properties of $\lambda$ are good enough.
We conclude that there exist exactly invariant $K, \mu$ solutions and
that the distance between $K$, $\mu$ and the formal expansions
to order $N$ are bounded in the domain,
namely the formal series expansions obtained by Lindstedt
series are an asymptotic expansion of the true solution.

\subsection{Description of the set up}

Consider a family of mappings $f_{\mu_\eps,\eps}:\M\rightarrow\M$, such that
$$
f_{\mu_\eps,\eps}^*\,\Omega=\lambda(\eps)\Omega\ ,
$$
where in this Section the conformal factor $\lambda$ is assumed to be an analytic function of a small parameter $\eps$
and it is such that $\lambda(0)=1$, thus allowing to recover the symplectic case for $\eps=0$.
In particular, we assume that $\lambda=\lambda(\eps)$ takes the form
$$
\lambda(\eps) = 1 + \alpha \eps^a + O(|\eps|^{a+1})
$$
for some $a>0$ integer and $\alpha\in\complex\setminus\{0\}$.

\subsection{Some standard definitions}

In our results, we will show
that some functions
depending  on a parameter $\eps$ are analytic in $\eps$.
The stronger (and most natural) way to study
the dependence is to consider the functions for
fixed $\eps$ taking values in a Banach space of
functions. The analyticity is taken to mean that the
function can be expressed as a Taylor series
in the neighborhood of each point.  Other seemingly weaker
definitions turn out to be equivalent (\cite[Chapter III]{HilleP}).

As standard, given a sequence $B_j$ of
elements in a Banach space,  a formal power series in $\eps$
is an expression of the form  $B^\infty_\eps
= \sum_{j =0}^\infty \eps^j B_j$ (the sum is  not meant to converge).
We  denote by
$B^{[\le N]}_\eps =  \sum_{j =0}^N \eps^j B_j$.

In our application, we will consider series in which
the Banach spaces are just $\complex^d$
(in the case of $K_\eps$, the maps take  values in the phase space
and in the case of $A^\sigma_\eps$, it is a space of bundle maps).

\subsection{Description of the domains of analyticity}
Recalling Definition~\ref{def:DC}, we introduce the following sets, where the Diophantine constants behave in such a way that one can start
an iterative convergent procedure (see \cite{CCLdomain}).

\begin{definition}\label{def:sets}
For $A>0$, $N\in\integer_+$, $\omega\in\real^d$, let the sets $\G=\G(A;\omega,\tau,N)$ and
$\Lambda=\Lambda(A;\omega,\tau,N)$ be defined as
\beqano
\G(A;\omega,\tau,N)&\equiv&\{\eps\in\complex:\ \nu(\lambda(\eps);\omega,\tau)\ |\lambda(\eps)-1|^{N+1}\leq A\}\ ,\nonumber\\
\Lambda(A;\omega,\tau,N)&\equiv&\{\lambda\in\complex:\ \nu(\lambda;\omega,\tau)\ |\lambda-1|^{N+1}\leq A\}\ .
\eeqano
For $r_0\in\real$, let
\beq{Gr0}
\G_{r_0}(A;\omega,\tau,N)=\G\cap \{\eps\in\complex:\  |\eps| \le r_0\}\ .
\eeq
\end{definition}

\subsection{Statement of the main result on
domains of analyticity, Theorem~\ref{thm:domain}}\label{sec:dom}

We will prove that the parametrization and the
drift are analytic in a domain $\G_{r_0}$ as in \equ{Gr0} for a
sufficiently small $r_0$. This domain is obtained by removing from a
ball centered at zero a sequence of smaller balls whose center lies
along smooth lines going through the origin. The radii of the balls
which have been removed decrease faster than any power of the distance
of their center from the origin.
As in \cite{CCLdomain}, we conjecture that this domain is essentially optimal.

\begin{theorem}\label{thm:domain}
Let $\omega \in \D_d(\nu, \tau)$, $d\leq n$, as in \eqref{DC},
let $\M$ be as in Section~\ref{sec:CS}, and let $f_{\mu,\eps}$ with
$\mu\in\Gamma$ with $\Gamma\subseteq \complex^d$ open, $\eps\in\complex$, be a family of conformally
symplectic maps.

Assume that the family of maps $f_{\mu,\eps}$ is conformally symplectic,
with a conformal symplectic factor that depends analytically
on $\eps$, $\lambda_\eps = 1 + \alpha \eps^a  + O(\eps^{a+1})$ for
$\alpha \in \real$, $\alpha \ne 0$, $a \in \nat$.

(A1) Assume that for $\mu=\mu_0,\eps= 0$ the map $f_{\mu_0, 0}$ admits a whiskered invariant torus.

This assumption amounts to the following requirements:

(A1.1) There exists an embedding $K_0:\torus^d \rightarrow \M$, $K_0  \in \A_\rho$ for some
$\rho>0$, such that
\beq{inv0}
f_{\mu_0,0} \circ K_0 = K_0 \circ T_\omega\ .
\eeq

(A1.2) There exists a splitting $T_K(\th)\M =
E^s_\th \oplus E^c_{\th}  \oplus E^u_\th$.
This splitting is invariant under the cocycle
$\gamma^0_\th = D f_{\mu_0,0} \circ K_0(\th) $
and satisfies Definition~\ref{def:trichotomy}.
The ratings of the  splitting  satisfy the assumptions (H3), (H3')
and (H4) of Theorem~\ref{whiskered}.

(A.2) We assume that the function $f_{\mu, \eps}(x)$ is
analytic in all of its arguments and that the analyticity domains
are large enough. That is:

(A2.1) Both the embedding $K_0(\th)$ and the splittings $E^{s,c,u}_\th$
considered as a function of $\th$ are  in $\A_{\rho_0}$ for some $\rho_0 > 0$.

(A2.2) Assume that  there is a domain
$\U\subset\complex^n/\integer^n \times \complex^n $ such
that for $|\eps| \le \eps^*$ and all $\mu$ with
$| \mu - \mu_0| \le \mu^*$, we have that $f_{\mu, \eps}$ is
defined in $\U$ and we also have \eqref{compositions}.

(A3) The invariant torus satisfies the twist
condition (H5) of Theorem~\ref{whiskered}.

\noindent
Then, we have:

\medskip

B.1) We can compute formal power series expansions
\[
K_\eps^{[\infty]}= \sum_{j=0}^\infty\eps^j  K_j
\qquad
\mu_\eps^{[\infty]} =  \sum_{j=0}^\infty\eps^j  \mu_j
\]
satisfying \eqref{inv} in the sense of formal power series, which means that
for any $0 < \rho' \le \rho$ and $N\in\nat$, we have
$$
|| f_{\mu_{\eps}^{[\le N]}, \eps} \circ K^{[\le N]}_\eps -  K^{[\le N]}_\eps \circ T_\omega ||_{\rho'} \le C_N |\eps|^{N+1}\ .
$$

B.2) We can compute four formal power series expansions
\[
A^{\sigma,\infty}_\eps= \sum_{j=0}^\infty\eps^j  A_j^\sigma\ ,\qquad
\sigma = \st, \hat \st, \un, \hat \un\ ,
\]
\[
A^{\sigma}_j(\th) : E^\sigma_0(\th) \rightarrow E^{\hat \sigma}_0(\th)
\]
and the $A^\sigma_j \in \A_\rho$ in such a way that
the operators satisfy the equation
\eqref{invariance2-s}, \eqref{invariance2-u}
for invariant dichotomies  in the sense of power series.

\medskip

B.3) For the set $\G_{r_0}$ as in
\equ{Gr0} with $r_0$ sufficiently small and for $0 < \rho' < \rho$,
there  is $K_\eps:\G_{r_0}\rightarrow\A_{\rho'}$, $\mu_\eps:\G_{r_0}\rightarrow\complex^d$, analytic in the interior of
$\G_{r_0}$ taking values in $\A_{\rho'}$
which extends continuously to the boundary of $\G_{r_0}$,
such that for $\eps\in\G_{r_0}$ the invariance equation is satisfied exactly:
$$
f_{\mu_\eps, \eps} \circ K_\eps - K_\eps \circ T_\omega=0\ .
$$
Moreover, the above solutions admit the formal series in A) as an asymptotic expansion, namely for
$0<\rho'<\rho$, $N \in \nat$, one has:
$$
||K^{[\le N]}_\eps -  K_\eps||_{\rho'}  \le C_N |\eps|^{N+1} \ ,\qquad
|\mu^{[\le N]}_\eps -  \mu_\eps|  \le C_N |\eps|^{N+1}\ .
$$

\end{theorem}

The proof of Theorem~\ref{thm:domain} is very similar to the main theorem in \cite{CCLdomain}
and we sketch in the following Sections just the main ingredients of the proof.

\subsection{Proof of Theorem~\ref{thm:domain}}

By hypothesis we can find an embedding $K_0:\torus^d\rightarrow\M$
satisfying \eqref{inv0}. Since $f_{\mu_0,0}$ is symplectic, one has
$$
f_{\mu_0,0}^*\Omega=\Omega\ .
$$
Substituting $K_\eps^{[\leq N]}$, $\mu_\eps^{[\leq N]}$ in \equ{inv0},
and equating the coefficients with the same
power of $\eps$, we obtain recursive relations for the terms $K_j$, $\mu_j$.

Indeed, the first order in $\eps$ is
$$
(Df_{\mu_0,0}\circ K_0)K_1-K_1\circ T_\omega +(D_\mu f_{\mu_0,0}\circ K_0)\mu_1+D_\eps f_{\mu_0,0}\circ K_0=0\ .
$$
More generally, for any order $j$, we obtain that
\beq{ordj}
(Df_{\mu_0,0}\circ K_0)K_j-K_j \circ T_\omega +(D_\mu f_{\mu_0,0}\circ K_0)\mu_j = R_j\ ,
\eeq
where $R_j$   is a polynomial in $K_0,\, \ldots \, , K_{j-1}$,
$\mu_1,\, \ldots \, , \mu_{j-1}$,
$D_\eps f_{\mu_0, \eps =0 } \circ K_0, \, \ldots \, ,
D_\eps^j f_{\mu_0, \eps =0 } \circ K_0$.

It is important to note that the coefficients
multiplying the unknowns  $K_j$, $\mu_j$ in \eqref{ordj}  are
$(Df_{\mu_0,0}\circ K_0)$ and $(D_\mu f_{\mu_0,0}\circ K_0)$,
respectively. These coefficients do not depend on $j$
and they can be evaluated on the zero order approximation.

Using that the zero order corresponds to a whiskered
torus in the Hamiltonian case which is exactly invariant,
 we obtain that the
coefficient $(Df_{\mu_0,0}\circ K_0)$ is exactly reducible
(see Lemma~\ref{approximate-reducibility} and the discussion
around it).

That is, defining the matrix $M_0$ as in \eqref{M}, we have
in the center direction,
\beq{centerdir}
Df_{\mu_0,0}\circ K_0 M_0(\th) = M_0(\th + \omega)
\begin{pmatrix} \Id_d & S_0(\th)\\ 0&\lambda\Id_d \end{pmatrix}\ .
\eeq
Note that, in our case, using that $K_0$ is exactly invariant,
there is no $\E_R$ term (see \eqref{approxredestimates}).

The equation \eqref{ordj} for $K_j$, $\mu_j$ can be solved
taking advantage of the exact reducibility. It is the same
procedure that we had in Section~\ref{sec:centerspace},
but now  the reducibility equation holds exactly.

 Let us write $K_j(\th)=M_0(\th) W_j(\th)$. Using \equ{ordj},
we have:
$$
(Df_{\mu_0,0}\circ K_0)\, M_0W_j-M_0\circ T_\omega\ W_j\circ T_\omega+
(D_\mu f_{\mu_0,0}\circ K_0)\mu_j=R_j\ ;
$$
using \equ{centerdir} it follows that
$$
M_0\circ T_\omega\ \begin{pmatrix} \Id_d & S_0(\th)\\ 0&\lambda\Id_d \end{pmatrix} W_j
-M_0\circ T_\omega\ W_j\circ T_\omega+
(D_\mu f_{\mu_0,0}\circ K_0)\mu_j=R_j\ ,
$$
namely
$$
\begin{pmatrix} \Id_d & S_0(\th)\\ 0&\lambda\Id_d \end{pmatrix} W_j-W_j\circ T_\omega
+(M_0\circ T_\omega)^{-1}(D_\mu f_{\mu_0,0}\circ K_0)\mu_j=(M_0\circ T_\omega)^{-1}R_j\ .
$$
In components, namely taking the first $d$ rows and the last $d$ rows, say $W=(W_{j1}|W_{j2})$,
denoting by $[\cdot]_k$ the $k$-th component, we have:
\beqano
\lambda W_{j2}-W_{j2}\circ T_\omega+[(M_0\circ T_\omega)^{-1}(D_\mu f_{\mu_0,0}\circ K_0)]_2\mu_j
&=&[(M_0\circ T_\omega)^{-1}R_j]_2\nonumber\\
W_{j1}-W_{j1}\circ T_\omega+S_0W_{j2}+[(M_0\circ T_\omega)^{-1}(D_\mu f_{\mu_0,0}\circ K_0)]_1\mu_j
&=&[(M_0\circ T_\omega)^{-1}R_j]_1\ .
\eeqano
Define $\widetilde E_j=(M_0\circ T_\omega)^{-1} R_j$, $\widetilde A_0=(M_0\circ T_\omega)^{-1}(D_\mu f_{\mu_0,0}\circ K_0)$,
we obtain:
\beqa{WW}
\lambda W_{j2}-W_{j2}\circ T_\omega+\widetilde A_{20}\mu_j&=&\widetilde E_{2j}\nonumber\\
W_{j1}-W_{j1}\circ T_\omega+\widetilde A_{10}\mu_j&=&\widetilde E_{1j}-S_0 W_{j2}\ .
\eeqa
Taking the average of the first and second equation in \equ{WW}, one has
\beqa{aveWW}
(\lambda-1)\overline{W_{j2}}&=&\overline{\widetilde E_{2j}}-\overline{\widetilde A_{20}}\ \mu_j\nonumber\\
\overline{\widetilde A_{10}}\ \mu_j&=&\overline{\widetilde E_{1j}}-\overline{S_0 W_{j2}}\ .
\eeqa
Defining $W_{j2}=\overline{W_{j2}}+(W_{j2})^o$, one has $\overline{S_0 W_{j2}}=\overline{S_0}\overline{W_{j2}}
+\overline{S_0 (W_{j2})^o}$. Since $W_{j2}$ is an affine function of $\mu_j$,
let $(W_{j2})^o=(B_{a0})^o+(B_{b0})^o\mu_j$, where $(B_{a0})^o$, $(B_{b0})^o$ are solutions of
\beqano
\lambda (B_{a0})^o-(B_{a0})^o\circ T_\omega&=&(\widetilde E_{2j})^o\nonumber\\
\lambda (B_{b0})^o-(B_{b0})^o\circ T_\omega&=&-(\widetilde A_{20})^o\ .
\eeqano
Using the second of \equ{aveWW} we have
$$
\overline{\widetilde A_{10}}\ \mu_j+\overline{S_0}\ \overline{W_{j2}}+
\overline{S_0(B_{b0})^o}\mu_j=-\overline{S_0(B_{a0})^o}+\overline{\widetilde E_{1j}}\ ,
$$
so that we have:
$$
\begin{pmatrix} \overline{S_0} & \overline{\widetilde A_{10}}+\overline{S_0(B_{b0})^o}\\
(\lambda-1)\Id_d&\overline{\widetilde A_{20}} \end{pmatrix}
\begin{pmatrix} \overline{W_{j2}} \\
\mu_j \end{pmatrix}=
\begin{pmatrix} -\overline{S_0(B_{a0})^o}+\overline{\widetilde E_{1j}} \\
\overline{\widetilde E_{2j}} \end{pmatrix}\ ,
$$
which can be solved for $\overline{W_{j2}}$, $\mu_j$ under the non-degeneracy condition
\equ{non-degeneracyW}. This concludes the  proof of B.1).

\bigskip
To establish B.2), we just observe that, once we have
the expansions in powers  of $\eps$
for the $K_j$, $\mu_j$, we can obtain the power series expansion in
$\eps$ for
$Df_{\mu_0,0}\circ K_j$, hence for the $\gamma$'s, which are
obtained from $Df_{\mu_0,0}\circ K_0$ by taking projections on
fixed spaces.
We also note that, if we take projections over the
original splittings, we have that $A^{\sigma, (0)} = 0$.
Also, the approximate invariance of the initial splitting
tells that $\gamma^{\st, \hat \st, (0)}, \gamma^{\hat \st, \st, (0)},
\gamma^{\un, \hat \un, (0)}, \gamma^{\hat \un, \un,(0)}$ are small.

If we substitute the expansions for $A^\sigma$ and
equate terms of order $\eps^j$ in
\eqref{invariance2-s}, \eqref{invariance2-u},
we obtain that  these equations are satisfied in the sense of power
series if and only if for all $j$ we have
\begin{equation} \label{invariance2-orderj}
\begin{split}
& (\gamma_{\th,0}^{\hat \st, \hat \st} )^{-1} \left[ A^{\st}_{\th + \omega,j}
\gamma^{\st, \st}_{\th,0}  - \gamma_{\th,0}^{\hat \st, \st} \right] =
A^{\st}_{\th,j} + R_{\st,\theta, j}^{I} \\
& [\gamma_{\th-\omega,0}^{\st, \hat \st}+\gamma_{\th-\omega,0}^{\st,\st}A^{\hat\st}_{\th - \omega,j}]
(\gamma_{\th-\omega,0}^{\hat \st, \hat \st} )^{-1}=
A^{\hat\st}_{\th,j} + R_{\st,\theta, j}^{II} \\
& [\gamma_{\th-\omega,0}^{\hat\un,\un}+\gamma_{\th-\omega,0}^{\hat\un,\hat\un}A^{\un}_{\th - \omega,j}]
(\gamma_{\th-\omega,0}^{\un, \un} )^{-1}=
A^{\un}_{\th,j} + R_{\un,\theta, j}^{I} \\
& (\gamma_{\th,0}^{\un, \un} )^{-1} \left[ A^{\hat\un}_{\th + \omega,j}
\gamma^{\hat\un,\hat\un}_{\th,0}  - \gamma_{\th,0}^{\un, \hat\un} \right] =
A^{\hat\un}_{\th,j} + R_{\un,\theta, j}^{II} \ ,
\end{split}
\end{equation}
where the $R_{\sigma,\th, j}^{I}$, $R_{\sigma,\th, j}^{II}$ are explicit polynomial expressions
involving only $A^{\sigma}_{\th,l}$ for $l \le j -1$. Notice that only the $\gamma$ coefficients enter.

\begin{remark}
We note that all the equations in \eqref{invariance2-orderj}
have the form of fixed points of an operator, some of whose power
is a contraction.   Note that finding the $A^{\sigma, (j)}$ does
not entail loosing any domain of analyticity.
Because the structure of the $R_{\sigma,\th,j}^{(k)}$, we can
solve the equations recursively and proceed to find solutions
which have the same domain as the $R_{\sigma,\th,j}^{(k)}$, $\gamma$. The domain of
the $\gamma$ can be taken to be as close to the domain of $K_0$
as desired.
\end{remark}

We fix $N$ sufficiently large (say 5). Then, by
choosing $\eps$ sufficiently small, all assumptions of
Theorem~\ref{thm:domain} are satisfied.
We note that if we take the approximate solution as
$(K,\mu)=(K_\eps^{[\leq  N]},\mu_\eps^{[\leq N]})$
and as the approximate splitting the results of
the expansion, we have
that
$\E_h, \E \le C_{N,\rho} |\eps|^{N+1}$
for an analytic norm in a fixed radius $\rho$ slightly smaller
than the analytic domain of the original radius.

Since we assume that the torus in the Hamiltonian case is
non degenerate, we get that the non-degeneracy conditions are
uniform for $|\eps|$ small. If we choose a $\delta < \rho/2$, we
can obtain  the smallness conditions of Theorem~\ref{whiskered}
for small enough $|\eps|$.
Also the assumption of  the range of $K$ being inside the domain of
$f_{\mu, \eps}$ are uniform for $|\eps|$ small enough.

We also note that the non-degeneracy conditions for Theorem~\ref{whiskered}
are
uniform in the sets
\begin{equation}\label{sets}
\{ a_- < |\eps| < a_+\} \cap \G
\end{equation}
for sufficiently small $0< a_- < a_+$. As we argued before,
the condition that $\eps$ is small enough ensures that the non-degeneracy conditions
are uniform.  The intersection with $\G$ ensures that the Diophantine
properties are uniform.

Therefore, the iterative procedure in the proof of
Theorem~\ref{whiskered} is uniform for all the $\eps$ in the
sets \eqref{sets}.  We also recall that the iterative step
to prove Theorem~\ref{whiskered} consists just in performing
algebraic operations, shifting functions and solving cohomology
equations. In all these operations, it is clear that if
the data depend analytically on $\eps$, so does the correction.

Putting together the two remarks above, we obtain that the procedure
leads to a sequence of functions all of which are analytic in $\eps$
and which converge uniformly in sets of the form \eqref{sets}.
Therefore the solution will be analytic in sets of the form
\eqref{sets}. Due to the local uniqueness of the solution,
we obtain that the suitably normalized  solutions in different patches
that overlap have to agree.

\begin{remark}
Note that the conditions of smallness in $\eps$, so
that we can apply Theorem~\ref{whiskered} to the truncated series,
depend on the size of the coefficients and the domain
loss. It would be interesting to try to optimize the choices of
the orders of truncation and the domain losses depending on
$\eps$ -- similar calculations are often  done in the study of
Birkhoff Normal forms.
\end{remark}

We conclude by writing the Lindstedt series in B.2) associated to \eqref{invariance2-s}, \eqref{invariance2-u}.
We start from the first of \eqref{invariance2-s}. We expand $A^s$ as
$$
A_\th^s=\sum_{j=0}^\infty \eps^j A_{\th,j}^s\ ,
$$
and we expand $\gamma_\th^{\sigma\eta}$ as
\beqano
\gamma_\th^{\sigma\eta}=\sum_{j=0}^\infty \eps^j \gamma_{\th,j}^{\sigma\eta}\qquad if\ \sigma=\eta\ ,\nonumber\\
\gamma_\th^{\sigma\eta}=\sum_{j=1}^\infty \eps^j \gamma_{\th,j}^{\sigma\eta}\qquad if\ \sigma\not=\eta\ .
\eeqano
Inserting the above  series expansions in the first of \eqref{invariance2-orderj} and equating same orders of
$\eps^j$, $j\geq 1$, one obtains:
\beq{A1bis}
A_{\th,j}^s=(\gamma_{\th,0}^{\hat s,\hat s})^{-1}A_{\th+\omega,j}^s \gamma_{\th,0}^{s,s}
-(\gamma_{\th,0}^{\hat s,\hat s})^{-1} \gamma_{\th,j}^{\hat s,s}-R_{s,\th,j}^I(A_{\th,1}^s,\ldots,A_{\th,j-1}^s)\ .
\eeq
Equation \equ{A1bis} can be solved by iteration to obtain:
\beqano
A_{\th,j}^s&=&(\gamma_{\th,0}^{\hat s,\hat s})^{-1}\ \{(\gamma_{\th+\omega,0}^{\hat s,\hat s})^{-1}
A_{\th+2\omega,j}^s \gamma_{\th+\omega,0}^{s,s}-(\gamma_{\th+\omega,0}^{\hat s,\hat s})^{-1}
\gamma_{\th+\omega,j}^{\hat s,s}-R_{s,\th+\omega,j}^I\}\ \gamma_{\th,0}^{s,s}\nonumber\\
&-&(\gamma_{\th,0}^{\hat s,\hat s})^{-1}\ \gamma_{\th,j}^{\hat s,s}+R_{s,\th,j}^I\nonumber\\
&=&\sum_{k=0}^\infty 1\times \ldots\times (\gamma_{\theta+(k-1)\omega,0}^{\hat s,\hat s})^{-1}\
R_{s,\th+k\omega,j}^I\ (\gamma_{\theta+(k-1)\omega,0}^{s,s})\times \ldots\times 1\nonumber\\
&-&\sum_{k=0}^\infty (\gamma_{\th,0}^{\hat s,\hat s})^{-1}\times \ldots\times
(\gamma_{\theta+(k-1)\omega,0}^{\hat s,\hat s})^{-1}
(\gamma_{\theta+k\omega,0}^{\hat s,\hat s})^{-1}\ \gamma_{\theta+k\omega,j}^{\hat s,s}\
(\gamma_{\theta+(k-1)\omega,0}^{s,s})\times \ldots\times 1\ ,
\eeqano
where $1\times \ldots\times (\gamma_{\theta+(k-1)\omega,0}^{\hat s,\hat s})^{-1}=1$ and
$(\gamma_{\theta+(k-1)\omega,0}^{s,s})\times \ldots\times 1=1$ for $k=0$.
We remark that the products of $1\times \ldots\times (\gamma_{\theta+(k-1)\omega,0}^{\hat s,\hat s})^{-1}$
by $(\gamma_{\theta+(k-1)\omega,0}^{s,s})\times \ldots\times 1$ are contractions.
The other equations in \eqref{invariance2-s}, \eqref{invariance2-u} are treated in the same way;
we omit the details and provide just the final results.
Analogously, the second equation in \eqref{invariance2-s} gives the following solution:
\beqano
A_{\th,j}^{\hat s}
&=&-\sum_{k=0}^\infty 1\times \ldots\times \gamma_{\theta-k\omega,0}^{s,s}\
R_{s,\th-k\omega,j}^{II}\ (\gamma_{\theta-k\omega,0}^{\hat s,\hat s})^{-1}\times \ldots\times
(\gamma_{\theta-\omega,0}^{\hat s,\hat s})^{-1}\nonumber\\
&+&\sum_{k=0}^\infty 1\times \ldots\times \gamma_{\theta-k\omega,0}^{s,s}\ \gamma_{\theta-(k+1)\omega,j}^{s,\hat s}
\ (\gamma_{\theta-(k+1)\omega,0}^{\hat s,\hat s})^{-1}\times \ldots\times
(\gamma_{\theta-\omega,0}^{\hat s,\hat s})^{-1}\ ,
\eeqano
where $1\times \ldots\times \gamma_{\theta-k\omega,0}^{s,s}=1$ and
$(\gamma_{\theta-k\omega,0}^{\hat s,\hat s})^{-1}\times \ldots\times
(\gamma_{\theta-\omega,0}^{\hat s,\hat s})^{-1}=1$ for $k=0$.
As for the first equation in \eqref{invariance2-u}, we obtain:
\beqano
A_{\th,j}^{u}
&=&-\sum_{k=0}^\infty 1\times \ldots\times \gamma_{\theta-k\omega,0}^{\hat u,\hat u}\
R_{u,\th-k\omega,j}^{I}\ (\gamma_{\theta-k\omega,0}^{u,u})^{-1}\times \ldots\times (\gamma_{\theta-\omega,0}^{u,u})^{-1}\nonumber\\
&+&\sum_{k=0}^\infty 1\times \ldots\times \gamma_{\theta-k\omega,0}^{\hat u,\hat u}
\ \gamma_{\theta-(k+1)\omega,j}^{\hat u,u}\ (\gamma_{\theta-(k+1)\omega,j}^{u,u})^{-1}\times \ldots\times
(\gamma_{\theta-\omega,j}^{u,u})^{-1}\ ,
\eeqano
where $1\times \ldots\times \gamma_{\theta-k\omega,0}^{\hat u,\hat u}=1$ and
$(\gamma_{\theta-k\omega,0}^{u,u})^{-1}\times \ldots\times
(\gamma_{\theta-\omega,0}^{u,u})^{-1}=1$ for $k=0$.
The second equation in \eqref{invariance2-u} is solved as follows:
\beqano
A_{\th,j}^{\hat u}
&=&-\sum_{k=0}^\infty (\gamma_{\theta,0}^{u,u})^{-1}\times \ldots\times (\gamma_{\theta+(k-1)\omega,0}^{u,u})^{-1}
R_{u,\th+k\omega,j}^{II}\ (\gamma_{\theta+(k-1)\omega,j}^{\hat u,\hat u})\times \ldots\times 1\nonumber\\
&-&\sum_{k=0}^\infty (\gamma_{\theta,0}^{u,u})^{-1}\times \ldots\times (\gamma_{\theta+k\omega,0}^{u,u})^{-1}
\gamma_{\theta+k\omega,j}^{u,\hat u}\ \gamma_{\theta+(k-1)\omega,0}^{\hat u,\hat u}\times \ldots\times 1\ ,
\eeqano
where $(\gamma_{\theta,0}^{u,u})^{-1}\times \ldots\times
(\gamma_{\theta+(k-1)\omega,0}^{u,u})^{-1}=1$  and
$\gamma_{\theta+(k-1)\omega,0}^{\hat u,\hat u}\times \ldots\times 1=1$ for $k=0$.

\appendix

\section{Proof of the closing lemma and its consequences}\label{app:closing}

In this Appendix we provide the proof of Lemma~\ref{lem:closing} and
of some of its consequences, precisely Lemma~\ref{general} and
Lemma~\ref{precise}.

\subsection{Proof of Lemma~\ref{lem:closing}}

Assume  that we have a  splitting in the neighborhood of the
reference splitting, so that we can describe
the splitting by the functions $A_\th^\sigma$ as in \eqref{graph}.
Let $\gamma$ be a cocycle over a rotation.

Our first task is to formulate a functional equation for
the $A_\th^\sigma$ that is equivalent to their graphs being invariant.
Then, we will transform this equation into a contraction mapping theorem.
These constructions are very standard in the theory of hyperbolic systems
(\cite{Anosov69, HirschPS77}).

We see that for a vector
in  the graph of $A^\sigma_\th$ (which we write as
$x + A^\sigma_\th x$ with $x \in E^\sigma_\th$),
we have that its image under $\gamma_\th=\gamma(\theta)$,
 expressed in components, is:
\begin{equation}\label{newpoint}
\gamma_\th( x + A^\sigma_\th x) =
\left( \gamma^{\sigma, \sigma}_\th x +
\gamma^{\sigma, \hat \sigma}_\th  A^\sigma_\th x  \right)
+ \left( \gamma^{\hat \sigma, \sigma}_\th x
+ \gamma^{\hat \sigma, \hat \sigma}_\th  A^\sigma_\th x \right)\ .
\end{equation}

The  point \eqref{newpoint} is in the graph
of $A_{\th + \omega}^\sigma$ for all $x \in E^\sigma$,
if and only if the  matrices $A_\th^\sigma$ satisfy:
\begin{equation} \label{invariance1}
A^\sigma_{\th + \omega} \big(
\gamma^{\sigma, \sigma}_\th  +
\gamma^{\sigma, \hat \sigma}_\th  A^\sigma_\th \big)
= \gamma^{\hat \sigma, \sigma}_\th
+ \gamma^{\hat \sigma, \hat \sigma}_\th  A^\sigma_\th\ .
\end{equation}
Conversely, since the derivation of
\eqref{invariance1} is just algebra, we
see that if \eqref{invariance1} holds, all the points in
the graph of $A_\th^\sigma$ will be transformed into maps
in the graph of $A_{\th + \omega}^\sigma$.

Hence, our treatment will be based on discussing \eqref{invariance1},
manipulating it algebraically till it becomes a contraction.
Note that \eqref{invariance1} is a very general calculation and that it
applies to any dichotomy.

To guess the algebraic transformations
that make \eqref{invariance1} into a contraction in our cases,
it is useful to remark that $\gamma^{\sigma, \hat \sigma}$,
$\gamma^{\hat \sigma, \sigma}$  will be assumed to be
sufficiently small and that the cocycles generated
by $\gamma^{\sigma, \sigma}$ and $\left(\gamma^{\hat \sigma, \hat \sigma}\right)^{-1}$
have different contraction/growth rates, see \eqref{growthrates}.

Hence, \eqref{invariance1} is
heuristically a small perturbation of
\begin{equation} \label{invariance1-simple}
A^\sigma_{\th + \omega}
\gamma^{\sigma, \sigma}_\th
= \gamma^{\hat \sigma, \hat \sigma}_\th  A^\sigma_\th\ .
\end{equation}

The rearrangements of the equation \eqref{invariance1}  that are useful to
reformulate it as a contraction  are different depending on the cases
we consider.
Note that we need to study two dichotomies:
$\sigma = s, \hat \sigma = \hat s$ and  $\sigma = u, \hat \sigma = \hat u$.
Hence, we will need two equations for each of the two dichotomies.

The manipulations needed can be understood by looking at
\eqref{invariance1-simple}. We want to isolate one of the $A^\sigma$
appearing in \eqref{invariance1-simple} in such a way that the
RHS is a contraction. Once we get that the main part is a contraction,
it will follow that the (arbitrarily)  small terms omitted from \eqref{invariance1-simple}
do not affect the contraction properties.

For the dichotomy between $s, \hat s$ spaces we use:
\begin{equation} \label{invariance2-sbis}
\begin{split}
& (\gamma_\th^{\hat \st, \hat \st } )^{-1} \left[ A^\st_{\th + \omega}  \big(
\gamma^{\st, \st}_\th  +
\gamma^{\st, \hat \st}_\th  A^\st_\th \big)-
\gamma^{\hat \st, \st}_\th \right]  = A^\st_\th\ ,\\
&\left[ \gamma^{\st, \hat \st}_{\th - \omega} +
 \gamma^{\st, \st}_{\th-\omega}A^{\hat \st}_{\th -
\omega}
- A^{\hat \st}_\th \gamma_{\th -\omega}^{\hat s, s} A^{\hat \st}_{\th -\omega}\right]
 \left( \gamma^{\hat \st, \hat \st}_{\th -\omega}\right)^{-1}
= A^{\hat \st}_\th\ .
\end{split}
\end{equation}

\medskip

For the dichotomy corresponding to $\un, \hat \un$, we obtain the pair of equations:

\begin{equation} \label{invariance2-ubis}
\begin{split}
& \left[-A^\un_{\th} \gamma^{\un, \hat \un}_{\th -\omega}A^\un_{\th -\omega}
+ \gamma^{\hat u, u}_{\th - \omega}
+ \gamma^{\hat u, \hat u}_{\th - \omega} A^\un_{\th - \omega} \right]
(\gamma^{\un,\un}_{\th - \omega})^{-1} =
 A^\un_{\th}  \\
& (\gamma_\th^{\un, \un } )^{-1} \left[ A^{\hat \un}_{\th + \omega}  \big(
\gamma^{\hat \un, \hat \un}_\th  +
\gamma^{\hat \un, \un}_\th  A^{\hat \un}_\th \big)-
\gamma^{\un, \hat \un}_\th \right]  = A^{\hat \un}_\th\ .
\end{split}
\end{equation}
The two systems \eqref{invariance2-sbis} and \eqref{invariance2-ubis} can be dealt with by the same methods.
So, we will only discuss \eqref{invariance2-ubis}.
It will be important to note that the estimates that we obtain for
the solutions depend only on the constants $C_0$ and the rates
entering into \eqref{growthrates}.

We realize that, if we eliminate from \eqref{invariance2-ubis} the blocks of $\gamma$  that
can be made small by assuming that the splitting is
almost invariant, then we are led to consider the fixed point of the operator $\N_0$ defined as
\beqano
\N_0( A^\un, A^{\hat \un})_\th =
\begin{pmatrix}
&  \gamma^{\hat \un, \hat \un }_{\th-\omega} A^\un_{\th - \omega}
(\gamma_{\th-\omega}^{\un, \un})^{-1} \\
& (\gamma_{\th}^{\un, \un})^{-1}  A^{\hat \un}_{\th+\omega}
\gamma^{ \hat \un, \hat  \un}_{\th}
\end{pmatrix}\ .
\eeqano

In the following we present some (rather arbitrary choices)
that work.

Note that $\N_0$ is a linear operator
and that powers of it are obtained by multiplying the arguments
by cocycles in the right and in the left (and by shifting the
arguments).

Due to the rate conditions, there exists an $L > 0$ such
that, by iterating $\N_0$, $L$ times, we can make the
Lipschitz constant of the iterate (in the analytic norm)
smaller than $1/2$,  $\Lip(\N_0^L) < 1/2$ -- where the
Lipschitz constant is in the space $\A_\rho$.

If we consider the ball $\| A_\th^\st\|_\rho, \| A_\th^\un \|_\rho < M_1$ for some $M_1>0$,
we can find smallness conditions on
$\|\gamma^{\un, \hat \un}\|_\rho, \| \gamma^{\hat \un, \un}\|_\rho$, so that the
contraction of $\N_0$ in this ball is smaller than $3/4$.

We recall that the splitting $E$ is $\eta$-approximately invariant and that the
distance between the splittings can be measured by \equ{distance2}.
Using that $(\N_0)^L$ is a contraction, it follows that $\N_0$ has a unique fixed point. We conclude with the
standard fixed point estimates which, together with \equ{approximatelyinvariant}, lead to
$$
\max_\sigma \|A_\th^\sigma\|\leq C\eta\ .
$$

\subsection{Proof of Lemma~\ref{general}}

  There exists $N\in \mathbb N$ such that $\|\Gamma^N\|_\rho \leq {1 \over 4}$ (indeed just
  take $N =\ln({1\over 4C_0})/\ln(\xi)$). Now, there exists an $\eps^* > 0$ so that
  $\|\gamma - \tilde \gamma\|_\rho \leq \eps^*$ implies that
  \[\|\widetilde \Gamma^N\|_\rho \leq {1 \over 2}\ .\]
We recall that $\widetilde \Gamma_0^\ell=\widetilde \Gamma^\ell$, we write
$\widetilde \Gamma_{Nk}^{\ell+Nk}=\widetilde \Gamma_0^{\ell+Nk-Nk}\circ T_{Nk\omega}$, so that
$\widetilde\Gamma_0^{\ell+Nk}=\widetilde\Gamma_{Nk}^{\ell+Nk}\widetilde\Gamma_0^{Nk}=
\widetilde\Gamma_0^{\ell}\circ T_{Nk\omega}\widetilde\Gamma_0^{Nk}$.
Then, we have:
  \[\|\widetilde \Gamma^{Nk+\ell}\|_\rho \leq \left({1 \over 2}\right)^k \sup_{0<\ell \leq N}
   \|\widetilde \Gamma^\ell\|_\rho = \left ( {1 \over 2^{1/N}} \right )^{Nk+ \ell}
     {1 \over 2^{-\ell/N}} \sup_{0<\ell \leq N} \| \widetilde \Gamma^\ell \|_\rho\ .\]

     Hence we obtain the desired result with $\tilde \xi = (1/2)^{1/N}$, $\tilde C_0 = 2
     \sup_{0<\ell \leq N}\| \widetilde \Gamma^\ell \|_\rho$.

\subsection{Proof of Lemma~\ref{precise}}
The method of proof is very similar to perturbation arguments of semigroups (\cite{HilleP}).

The first observation is that, using the operators $A^\sigma_\th$
as in \eqref{graph}, we
can identify the approximately invariant
spaces $E^\sigma_\th$ with the invariant ones $\widetilde E^\sigma_\th$.
Let the invariant cocycle be
\[
\tgamma^{\sigma,\sigma}_\th: \widetilde E^\sigma_\th
\rightarrow \widetilde E^\sigma_{\th + \omega}\ .
\]


We will only consider the case of the forward cocycles.
The case of the inverse requires only to change $\omega$ to
$-\omega$ and to consider the inverse cocycle.

Adding and subtracting appropriate terms, using the notation in \eqref{product}-\eqref{cocyclerotation}, we have
\begin{equation}\label{duhameldiscrete}
\begin{split}
\tGamma^n_\th & \equiv \tgamma^{\sigma,\sigma}_{\th + (k-1) \omega}  \cdots
\tgamma^{\sigma,\sigma}_\th \\
&  =  \gamma^{\sigma,\sigma}_{\th + (k-1) \omega}  \cdots
\gamma^{\sigma,\sigma}_\th +
\sum_{j = 0}^{k-1} \gamma^{\sigma,\sigma}_{\th + (k-1) \omega}  \cdots
\gamma^{\sigma,\sigma}_{\th + (j+1) \omega}
 \left[
\tgamma^{\sigma, \sigma}_{\th + j \omega} -
\gamma^{\sigma, \sigma}_{\th + j \omega} \right]
\tgamma^{\sigma, \sigma}_{\th + (j -1) \omega} \cdots \tgamma^{\sigma, \sigma}_{\th} \\
& =
\Gamma^k_\th + \sum_{j =0}^{k-1} \Gamma^k_{j+1}
 \left[
\tgamma^{\sigma, \sigma}_{\th + j \omega} -
\gamma^{\sigma, \sigma}_{\th + j \omega} \right] {\tGamma}^{j}_\th\ ,
\end{split}
\end{equation}
where we define $ {\tGamma}^{0}_\th$ and  $ {\Gamma}^{0}_\th$ to be the identity, and we intend that
$\tgamma^{\sigma,\sigma}_{\th + (j-1) \omega}  \cdots \tgamma^{\sigma,\sigma}_\th=0$ for $j=0$ and
$\gamma^{\sigma,\sigma}_{\th + (k-1) \omega}  \cdots \gamma^{\sigma,\sigma}_{\th+(j+1)\omega}=0$ for
$j=k-1$. Note
that \eqref{duhameldiscrete}  is a discrete version of
Duhamel formula, so that the rest of
the argument is very similar to the arguments
in perturbation theory of semigroups.

We will consider \eqref{duhameldiscrete} as a fixed point equation
for $\tGamma^k$ lying in an appropriate space of sequences
with an appropriate norm.
We therefore write \eqref{duhameldiscrete} as
\beq{fixedp}
\tGamma = \Gamma + \LL \tGamma.
\eeq
We will
show that the operator $\LL$ given by
\[
(\LL\tGamma)^k \equiv \sum_{j =0}^{k-1} \Gamma^k_{j+1}
 \left[
\tgamma^{\sigma, \sigma}_{\th + j \omega} -
\gamma^{\sigma, \sigma}_{\th + j \omega} \right] {\tGamma}^{j}_\th
\]
is a contraction in a space of
sequences endowed with a norm that captures the rate.

Iterating \eqref{fixedp} we obtain
\beq{Dyson}
\tGamma = \Gamma + \LL \Gamma + \LL^2 \Gamma + ... + \LL^k \Gamma + ...
\eeq
The above treatment is very similar to perturbation theory of
semigroups in the Physics literature; equation \eqref{Dyson} is known as the Dyson expansion.

To study \eqref{fixedp} we introduce appropriate norms in spaces of sequences of operators, precisely:
$$
\|\tGamma\|_{\txi} = \sup_{k \in \nat} \left({\txi}^{-k} \| \tGamma^k\|_\rho\right)\ ;
$$
we suppress the $\A_\rho$ from the notation for $\|\Gamma\|$
since it will be fixed in this argument. On the other hand,
the $\xi$ will be important for us.

We fix $\xi < \txi$ and estimate $\LL$ in the norm.
We introduce the quantity $a$ as
\[
a  \equiv
\|
\tgamma^{\sigma, \sigma} -
\gamma^{\sigma, \sigma} \|_\rho\ .
\]
We note that for any choice of $\txi$ for which we can show that $\LL$ is a contraction, we show that
there is a solution of \eqref{fixedp} in the space of functions with rate $\txi$ (and of course, that
the solution is unique). Obviously the $\tGamma$ produced by the recursion \eqref{Dyson} is a solution.
Therefore, we show that the \eqref{Dyson} has growth with exponent $\txi$:
\beqano
  \|(\LL\tGamma)^k \|_\rho \txi ^{-k}
  & \leq & \txi^{-k} \sum_{j = 0}^{k-1} \|\Gamma_{j+1}^k \|_\rho a \|\tGamma^j\|_\rho \\
  &  \leq & \txi^{-k} \sum_{j = 0}^{k-1} C_0 \xi^{k-(j+1)} a \txi^j \|\tGamma\|_\txi \\
  &  = & C_0 a \sum_{j = 0}^{k-1} \xi^{k-j}\txi^{-k+j} \|\tGamma\|_\txi (\xi^{-1}) \\
  & \leq & C_0 a \txi^{-1} {1 \over 1-{\xi}/\txi} \|\tGamma\|_\txi\ .
\eeqano
Hence, if we take $\xi/\txi<{1\over 2}$ and
$$
C_0\ a\ \txi^{-1} \leq 1/4 ,
$$
we ensure that $\|\LL\|$ is a contraction.

The estimate of the constant $\tilde C_0$ follows from the fact that
$\|\LL \| = C_0 a \txi^{-1} {1 \over 1-{\xi}/\txi} < 1/2$.

\section{Non-Euclidean manifolds}
\label{sec:precisions}
In this Section, we discuss how one can adapt the results
for non-Euclidean manifolds. For lower dimensional tori, this is
interesting because there are examples with lower dimensional tori
with a non-trivial topology of the neighborhoods (in Lagrangian
tori, this cannot happen).

In non-Euclidean manifolds, we run into two problems.

One is that for approximately invariant tori,
$\T_{f_\mu\circ K (\th) )} \ne \T_{ K(\th + \omega)}$
and, hence for tori in a non-Euclidean manifold we cannot write
$Df_\mu \circ K(\th) = DK(\th + \omega)$, which
is a very suggestive notation for effects of iterations.
This problem has appeared very frequently in dynamics.
A standard way of fixing the problem is to construct
connectors $S_x^y$ (\cite{HirschPPS70}) which identify the tangent
spaces of close enough points\footnote{
A very natural definition of connectors is
$y = \exp_x(v)$ with $v$ sufficiently small; we fix $x$ and
the derivative of the exponential mapping at $v$ identifies
the tangent  space  at $x$ with the tangent space at $y$.
The chain rule gives that $S_y^z S_x^y = S_x^z$, when
$x$ is sufficiently close to $y,z$ and $y, z$ are sufficiently
close.}.
Hence, the cocycles  one should consider  are
\[
 Df_\mu\circ K(\th + n \omega) S_{f\circ K(\th +(n-1)\omega)}^{K(\th + n \omega)}
 \times \cdots \times Df_\mu \circ K(\th + \omega)S_{f_\mu \circ K(\th)}^{K(\th + \omega)}\ .
\]
Note, however, that if the bundles are not trivial, these cannot be
identified with matrix cocycles.

A second problem  is that $f_\mu \circ K(\th) $ transforms
$\Omega_{K(\th)}$ into a multiple of $\Omega_{f_\mu(K(\th))}$.
Even if one identifies the tangent spaces, it is not clear
what are the geometric properties of the product \eqref{product}.

There are standard ways of correcting this. For example
\cite{GonzalezL08} uses  the Global Darboux theorem
from \cite{Moser65, Weinstein73} to change slightly the map in such
a way that approximate cocycles are exactly conformally symplectic.
As shown in \cite{GonzalezL08} these changes do not alter the quadratic
convergence of the algorithm because the size of the required
changes can be bounded by the error in the invariance equation.

In our case, this problem appears only in Section~\ref{sec:geometry}.

\def\cprime{$'$} \def\cprime{$'$} \def\cprime{$'$} \def\cprime{$'$}


\end{document}